%% file: effective-one-place-v6-arxiv.tex
\documentclass[11pt,a4paper]{article}
\usepackage{amsthm, amssymb, amsmath, a4wide}
\usepackage[all]{xy} 
\usepackage{paralist} 
\usepackage{graphicx} 
\usepackage[sans]{dsfont} 
\usepackage{varioref} 
\usepackage{subfigure}
\usepackage{tikz}
\usepackage{pgf}
\usetikzlibrary{decorations.pathreplacing,calc}
\usepackage{xifthen}
\usepackage{hyperref}

\title{An effective criterion for algebraic contractibility of rational curves}
\author{Pinaki Mondal\\ Weizmann Institute of Science\\ pinaki@math.toronto.edu}

\DeclareMathOperator\lt{LT}
\input{commands-2012}

\newcommand{\xdelta}{\bar X^\delta}
\newcommand{\hot}{\text{h.o.t.}}
\newcommand{\lot}{\text{l.o.t.}}

\newcommand{\ld}{\mathfrak{L}}

\newcommand{\dpsx}[2]{{#1} \langle \langle #2 \rangle \rangle }
\newcommand{\dpsxc}{\dpsx{\cc}{x}}

\def\noqed{\renewcommand{\qedsymbol}{}}
\newcommand{\quotequal}{\text{`='}}

\theoremstyle{definition} 
\newtheorem{assumption}[thm]{Assumption}

\begin{document}
\maketitle

\begin{abstract} 
Let $\pi:\tilde Y \to \cc\pp^2$ be a birational morphism of non-singular (rational) surfaces. We give an effective (necessary and sufficient) criterion for algebraicity of the surfaces resulting from contraction of the union of the strict transform of a line on $\cc\pp^2$ and all but one of the exceptional divisors of $\pi$. As a by-product we construct normal non-algebraic Moishezon surfaces with the `simplest possible' singularities, which in particular completes the answer to a remark of Grauert. Our criterion involves {\em global variants} of {\em key polynomials} introduced by MacLane. The geometric formulation of the criterion yields a correspondence between normal algebraic compactifications of $\cc^2$ with one irreducible curve at infinity and algebraic curves in $\cc^2$ with one {\em place} at infinity.
\end{abstract}

\section{Introduction} \label{sec-intro}

A (possibly reducible) curve $\tilde E$ on a non-singular algebraic surface $\tilde Y$ (defined over $\cc$) is called (analytically) {\em contractible} iff there is an analytic map $\tilde \pi: \tilde Y \to Y$ such that $\tilde \pi(\tilde E)$ is a collection of points and $\tilde \pi$ induces an isomorphism between $\tilde Y\setminus \tilde E$ and $Y \setminus \tilde \pi(\tilde E)$. A result of Grauert \cite{grauert} gives a complete characterization of curves which are contractible, namely: $\tilde E$ is contractible iff the matrix of intersection numbers of the irreducible components of $\tilde E$ is negative definite. In this article we consider (a special case of) the subsequent question: ``When is $\tilde E$ {\em algebraically contractible} (i.e.\ the surface $Y$ is also algebraic)?'' This question has been extensively studied, see e.g.\ \cite{artractability}, \cite{morrow-rossi}, \cite{brenton-algebraicity}, \cite{franco-lascu}, \cite{schroe-traction}, \cite{badescu-contractibility}, \cite{palka-Q1}. The ``strongest available numerical criterion for (algebraic) contractibility'' has been given by Artin \cite{artractability} and it states that $Y$ is algebraic (in fact, projective) if the singularities of $Y$ at the points in $\tilde \pi(\tilde E)$ are {\em rational}. However, it is well known, and was observed in \cite{artractability} that ``in general there are no numerical criteria equivalent with (algebraic) contractibility of a given curve.'' In this article we consider the {\em simplest} set up (see Remark \ref{simplest-remark}) for which algebraic contractibility does not readily follow from existing results, and give an {\em effective} (non-numerical) criterion which is {\em equivalent} to algebraic contractibility. \\

The only previously known {\em effective} criterion for algebraicity is the criterion of \cite{artractability}. It has been used extensively in the literature and a part of its usefulness derives from the fact that it is explicitly computable. However, it is only a {\em sufficient} condition and in general is {\em not} equivalent to algebraic contractibility. \cite[Theorem 3.4]{schroe-traction} gives a condition which is applicable in general and is {\em equivalent} to algebraic contractibility, but it is {\em not} effective. A similar statement is true for \cite[Corollary 2.6]{palka-Q1}. To the best of our knowledge our criterion is the first which is {\em both} effective and equivalent (in the situations where it applies) to algebraic contractibility. As a by product we get a new class of non-algebraic normal {\em Moishezon surfaces} (i.e.\ analytic surfaces for which the fields of meromorphic functions have transcendence degree 2 over $\cc$), including some which have possibly the `simplest possible' singularities (see Remark \ref{simplest-sing}). The first examples of non-algebraic normal Moishezon surfaces were constructed by Grauert in \cite[Section 4.8, Example d]{grauert} by contracting curves with genus $\geq 2$ and he remarked there that it was unknown to him if it is possible to construct non-algebraic surfaces by contracting even a torus. While an example of Nagata \cite[Example 3.1]{badescu} shows that it is indeed possible with a torus, our effective criterion gives (as far as we know) the first construction of non-algebraic surfaces by contracting (trees of) {\em rational} curves.\\

Finally, our effective criterion (Theorem \ref{key-thm-effective}) is stated in terms of {\em key polynomials} introduced by  MacLane \cite{maclane-key} (the `effectiveness' of the criterion stems from the fact that the construction of key polynomials are effective - see Remark \ref{essential-remark}). The key polynomials were introduced (and are used) to study valuations in a {\em local} setting. However, our criterion shows how they retain information about the {\em global geometry} when computed in `global coordinates'. The geometric analogue of our criterion establishes a new correspondence between (normal) algebraic compactifications of $\cc^2$ with one (irreducible) curve at infinity and algebraic curves in $\cc^2$ with one place at infinity. 

\subsection{The question} \label{sec-question}
Let $\pi: \tilde Y \to \pp^2$ be a birational morphism of non-singular surfaces (defined over $\cc$), $L \subseteq \pp^2$ be a line, and $E$ be the {\em exceptional divisor} of $\pi$ (i.e.\ $E$ is the union of curves in $\tilde Y$ which map to points in $\pp^2$). Let $E^*$ be an irreducible component of $E$ and $\tilde E$ be the union of the strict transform $\tilde L$ (on $\tilde Y$) of $L$ and all components of $E$ {\em excluding} $E^*$.

\begin{question} \label{contract-question}
Assume $\tilde E$ is (analytically) contractible. When is $\tilde E$ {\em algebraically} contractible?
\end{question}

\begin{rem} \label{simplest-remark}
In the set up of Question \ref{contract-question}, the answer is {\em always} affirmative if we replace $\tilde E$ by $\tilde E \setminus \tilde L$. More precisely, if $E'$ is any collection of curves which is contained in $E$, then $E'$ is algebraically contractible. Indeed, since $E$ is contractible by construction, it follows that $E'$ is contractible by Grauert's criterion \cite{grauert} (mentioned in the beginning of the introduction); let $\pi' : \tilde Y \to Y'$ be the contraction of $E'$. Then the birational morphism $Y' \to \pp^2$ induced by $\pi$ extends to a regular morphism. It follows that the singularities of $Y'$ are {\em sandwiched} (since $Y'$ is `sandwiched' between non-singular surfaces $\tilde Y$ and $\pp^2$). Since sandwiched singularities are rational \cite[Proposition 1.2]{lipman}, Artin's criterion \cite{artractability} implies that $Y'$ is projective.
\end{rem}

We already noted that in general there are no {\em numerical} criteria (i.e.\ criteria determined by numerical invariants of the singularities) which are equivalent to algebraic contractibility. The following example shows that this is already true in the set up of Question \ref{contract-question}. 

\begin{example}\label{non-example}
Let $(u,v)$ be a system of `affine' coordinates near a point $O \in \pp^2$ (`affine' means that both $u=0$ and $v=0$ are lines on $\pp^2$) and $L$ be the line $\{u=0\}$. Let $C_1$ and $C_2$ be curve-germs at $O$ defined respectively by $f_1 := v^5 - u^3$ and $f_2 := (v-u^2)^5 - u^3$. For each $i$, Let $\tilde Y_i$ be the surface constructed by resolving the singularity of $C_i$ at $O$ and then blowing up $8$ more times the point of intersection of the (successive) strict transform of $C_i$ with the exceptional divisor. Let $E^*_i$ be the {\em last} exceptional curve, and $\tilde E^{(i)}$ be the union of the strict transform $\tilde L_i$ (on $\tilde Y_i$) of $L$ and (the strict transforms of) all exceptional curves except $E^*_i$. \\

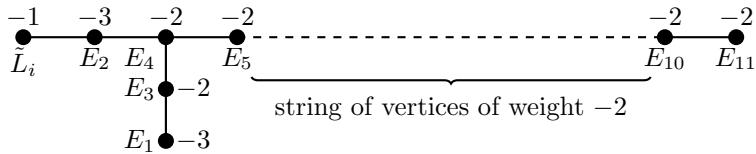
\begin{figure}[htp]
\begin{center}

\begin{tikzpicture}[scale=1, font = \small] 

	\pgfmathsetmacro\factor{1.25}
 	\pgfmathsetmacro\dashedge{4.5*\factor}	
 	\pgfmathsetmacro\edge{.75*\factor}
 	\pgfmathsetmacro\vedge{.55*\factor}

 	\draw[thick] (-2*\edge,0) -- (\edge,0);
 	\draw[thick] (0,0) -- (0,-2*\vedge);
 	\draw[thick, dashed] (\edge,0) -- (\edge + \dashedge,0);
 	\draw[thick] (\edge + \dashedge,0) -- (2*\edge + \dashedge,0);
 	
 	\fill[black] ( - 2*\edge, 0) circle (3pt);
 	\fill[black] (-\edge, 0) circle (3pt);
 	\fill[black] (0, 0) circle (3pt);
 	\fill[black] (0, -\vedge) circle (3pt);
 	\fill[black] (0, - 2*\vedge) circle (3pt);
 	\fill[black] (\edge, 0) circle (3pt);
 	\fill[black] (\edge + \dashedge, 0) circle (3pt);
 	\fill[black] (2*\edge + \dashedge, 0) circle (3pt);
 	
 	\draw (-2*\edge,0)  node (e0up) [above] {$-1$};
 	\draw (-2*\edge,0 )  node (e0down) [below] {$\tilde L_i$};
 	\draw (-\edge,0 )  node (e2up) [above] {$-3$};
 	\draw (-\edge,0 )  node [below] {$E_2$};
 	\draw (0,0 )  node (e4up) [above] {$-2$};	
 	\draw (0,0 )  node [below left] {$E_4$};
 	\draw (0,-\vedge )  node (down1) [right] {$-2$};
 	\draw (0,-\vedge )  node [left] {$E_3$};
 	\draw (0, -2*\vedge)  node (down2) [right] {$-3$};
 	\draw (0,-2*\vedge )  node [left] {$E_1$};
 	\draw (\edge,0)  node (e+1-up) [above] {$-2$};
 	\draw (\edge,0)  node [below] {$E_5$};
 	\draw (\edge + \dashedge,0)  node (e-last-1-up) [above] {$-2$};
 	\draw (\edge + \dashedge,0)  node [below] {$E_{10}$};
 	\draw (2*\edge + \dashedge,0)  node (e-last-up) [above] {$-2$};
 	\draw (2*\edge + \dashedge,0)  node [below] {$E_{11}$};
 	
 	\pgfmathsetmacro\factorr{.2}
 	\draw [thick, decoration={brace, mirror, raise=15pt},decorate] (\edge + \edge*\factorr,0) -- (\edge + \dashedge - \edge*\factorr,0);
 	\draw (\edge + 0.5*\dashedge,-\edge) node [text width= 5cm, align = center] (extranodes) {string of vertices of weight $-2$};
 			 	
\end{tikzpicture}
\caption{Weighted dual graph of $\tilde E^{(i)}$}\label{fig:non-example}
\end{center}
\end{figure}

Note that the the pair of germs $(C_1,L)$ and $(C_2, L)$ are {\em isomorphic} via the map $(u,v) \mapsto (u,v+u^2)$. It follows that the $\tilde E^{(i)}$'s have {\em identical} `weighted dual graphs' (the {\em dual graph} of a curve is a graph such that the vertices correspond to the irreducible components of the curve and two vertices are connected by an edge iff corresponding curves intersect; in the {\em weighted dual graph} the weight of every vertex is the self-intersection number of the corresponding curve); in particular, both $\tilde E^{(i)}$'s are analytically contractible. Figure \ref{fig:non-example} depicts the weighted dual graph (we labelled the vertices according to the order of appearance of the corresponding curves in the sequence of blow-ups). Since $\tilde E^{(i)}$ is connected, contraction of $\tilde E^{(i)}$ produces an analytic surface $Y_i$ with a unique singular point $P_i$. It can be computed that each $P_i$ has multiplicity $2$, geometric genus $1$ and the singularity at $P_i$ is {\em almost rational} in the sense of \cite{nemethi}. However, it turns out that $Y_1$ is algebraic, but $Y_2$ is {\em not} (see Example \ref{non-example-again}).
\end{example}

\begin{rem} \label{simplest-sing}
It can be shown that the singularities at $P_i$ (of Example \ref{non-example}) are in fact {\em hypersurface singularities} which are {\em Gorenstein} and {\em minimally elliptic} (in the sense of \cite{laufer-elliptic}). Minimally elliptic Gorenstein singularities have been extensively studied, and in a sense they form the simplest class of non-rational singularities. Since having only rational singularities imply algebraicity of the surface \cite{artractability}, it follows that the non-algebraic surface $Y_2$ of Example \ref{non-example} is a normal non-algebraic Moishezon surface with the `simplest possible' singularity. 
\end{rem} 

\begin{rem}
Remark \ref{simplest-remark} implies that in the set up of Question \ref{contract-question} deleting the vertex corresponding to $\tilde L$ from the weighted dual graph of $\tilde E$ turns it into the weighted dual graph of the resolution of a {\em sandwiched} singularity. In this sense the singularities resulting from contraction of $\tilde E$ are {\em almost sandwiched}.
\end{rem}

\subsection{The answers}

\subsubsection{Geometric ({\em non-effective}) answer}
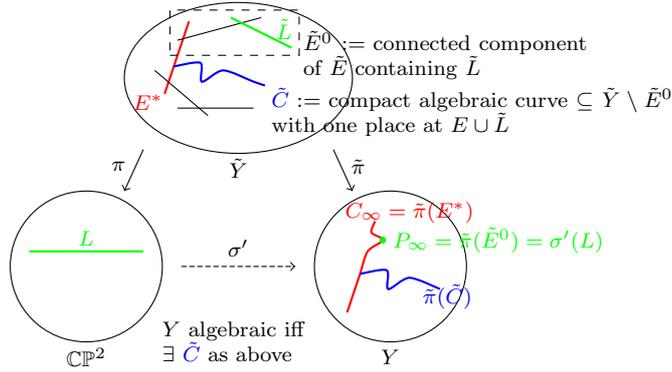
\begin{figure}[htp]
		\begin{center}
			\begin{tikzpicture}[font = \scriptsize]
			\pgfmathsetmacro\scalefactor{1}
			\pgfmathsetmacro\centerdist{2*\scalefactor}

			\pgfmathsetmacro\distop{2.5*\scalefactor}
			\pgfmathsetmacro\arrowbelowdist{0.75*\scalefactor}

			\pgfmathsetmacro\arrowdiagtopydist{1.55*\scalefactor}
			\pgfmathsetmacro\arrowdiagtopxdist{1.25*\scalefactor}

			\pgfmathsetmacro\arrowdiagbottomydist{1*\scalefactor}
			\pgfmathsetmacro\arrowdiagbottomxdist{1.5*\scalefactor}

			\begin{scope}[shift={(\centerdist,0)}, scale=\scalefactor]

				\draw[] (0,0) circle (1cm);
				
				\draw [red, thick] plot [smooth] coordinates {(-0.2,0.6) (-0.25,0.45) (-0.1,0.35) (-0.3,0.2) (-0.4,-0.1) (-0.45, -0.25) (-0.57,-0.61)};
				
				\fill [green] (-0.1,0.35) circle (1.25pt);
				\draw [green] (-0.1,0.35) node[right] {$P_\infty = \tilde \pi(\tilde E^0) = \sigma'(L)$};
				
				\begin{scope}[shift={(0.45,-0.25)}]
					\draw [blue,thick] plot [smooth] coordinates {(-0.85, 0.15) (-0.5,0.2) (-0.48,-0.05)  (-0.2,0.1) (0.2, -0.05)};
					\draw (0.3,-0.15) node [blue]{$\tilde \pi(\tilde C)$};
				\end{scope}

				\draw [red] (0.25,0.72) node {$C_\infty = \tilde \pi(E^*)$};
				\draw (0,-1.2) node {$Y$};
			\end{scope}

			\begin{scope}[shift={(-\centerdist,0)}, scale=\scalefactor]

				\draw[] (0,0) circle (1cm);
				\draw [green,thick] (-0.75, 0.2) -- (0.75, 0.2);


				\draw [green] (0,0.375) node {$L$};
				\draw (0, -1.2) node {$\cc\pp^2$};
			\end{scope}

			\begin{scope}[shift={(0, \distop)}, scale=\scalefactor]

				\draw[] (0,0) ellipse (1.5cm and 1cm);
				\begin{scope}[shift= {(-0.4, 0.4)}] 
					\draw [red, thick] plot [smooth] coordinates {(-0.25,0.35) (-0.4,-0.1) (-0.45, -0.25) (-0.57,-0.61)};
					\draw (-0.79,-0.72) node[red] {$E^*$};
				\end{scope}

				\draw [] (-0.8, 0.5) -- (0.3,0.8);
				\draw [green, thick] (-0.1, 0.8) -- (0.7, 0.4);
				\draw (0.6,0.65) node [green]{$\tilde L$};
				\draw [dashed] (-0.9,0.3) rectangle (0.8, 0.9);
				\draw (0.72,0.3) node [right, text width = 4cm] {$\tilde E^0:=$ connected component of $\tilde E$ containing $\tilde L$}; 
				\draw [blue,thick] plot [smooth] coordinates {(-0.85, 0.15) (-0.5,0.2) (-0.48,-0.05)  (-0.2,0.1) (0.2, -0.05) (0.35,-0.1)};
				\draw (0.3,-0.45) node[right, text width = 5.5cm]{\textcolor{blue}{$\tilde C$} $:=$ compact algebraic curve $\subseteq \tilde Y \setminus \tilde E^0$ with one place at $E \cup \tilde L$};

				\draw [] (-1.1, 0.1) -- (-0.4,-0.5);
				\draw [] (-0.8, -0.4) -- (0.2, -0.4);
				\draw (0, -1.2) node {$\tilde Y$}; 
			\end{scope}

			\draw[->] (-\arrowdiagtopxdist,\arrowdiagtopydist) -- (-\arrowdiagbottomxdist, \arrowdiagbottomydist) node [left, pos=0.4] {$\pi$};

			\draw[->] (\arrowdiagtopxdist,\arrowdiagtopydist) -- (\arrowdiagbottomxdist, \arrowdiagbottomydist) node [right, pos=0.4] {$\tilde \pi$};
			\draw[dash pattern=on 2pt off 1pt,->] (-\arrowbelowdist, 0) -- (\arrowbelowdist,0) node [above, pos=0.5]{$\sigma'$};

			\draw (0, -1) node[text width = 2cm]{$Y$ algebraic iff $\exists$ \textcolor{blue}{$\tilde C$} as above};
			
			\end{tikzpicture}
			\caption{Geometric answer to Question \ref{contract-question}} \label{fig:geom-answer}
		\end{center}
\end{figure} 

\begin{thm}[Geometric answer to Question \ref{contract-question}] \label{key-thm-geometric}
In the set up of Question \ref{contract-question}, the following are equivalent:
\begin{enumerate}
\item \label{algebraic-assertion} $\tilde E$ is algebraically contractible.
\item \label{existence-assertion-1} there is a (compact) algebraic curve $\tilde C$ on $\tilde Y$ such that $\tilde C$ does not intersect $\tilde E$.
\item \label{existence-assertion-2} there is a (compact) algebraic curve $\tilde C$ on $\tilde Y$ such that $\tilde C$ does not intersect $\tilde E^0$, where $\tilde E^0$ is the connected component of $\tilde E$ that contains $\tilde L$.
\item \label{existence-assertion-one-place} there is a (compact) algebraic curve $\tilde C$ on $\tilde Y$ such that $\tilde C$ does not intersect $\tilde E^0$ and $\tilde C$ has {\em only one place at} $E \cup \tilde L$ (i.e.\ $\tilde C$ intersects $E \cup \tilde L$ only at one point $P$ and $\tilde C$ is analytically irreducible at $P$).
\end{enumerate}
\end{thm}

\begin{rem}
The equivalence of Assertions \ref{algebraic-assertion}, \ref{existence-assertion-1} and \ref{existence-assertion-2} of Theorem \ref{key-thm-geometric} is not hard to see (see the proof of Theorem \ref{key-thm-geometric} in Section \ref{sec-proof}). The hardest part is to establish the implication \eqref{algebraic-assertion} $\im$ \eqref{existence-assertion-one-place}, and it is essentially equivalent to our `effective criterion' described below.
\end{rem}

\subsubsection{Effective answer}
Our {\em effective criterion} is given in terms of {\em key polynomials} of valuations introduced in \cite{maclane-key}. More precisely, in the set up of Question \ref{contract-question}, choose an affine system of coordinates $(u,v)$ near $O := \pi(E^*)$ such that $L = \{u=0\}$. Then $E^*$ induces a {\em discrete valuation} $\nu$ on $\cc(u,v)$: the value of $\nu$ on a rational function $g \in \cc(u,v)$ is the order of vanishing of $g\circ\pi$ along $E^*$. Then $\nu$ can be described by a finite sequence of {\em key polynomials} $\tilde g_0, \ldots, \tilde g_{n+1} \in \cc[u,v]$, $n \geq 0$, starting with $\tilde g_0 := u$ and $\tilde g_1 := v$ (see Definition \ref{key-pol-defn}). Let $(x,y) := (1/u,v/u)$, i.e.\ $(x,y)$ is a system of coordinates for $\pp^2\setminus L \cong \cc^2$. Define $g_0 := x$ and for each $k$, $1 \leq k \leq n+1$, let 
\begin{align}
g_k(x,y) := x^{\deg_{v}(\tilde g_k)}\tilde g_k(1/x,y/x) \in \cc[x,x^{-1},y] \label{form-from-pol-0}
\end{align}
We call $g_k$'s the {\em key forms} of the {\em semidegree} $\delta := -\nu$\footnote{Valuations $\nu$ on $\cc(x,y)$ which are {\em centered at infinity} with respect to $\cc^2$ constitute a central topic of this article. While dealing with such $\nu$, we always work with $\delta := -\nu$ (which we call a {\em semidegree}) instead of $\nu$, since for polynomials $f \in \cc[x,y]$, $\delta(f)$ has a more `natural' meaning than $\nu(f)$, in the same sense that degree is a `more natural' function on polynomials than negative degree. More generally, semidegrees are special types of {\em degree-like functions} \cite{sub1} which correspond to compactifications of affine varieties.}. 

\begin{thm}[Effective answer to Question \ref{contract-question}] \label{key-thm-effective}

In the set up of Question \ref{contract-question}, the following are equivalent:
\begin{enumerate}
\item \label{algebraic-assertion-again} $\tilde E$ is algebraically contractible.
\item \label{key-pol-assertion-1} $g_k \in \cc[x,y]$ for all $k$, $0 \leq k \leq n+1$.
\item \label{key-pol-assertion-2} $g_{n+1} \in \cc[x,y]$.
\end{enumerate}
Moreover, if $\tilde E$ is algebraically contractible, then $\tilde C := V(g_{n+1}) \subseteq \tilde Y$ satisfies the property of Assertion \ref{existence-assertion-one-place} of Theorem \ref{key-thm-geometric}.
\end{thm}

\begin{rem} \label{essential-remark}
The `effectiveness' of Theorem \ref{key-thm-effective} stems from the fact that the key forms can be explicitly computed in terms of the input data of Question \ref{contract-question}. More precisely, in Section \ref{subsec-encoding} we reformulate Question \ref{contract-question} in terms of a curve-germ $C$ at $O$ and an integer $r$ (Question \ref{contract-2-question}), and we make the following 
\begin{assumption} \label{essential-assumption}
The {\em Puiseux expansion} $v = \phi(u)$ of $C$ at $O$ is known, where $(u,v)$ and $O$ are as in the paragraph preceding Theorem \ref{key-thm-effective}. 
\end{assumption}
An algorithm to compute the key forms from $\phi(u)$ is given in Remarks \ref{essential-example-remark} and \ref{essentially-all}. Note that
\begin{enumerate}
\item We need to know only finitely many terms of $\phi(u)$ to compute the key forms. More precisely, if $p$ is the smallest integer such that $\phi(u) \in \cc[[u^{1/p}]]$ and $q/p$ is the {\em last characteristic exponent} (Definition \ref{puiseuxnition}) of $\phi$, then it suffices to know the first $(q+r-1)$ terms of $\phi$ (see Remark \ref{essential-example-remark} and Algorithm \ref{essential-algorithm}). 
\item It is possible to compute key forms directly from the equation of $C$ in $(u,v)$-coordinates via a modification of Abhyankar's algorithm to determine irreducibility of power series \cite{abhyankar-irreducibility}. 
\end{enumerate}
\end{rem}

\begin{rem} \label{no-claim-remark}
Our algorithm to construct key forms (Algorithm \ref{essential-algorithm}) follows from the standard theory of key polynomials and Puiseux series corresponding to valuations (e.g.\ the results from \cite[Chapters 2 and 4]{favsson-tree}). Therefore we omit the proof (based on a long induction) of the correctness of Algorithm \ref{essential-algorithm} and its immediate corollaries contained in Proposition \ref{essential-prop}. However, we give a detailed example (Example \ref{essential-example}) which (we hope!) makes it clear how the algorithm works. 
\end{rem}

\subsubsection{Answer in terms of {\em valuative tree}}
The final version of our answer to Question \ref{contract-question} is an (immediate) translation of Theorem \ref{key-thm-geometric} in the terminology of the {\em valuative tree} introduced in \cite{favsson-tree}. Consider the set up of Question \ref{contract-question}. Let $X := \pp^2 \setminus L \cong \cc^2$ and, as in Theorem \ref{key-thm-effective}, let $(x,y)$ be a system of coordinates on $X$ and $\nu$ be the valuation on $\cc(x,y)$ corresponding to $E^*$. Let $\scrV_0$ be the space of all valuations $\mu$ on $\cc[x,y]$ such that $\min\{\mu(x),\mu(y)\} = -1$. It turns out that $\scrV_0$ has the structure of a tree with root at $-\deg_{(x,y)}$, where $\deg_{(x,y)}$ is the usual degree in $(x,y)$-coordinates; $\scrV_0$ is called the {\em valuative tree at infinity} on $X$  \cite[Section 7.1]{favsson-eigen}. Let $\tilde \nu := \nu/\max\{-\nu(x),-\nu(y)\}$ be the `normalized' image of $\nu$ in $\scrV_0$. 

\begin{cor} \label{valuative-global}
$\tilde E$ is algebraically contractible iff there is a tangent vector $\tau$ of $\tilde \nu$ on $\scrV_0$ such that 
\begin{compactenum}
\item $\tau$ is {\em not} represented by $-\deg$, and
\item $\tau$ is represented by a curve valuation corresponding to an algebraic curve with one place at infinity.
\end{compactenum}
\end{cor}

\subsection{Weighted dual graphs corresponding to only algebraic, only non-algebraic, or both types of contractions - the semigroup conditions}

Example \ref{non-example} shows that in general algebraic contractibility can not be determined only from the weighted dual graph (of the curve to be contracted). However, it is possible to completely characterize (in the set up of Question \ref{contract-question}) the weighted dual graphs which correspond to {\em only algebraic} contractions, those which correspond to {\em only non-algebraic} contractions, and those which correspond to {\em both} types of contractions (Theorem \ref{semigroup-prop}). The characterization is given in terms of two sets of {\em semigroup conditions} \eqref{semigroup-criterion-1} and \eqref{semigroup-criterion-2}. The `first' set of semigroup conditions \eqref{semigroup-criterion-1} are equivalent to the semigroup conditions that appear in the theory of plane curves with one place at infinity developed in \cite{abhya-moh-tschirnhausen}, \cite{abhyankar-expansion}, \cite{abhyankar-semigroup}, \cite{sathaye-stenerson}. We now elaborate on this connection: \\

As noted in Remark \ref{essential-remark}, we encode (in Subsection \ref{subsec-encoding}) the input data for Question \ref{contract-question} in terms of a curve-germ $C$ at a point $O \in L$ and a positive integer $r$. To such a pair $(C,r)$, we associate a sequence of {\em virtual poles} (Definition \ref{virtual-defn}) and show that the dual graph of $\tilde E$ is the dual graph for a ({\em possibly different}) $\tilde E'$ which is algebraically contractible iff the virtual poles satisfy the semigroup conditions \eqref{semigroup-criterion-1}. On the other hand, it follows from the theory of plane curves with one place at infinity that the {\em same} semigroup conditions are equivalent to the existence of a plane algebraic curve $C'$ with one place at infinity with `almost' the same singularity type at infinity as the singularity type of $C$ at $O$. Moreover, if the curve $C'$ exists, then the virtual poles of $(C,r)$ are (up to a constant factor) precisely the generators of the {\em semigroup of poles} at the point at infinity of $C'$ - i.e.\ in this case {\em virtual poles are real!} (See Remark \ref{curve-semigroup-remark}.)

\subsection{Further applications and comments on the structure of the proof}
Identifying $\cc^2$ with $\pp^2 \setminus L$ (as in the set up of Theorem \ref{key-thm-effective}), turns Question \ref{contract-question} into a question about algebraicity of normal (analytic) compactifications of $\cc^2$ with one irreducible curve at infinity, which we refer to as {\em primitive compactifications} of $\cc^2$. Lemma \ref{poly-positive-prop} implies that every primitive algebraic compactification of $\cc^2$ is in fact a {\em weighted complete intersection}. Moreover, it follows from Lemma \ref{semi-isomorphism} that the curve at infinity (on a primitive algebraic compactification) has at most one singular point, and the singularity is at most a {\em toric} (non-normal) singularity. The follow up of this article in \cite{sub2-2} is directed towards a description of the group of isomorphisms and moduli spaces of primitive compactifications and computation of their singularity invariants. \\

Finally, some remarks on our proofs of the results in this article. These are split into three steps: in Section \ref{sec-proof} we reduce the results quoted in the Introduction, namely Theorems \ref{key-thm-geometric}, \ref{key-thm-effective} and \ref{semigroup-prop}, to Proposition \ref{key-prop}. We prove Proposition \ref{key-prop} in Section \ref{sec-key-proof} based on some lemmas, whose proofs we defer to Section \ref{sec-technical-proofs}. Our proofs are self-contained modulo some properties of key polynomials which we list in Section \ref{sec-defn} and are motivated by the general theory of projective completions via `degree-like functions' \cite{sub1}. \\

The technical difficulty of the proofs can be attributed to the fact that it is possible for $x$ to appear with a {\em negative} exponent in a monomial in some of the key forms $g_k$ of Theorem \ref{key-thm-effective} (unlike the classical case of `key polynomials', which really are polynomials). The latter is in fact the crucial new element that makes the `calculus' for our effective criterion possible (i.e.\ it is essentially a blessing in disguise!). The primary `content' of the results in this article is concentrated in Assertion \ref{hard-assertion} of Proposition \ref{key-prop}, which roughly states that if the exponent of $x$ in a monomial of some $g_k$ is negative, then the same is true (i.e.\ $x$ appears with negative exponent in some monomial) for every $f \in \cc[x,x^{-1},y]$ such that {\em all} the branches at infinity of $f = 0$ have the same initial Puiseux expansion as that of the (unique) branch of $g_k  = 0$ at $O$ (where $O := \pi(E^*)$ is as in the paragraph preceding Theorem \ref{key-thm-effective}). The key idea of the proof (which is a natural consequence of the point of view of `degree-like functions') is to `lift' the cancellations of monomials (in which $x$ appears with negative exponents) in $\cc[x,y]$ to cancellations of monomials in the graded ring of the compactifications of $\cc^2$ corresponding to the {\em semidegrees} (Definition \ref{semi-defn}) $\delta_k$ of \eqref{delta-k} corresponding to $g_k$. \\

\subsection{Acknowledgements}
This project started during my Ph.D.\ under Professor Pierre Milman to answer some of his questions, and profited enormously from his valuable suggestions and speaking in his seminars. I would like to thank Professor Peter Russell for inviting me to present this result and for his helpful remarks. The exposition of Section \ref{sec-encoding} is motivated by his suggestions. The idea for presenting the `effective answer' in terms of key polynomials came from a remark of Tommaso de Fernex. I thank Mattias Jonsson and Mark Spivakovsky for bearing with my bugging at different stages of this write up, and finally, I have to thank Dmitry Kerner to force me to think in geometric terms by his questions. 

\section{Encoding Question \ref{contract-question} in terms of a curve-germ and an integer and statements of the semigroup conditions}  \label{sec-encoding}
In Section \ref{subsec-encoding} we encode the input data of Question \ref{contract-question} in terms of a curve-germ and an integer. We also state a simple version of our effective criterion which is applicable in the case that the singularity of the curve-germ has one Puiseux pair (Proposition \ref{effective-answer-Puiseux-1}). In Subsection \ref{subsec-semigroup} we state the {\em semigroup conditions} and characterize the dual graphs which correspond to only algebraic or only non-algebraic, or both types of contractions (Theorem \ref{semigroup-prop}).

\subsection{The encoding and a version of effective criterion in the case of a single Puiseux pair} \label{subsec-encoding}

We continue to use the notations of Section \ref{sec-question}. At first note that in the set up of Question \ref{contract-question} we may w.l.o.g.\ assume the following 
\begin{enumerate}
\item $\pi$ is a sequence of blow-ups such that every blow-up (other than the first one) is centered at a point on the exceptional divisor of the preceding blow-up. 
\item $E^*$ is the exceptional divisor of the {\em last} blow-up.
\end{enumerate}
Now assume the above conditions are satisfied. Let
\begin{description}
\item[] $C^* := $ an analytic curve germ at a generic point on $E^*$ which is transversal to $E^*$,
\item[] $C := \pi(C^*)$, 
\item[] $r := $ (number of total blow-ups in $\pi$) $-$ (the minimum number of blow-ups necessary to ensure that the strict transform of $C$ transversally intersects the union of the strict transform of $L$ and the exceptional divisor).
\end{description}
It is straightforward to see that $L$, $C$ and $r$ uniquely determine $\tilde Y$, $E^*$ and $\tilde E$ via the following construction:

\paragraph{Construction of $\tilde Y$, $E^*$ and $\tilde E$ from $(L,C,r)$:}
\mbox{}
\begin{description}
\item[] $\tilde Y :=$ the surface formed by at first constructing (via a sequence of blow-ups) the minimal resolution of the singularity of $C \cup L$ and then blowing up successively the point of intersection of the strict transform of $C$ and the exceptional divisor $r$ more times,
\item[] $E^* :=$ the `last' exceptional curve, i.e.\ the exceptional divisor of the last blow-up in the construction of $\tilde Y$,
\item[] $\tilde E :=$ the union of the strict transform $\tilde L$ (on $\tilde Y$) of $L$ with the strict transforms of all the exceptional curves (of the sequence of blow-ups in the construction of $\tilde Y$) except $E^*$. 
\end{description}

\begin{figure}[htp]
		\begin{center}
			\begin{tikzpicture}[font = \scriptsize]
			\pgfmathsetmacro\scalefactor{1.1}
			\pgfmathsetmacro\centerdist{3*\scalefactor}

			\pgfmathsetmacro\distop{3*\scalefactor}
			\pgfmathsetmacro\arrowhordist{1.85*\scalefactor}

			\pgfmathsetmacro\arrowverttopdist{1.85*\scalefactor}
			\pgfmathsetmacro\arrowvertbottomdist{1.15*\scalefactor}

			\begin{scope}[shift={(-\centerdist,0)}, scale=\scalefactor]

				\draw[] (0,0) circle (1cm);
				\draw [green,thick] (-0.75, 0.2) -- (0.75, 0.2);

				\fill [red] (-.4,0.2) circle (1.25pt);
				\draw [red] (-0.4,0.375) node {$O$};
				
				\draw [blue, thick] plot [smooth] coordinates {(-0.3,-0.2) (-.55, 0.125) (-.4, 0.2) };
				\draw [blue, thick] plot [smooth] coordinates {(-0.5,-0.2) (-.65, 0.1) (-.4, 0.2) };

				\draw (0.8,0) node {\textcolor{green}{$L$} $:=$ a line};
				\draw (-0.8,-0.4) node [right] {\textcolor{blue}{$C:=\!$} analytically};
				\draw (-0.8,-0.6) node [right] {\phantom{\textcolor{blue}{$C:=\!$}}  irr.\ curve germ};
				\draw (0, -1.2) node {$\cc\pp^2$};

			\end{scope}
			
			\begin{scope}[shift={(-\centerdist, \distop)}, scale=\scalefactor]

				\draw[] (0,0) circle (1cm);
				\draw [green, thick] plot [smooth] coordinates {(-0.4,0.7) (-0.8,-0.4)};
				\draw (-0.5,0) node [green]{$L'$};
				
				\draw [] (-0.7, 0.5) -- (0.3,0.8);
				\draw [blue,thick] plot [smooth] coordinates {(0, 0.2) (-0.1,0.4) (0,0.5) (-0.2, 0.75)};
				
				\draw (0.1,0.1) node [blue]{$C'$};

				\draw [] (-0.9, 0.1) -- (-0.4,-0.5);
				
			\end{scope}
			
			\begin{scope}[shift={(\centerdist, \distop)}, scale=\scalefactor]

				\draw[] (0,0) circle (1cm);
				\draw [green, thick] plot [smooth] coordinates {(-0.4,0.7) (-0.8,-0.4)};
				\draw (-0.45,0) node [green]{$\tilde L$};
				\draw [] (-0.7, 0.5) -- (0.3,0.8);
				\draw [] (-0.1, 0.8) -- (0.7, 0.4);
				\draw [] (-0.9, 0.1) -- (-0.4,-0.5);
				\draw [red, thick] (0.5, 0.6) -- (0.7, -0.1);
				\draw (0.5,-0.4) node [right, text width = 3cm] {\textcolor{red}{$E^*$} $:=$ `last' exceptional curve}; 
				
				\draw [blue,thick] plot [smooth] coordinates {(0.8, 0.1) (0.5,0) (0.6,-0.1) (0.4, -0.2)};
				\draw (0.25,-0.2) node [blue]{$C^*$};

				\draw [dashed] (-1, -0.6) -- (-1, 0.9) -- (0.8, 0.9) -- (0.8, 0.3) -- (-0.3,0.3) -- (-0.3, -0.6) -- cycle;
				\draw (0.8,0.6) node [right, text width = 4cm] {$\tilde E :=$ union of $\tilde L$ and all exceptional curves except $E^*$}; 
			\end{scope}

			\begin{scope}[shift={(\centerdist, 0)}, scale=\scalefactor]

				\draw[] (0,0) circle (1cm);
				\draw [red, thick] plot [smooth] coordinates {(0.5, 0.6) (0.7, -0.1)};
				
				\draw [blue,thick] plot [smooth] coordinates {(0.8, 0.1) (0.5,0) (0.6,-0.1) (0.4, -0.2)};		

				\draw (0, -1.2) node {$Y$}; 
				\draw (0,-1.5) node [shift={(2,0)}] {\phantom{f $\cc^2$ with one irr.\ curve at $\infty$}};
				
			\end{scope}

			\draw[->] (\arrowhordist, \distop) -- (-\arrowhordist,\distop) node [above, pos=0.5, text width = 3cm]{blow up $C' \cap $ (excep.\ divisor) $r$ more times};
			\draw[->] (-\centerdist, \arrowverttopdist) -- (-\centerdist, \arrowvertbottomdist) node [left, pos=0.4, text width=2.5cm, shift={(0,0)}] {resolution of singularities of $C \cup L$};

			\draw[->] (\centerdist, \arrowverttopdist) -- (\centerdist, \arrowvertbottomdist) node [right, pos=0.4, text width=3cm, shift={(0.2,0)}] {contract $\tilde E$ analytically (if possible)};
			
			\draw (0.85,-1.5) node [text width = 4cm] {Question: Is $Y$ algebraic?}; 

			\end{tikzpicture}
			\caption{Formulation of Question \ref{contract-question} in terms of $(L,C,r)$}\label{fig:L-c-r}
		\end{center}
\end{figure}
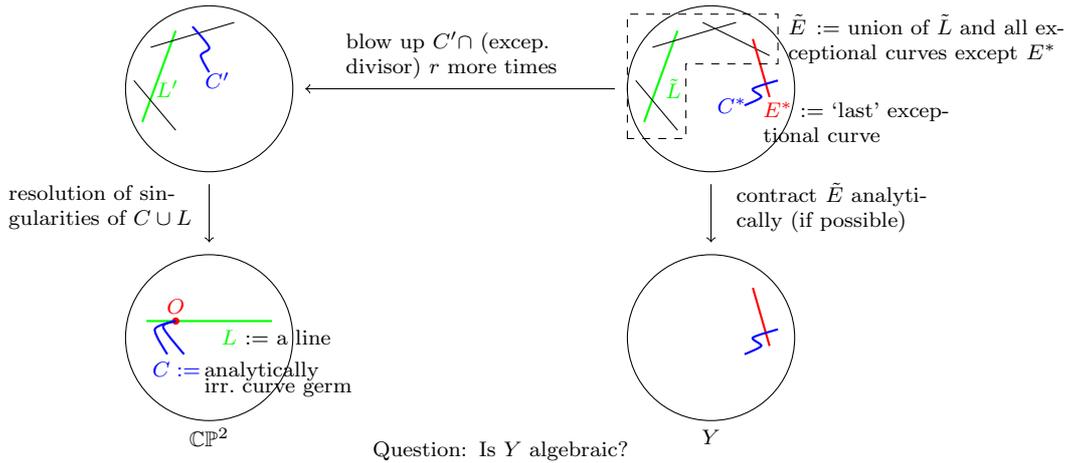

\begin{bold-question}[Reformulation of Question \ref{contract-question} - see Figure \ref{fig:L-c-r}]  \label{contract-2-question}
Let $L \subseteq \pp^2$ be a line, $C$ be an analytic curve-germ at a point $O \in L$ and $r$ be a non-negative integer. Let $\tilde Y_{L,C,r}$, $E^*_{L,C,r}$ and $\tilde E_{L,C,r}$ be the corresponding surface and divisors resulting from the above construction. Then we ask
\begin{enumerate}
\let\oldenumi\theenumi
\renewcommand{\theenumi}{\ref{contract-2-question}.\oldenumi}
\item \label{contract-2-question-1} When is $\tilde E_{L,C,r}$ contractible?
\item \label{contract-2-question-2} When is $\tilde E_{L,C,r}$ algebraically contractible?
\end{enumerate}
\end{bold-question}

In the set up of Question \ref{contract-2-question}, we will assume that the {\em Puiseux expansion} $v = \phi(u)$ of $C$ at $O$ is known, where $(u,v)$ is a system of affine coordinates near $O$ such that $L = \{u=0\}$ and $v=0$ is also a line on $\pp^2$ (Assumption \ref{essential-assumption}). We now give an answer of Question \ref{contract-2-question-1} in terms of the Puiseux series $\phi(u)$ of $C$ using a result of \cite{sub2-1} and then give a simple answer to Question \ref{contract-2-question-2} in the case that $\phi(u)$ has only a single {\em Puiseux pair} (see Definition \ref{puiseuxnition}). We start with a simple observation:

\begin{lemma}
If $\ord_u(\phi) \geq 1$, then $\tilde E_{L,C,r}$ is {\em not} contractible.
\end{lemma}

\begin{proof}
Indeed, $\ord_u(\phi) \geq 1$ implies that $C$ is {\em not} tangent to $L$, so that the strict transforms of $L$ and $C$ on the blow-up of $\pp^2$ at $O$ do not intersect. It follows that $\tilde L$ has self-intersection $\geq 0$, and consequently it is not contractible by Grauert's criterion. 
\end{proof}

From now on we assume that $\ord_u(\phi) < 1$. Let the Puiseux pairs of $\phi$ be $(\tilde q_1, p_1), \ldots, (\tilde q_{\tilde l}, p_{\tilde l})$ (note that $\ord_u(\phi)< 1$ implies that $\tilde l \geq 1$ and $\ord_u(\phi) =  \tilde q_1/p_1$). For every $\omega \in \rr$, let us write $[\phi]_{< \omega}$ for the (finite) Puiseux series obtained by summing up all terms of $\phi$ which have order $< \omega$. Define
\begin{align*}
\alpha_{L,C,r} 	&:= \text{intersection multiplicity at $O$ of $C$ and a curve-germ with Puiseux expansion}  \\
				& \mbox{\phantom{$:\mathrel{=}\ $} } v = [\phi(u)]_{< (\tilde q_{\tilde l}+r)/p}  + \xi^* u^{(\tilde q_{\tilde l}+r)/p} + \hot\ \text{for a generic}\ \xi^* \in \cc, \ \text{where} \label{local-alpha} \\
p 				&:= \text{polydromy order of $\phi$ (Definition \ref{puiseuxnition})} = p_1p_2 \cdots p_{\tilde l}. 
\end{align*}	
A straightforward computation using the definition of intersection product in terms of Puiseux series shows that
\begin{align}
\alpha_{L,C,r}	&= p \left( (p_1 \cdots p_{\tilde l} - p_2 \cdots p_{\tilde l}) \frac{\tilde q_1}{p_1} + (p_2 \cdots p_{\tilde l} - p_3 \cdots p_{\tilde l}) \frac{\tilde q_2}{p_1p_2} \right. \notag \\
				& \quad \qquad \left. + \cdots + (p_{\tilde l-1} p_{\tilde l} - p_{l}) \frac{\tilde q_{l-1}}{p_1 \cdots p_{l-1}}  + (p_{\tilde l} -1) \frac{\tilde q_{\tilde l}}{p_1 \cdots p_{\tilde l}} \right) + \tilde q_{\tilde l} + r,		
\end{align} 

The following proposition is an immediate corollary of \cite[Corollary 4.11 and Remark-Definition 4.13]{sub2-1}.

\begin{prop}[Answer to Question \ref{contract-2-question-1}] \label{existence-answer-local}
$\tilde E_{L,C,r}$ is contractible iff $\ord_u(\phi) < 1$ and $\alpha_{L,C,r} < p^2$. \qed
\end{prop}

The following is a reformulation of Theorem \ref{key-thm-effective} applicable in the case that $C$ has only one Puiseux pair, i.e.\ $\tilde l = 1$.

\begin{prop}[Answer to Question \ref{contract-2-question-2} when $\tilde l=1$] \label{effective-answer-Puiseux-1} 
Let $(L,C,r)$ be as in Question \ref{contract-2-question}. Assume that the Puiseux expansion $v = \phi(u)$ of $C$ at $O$ has only one Puiseux pair $(\tilde q,p)$. Let $\tilde \omega$ be the weighted order on $\cc(u,v)$ which gives weights $p$ to $u$ and $\tilde q$ to $v$. Let $f(u,v)$ be the (unique) Weirstrass polynomial in $v$ which defines $C$ near $O$. Define $\tilde f$ to be the sum of all monomial terms of $f$ which have $\tilde \omega$-value less than $\alpha_{L,C,r} = p\tilde q+r$. Then $\tilde E_{L,C,r}$ is algebraically contractible iff it is contractible and $\deg_{(u,v)}(\tilde f) \leq p$, where $\deg_{(u,v)}$ is the usual degree in $(u,v)$-coordinates.
\end{prop}

We prove Proposition \ref{effective-answer-Puiseux-1} in Section \ref{sec-proof}.

\begin{example}[Continuation of Example \ref{non-example} - see also Remark \ref{1-puiseux-remark}] \label{non-example-again}
Let $L$ and $C_1$ and $C_2$ be as in Example \ref{non-example}. We consider Question \ref{contract-2-question} for $C_1$ and $C_2$ and $r \geq 0$ (Example \ref{non-example} considered the case $r=8$). Figure \ref{fig:non-example-again} depicts the dual graph $\tilde E_{L,C_i,r}$; in particular $\tilde E_{L,C_i,r}$ is disconnected for $r = 0$.\\

\begin{figure}[htp]
\begin{center}

\subfigure[Case $r = 0$]{
\begin{tikzpicture}[scale=1, font = \small] 	
 	\pgfmathsetmacro\edge{.75}
 	\pgfmathsetmacro\vedge{.5}
	
 	\draw[thick] (-2*\edge,0) -- (-\edge,0);
 	\draw[thick] (0,-\vedge) -- (0,-2*\vedge);

 	\fill[black] ( - 2*\edge, 0) circle (3pt);
 	\fill[black] (-\edge, 0) circle (3pt);
 	\fill[black] (0, -\vedge) circle (3pt);
 	\fill[black] (0, - 2*\vedge) circle (3pt);
 	
 	\draw (-2*\edge,0)  node (e0up) [above] {$-1$};
 	\draw (-2*\edge,0 )  node (e0down) [below] {$\tilde L$};
 	\draw (-\edge,0 )  node (e1up) [above] {$-3$};	
 	\draw (0,-\vedge )  node (down1) [right] {$-2$};
 	\draw (0, -2*\vedge)  node (down2) [right] {$-3$};
 			 	
\end{tikzpicture}
}
\subfigure[Case $r \geq 1$]{
\begin{tikzpicture}[scale=1, font = \small] 	
 	\pgfmathsetmacro\dashedge{4.5}	
 	\pgfmathsetmacro\edge{.75}
 	\pgfmathsetmacro\vedge{.5}

 	\draw[thick] (-2*\edge,0) -- (\edge,0);
 	\draw[thick] (0,0) -- (0,-2*\vedge);
 	\draw[thick, dashed] (\edge,0) -- (\edge + \dashedge,0);
 	
 	\fill[black] ( - 2*\edge, 0) circle (3pt);
 	\fill[black] (-\edge, 0) circle (3pt);
 	\fill[black] (0, 0) circle (3pt);
 	\fill[black] (0, -\vedge) circle (3pt);
 	\fill[black] (0, - 2*\vedge) circle (3pt);
 	\fill[black] (\edge, 0) circle (3pt);
 	\fill[black] (\edge + \dashedge, 0) circle (3pt);
 	
 	\draw (-2*\edge,0)  node (e0up) [above] {$-1$};
 	\draw (-2*\edge,0 )  node (e0down) [below] {$\tilde L$};
 	\draw (-\edge,0 )  node (e1up) [above] {$-3$};
 	\draw (0,0 )  node (middleup) [above] {$-2$};	
 	\draw (0,-\vedge )  node (down1) [right] {$-2$};
 	\draw (0, -2*\vedge)  node (down2) [right] {$-3$};
 	\draw (\edge,0)  node (e+1-up) [above] {$-2$};
 	\draw (\edge + \dashedge,0)  node (e-last-1-up) [above] {$-2$};
 	
 	\draw [thick, decoration={brace, mirror, raise=5pt},decorate] (\edge,0) -- (\edge + \dashedge,0);
 	\draw (\edge + 0.5*\dashedge,-0.5) node [text width= 5cm, align = center] (extranodes) {$r-1$ vertices of weight $-2$};
 			 	
\end{tikzpicture}
}
\caption{Dual graph of $\tilde E_{L,C_i,r}$}\label{fig:non-example-again}
\end{center}
\end{figure}
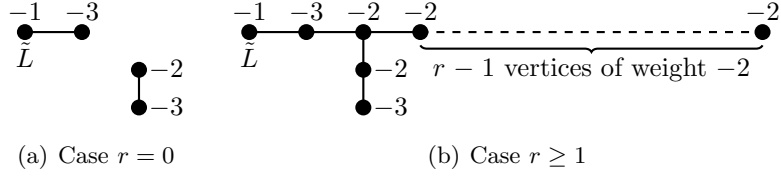

Recall that $C_i$'s are defined by $f_i = 0$, with $f_1 := v^5 - u^3$ and $f_2 := (v-u^2)^5 - u^3$. It follows that the Puiseux expansions in $u$ for each $C_i$ has only one Puiseux pair, namely $(3, 5)$. Moreover, each $f_i$ is a Weirstrass polynomial in $v$, so that we can use Propositions \ref{existence-answer-local} and \ref{effective-answer-Puiseux-1} to determine contractibility and algebraic contractibility of $\tilde E_{L,C_i,r}$.  \\

Identity \eqref{local-alpha} implies that $\alpha_{L,C_i,r} = r + 15$ for each $i = 1,2$, and therefore Proposition \ref{existence-answer-local} implies that $\tilde E_{L, C_i,r}$'s are contractible iff $r < 10$. We now determine if the contractions are algebraic. The weighted degree $\tilde \omega$ of Proposition \ref{effective-answer-Puiseux-1} is the same for both $i$'s, and it corresponds to weights $5$ for $u$ and $3$ for $v$. The $\tilde f$ of Proposition \ref{effective-answer-Puiseux-1} (computed from $f_i$'s) are as follows:
\begin{align*}
\tilde f_1 &= \begin{cases}
			0 & \text{if}\ r = 0,\\			
			v^5 - u^3 & \text{if}\ r\geq 1.
			\end{cases}
		&& &
\tilde f_2 &= \begin{cases}
			0 & \text{if}\ r = 0,\\			
			v^5 - u^3 & \text{if}\ 1 \leq r\leq 7, \\
			v^5 - u^3 - 5v^4u^2 & \text{if}\ 8 \leq r\leq 9.
			\end{cases}
\end{align*}			
Proposition \ref{effective-answer-Puiseux-1} therefore implies that $\tilde E_{L, C_1,r}$ is algebraically contractible for all $r < 10$, but $\tilde E_{L,C_2,r}$ is algebraically contractible only for $r \leq 7$. In particular, for $r = 8,9$, the contraction of $\tilde E_{L,C_2,r}$ produces a normal non-algebraic analytic surface. 
\end{example}

\subsection{The semigroup conditions on the sequence of {\em virtual poles}} \label{subsec-semigroup}

We continue to use the notations of the set-up of Subsection \ref{subsec-encoding}; in particular, we assume that the Puiseux expansion for $C$ is $v = \phi(u)$ with {\em Puiseux pairs} (Definition \ref{puiseuxnition}) $(\tilde q_1, p_1), \ldots, (\tilde q_{\tilde l},p_{\tilde l})$ with $\tilde l \geq 1$. Define $C_0 := L =  \{u = 0\}$, and for each $k$, $1 \leq k \leq \tilde l$, let $C_k$ be the curve-germ at $O$ with the Puiseux expansion $v = \phi_k(u)$, where $\phi_k(u)$ is the {\em Puiseux polynomial} (i.e.\ a Puiseux series with finitely many terms) consisting of all the terms of $\phi$ upto, but not including, the $k$-th characteristic exponent. Then it is a standard result (see e.g.\ \cite[Lemma 5.8.1]{casas-alvero}) that $\tilde \omega_k := (C, C_k)_O$, $0 \leq k \leq \tilde l$, are generators of the semigroup $\{(C, D)_O\}$ of intersection numbers at $O$, where $D$ varies among analytic curve-germs at $O$ not containing $C$. It follows from a straightforward computation that
\begin{subequations} \label{m-k-defn}
\begin{align}
\tilde \omega_0 &= p_1 \cdots p_{\tilde l},\quad \tilde \omega_1 = \tilde q_1p_2 \cdots p_{\tilde l},\ \text{and} \\
\tilde \omega_{k} &= p \left( (p_1 \cdots p_{k-1} - p_2 \cdots p_{k-1}) \frac{\tilde q_1}{p_1} + (p_2 \cdots p_{k-1} - p_3 \cdots p_{k-1}) \frac{\tilde q_2}{p_1p_2} \right. \notag  \\
		& \quad \qquad \left. + \cdots + (p_{k-1} - 1) \frac{\tilde q_{k-1}}{p_1 \cdots p_{k-1}}  + \frac{\tilde q_k}{p_1 \cdots p_k} \right),\ 2 \leq k \leq \tilde l. \label{m-k-general-defn}
\end{align}
\end{subequations} 

\begin{defn}[Virtual poles]\label{virtual-defn}
Let
\begin{align*}
l:= \begin{cases}
				\tilde l - 1 	& \text{if}\ r = 0, \\
				\tilde l		& \text{if}\ r > 0.
			\end{cases}
\end{align*}
The sequence of {\em virtual poles} at $O$ on $C$ are $\omega_0, \ldots, \omega_l$ defined as
\begin{subequations} \label{tilde-m-k-defn}
\begin{align}
\omega_0 &:= \tilde \omega_0,\quad 
\omega_1 := p_1 \cdots p_{\tilde l} - \tilde \omega_1,\quad
\omega_k := p_1^2 \cdots p_{k-1}^2 p_k \cdots p_{\tilde l} - \tilde \omega_k,\ 2 \leq k \leq l. \\
\intertext{The {\em generic virtual pole} at $O$ is}
\omega_{l+1} &:= \begin{cases}
							p_1^2 \cdots p_{\tilde l -1}^2 p_{\tilde l} - \tilde \omega_{\tilde l} = \frac{1}{p_{\tilde l}}\left(p^2 - \alpha_{L,C,r}\right)	& \text{if}\ r = 0, \\
							p_1^2 \cdots p_{\tilde l -1}^2 p_{\tilde l}^2 - p_{\tilde l}\tilde \omega_{\tilde l} - r = p^2 - \alpha_{L,C,r} 				& \text{if}\ r > 0.
						  \end{cases} 
\end{align}
\end{subequations}
\end{defn}

\begin{rem}
In Section \ref{essential-section} we define {\em essential key forms} corresponding to the {\em semidegree} associated to $E^*_{L,C,r}$. The $\omega_k$'s we defined here are precisely the orders of pole along $E^*_{L,C,r}$ of these essential key forms.
\end{rem}

Fix $k$, $1 \leq k \leq l$. The {\em semigroup conditions} for $k$ are:
\begin{gather}
 p_k \omega_k \in \zz_{\geq 0} \langle \omega_0, \ldots, \omega_{k-1}\rangle .  \tag{S1-k} \label{semigroup-criterion-1} \\
(\omega_{k+1}, p_k \omega_k) \cap \zz\langle \omega_0, \ldots, \omega_k \rangle = (\omega_{k+1}, p_k \omega_k) \cap \zz_{\geq 0}\langle \omega_0, \ldots, \omega_k \rangle,\ \tag{S2-k} \label{semigroup-criterion-2}
\end{gather}
where $(\omega_{k+1}, p_k \omega_k): = \{a \in \rr: \omega_{k+1} < a < p_k \omega_k\}$ and $\zz_{\geq 0}\langle \omega_0, \ldots, \omega_k \rangle$ (respectively, $\zz \langle \omega_0, \ldots, \omega_k \rangle$) denotes the semigroup (respectively, group) generated by linear combinations of $\omega_0, \ldots, \omega_k$ with non-negative integer (respectively, integer) coefficients. 
\begin{thm} \label{semigroup-prop}
Let $(\tilde q_1, p_1), \ldots, (\tilde q_{\tilde l},p_{\tilde l})$ be pairs of relatively prime positive integers with $p_k \geq 2$, $1 \leq k \leq \tilde l$, and $r$ be a non-negative integer. Let $l$ and $\omega_0, \ldots, \omega_{l+1}$ be as in Definition \ref{virtual-defn}. Assume $\omega_{l +1} > 0$ (so that $\tilde E_{L,C,r}$ is contractible for every curve $C$ with Puiseux pairs $(\tilde q_1, p_1), \ldots, (\tilde q_{\tilde l},p_{\tilde l})$). Then
\begin{enumerate}
\item \label{algebraic-existence} There exists a curve-germ $C$ at $O$ with Puiseux pairs $(\tilde q_1, p_1), \ldots, (\tilde q_{\tilde l},p_{\tilde l})$ such that $\tilde E_{L,C,r}$ is algebraically contractible, iff the semigroup condition \eqref{semigroup-criterion-1} holds for all $k$, $1 \leq k \leq l$.
\item \label{non-algebraic-existence} There exists a curve-germ $C$ at $O$ with Puiseux pairs $(\tilde q_1, p_1), \ldots, (\tilde q_{\tilde l},p_{\tilde l})$ such that $\tilde E_{L,C,r}$ is {\em not} algebraically contractible, iff either \eqref{semigroup-criterion-1} or \eqref{semigroup-criterion-2} {\em fails} for some $k$, $1 \leq k \leq l$.
\end{enumerate}
\end{thm}

We prove the theorem in Section \ref{sec-proof}.
\begin{rem}
Since the weighted dual graph of $\tilde E_{L,C,r}$ for a curve-germ at $C$ at $O$ is completely determined by $r$ and the Puiseux pairs of $C$, Theorem \ref{semigroup-prop} completely characterizes which weighted dual graphs (corresponding to $\tilde E_{L,C,r}$ of Question \ref{contract-2-question}) correspond to only algebraic, only non-algebraic, or both types of contractions.
\end{rem}

\begin{rem}[`Explanation' of the term `virtual poles'] \label{curve-semigroup-remark}
In the set up of Question \ref{contract-2-question}, identify $\pp^2 \setminus L$ with $\cc^2$, so that $(x,y) := (1/u, v/u)$ is a system of coordinates on $\cc^2$. The terminology `virtual poles' for $\omega_0, \ldots, \omega_{l}$ is motivated by the last assertion of the following result which is a reformulation of a fundamental result of the theory of plane algebraic curves with one place at infinity. 
\end{rem}

\begin{thm} [{\cite{abhya-moh-tschirnhausen}, \cite{abhyankar-expansion}, \cite{abhyankar-semigroup}, \cite{sathaye-stenerson}}] \label{curve-thm}
Consider the set up of Theorem \ref{semigroup-prop} and set $p_{\tilde l+1} := 1$. The semigroup condition \eqref{semigroup-criterion-1} is satisfied for all $k$, $1 \leq k \leq l$, iff there exists a curve $C'$ in $\cc^2$ such that $C'$ has only one place at infinity and has a Puiseux expansion at the point at infinity with Puiseux pairs $(\tilde q_1, p_1), \ldots, (\tilde q_{l}, p_{l})$. Moreover, if $C'$ exists, then $\omega_0/p_{l+1}, \ldots, \omega_{l}/p_{l+1}$ are the generators of the {\em semigroup of poles at infinity} on $C'$.
\end{thm}

In the situation of \ref{curve-thm}, the numbers $\omega_k$, $0 \leq k \leq l$, are usually denoted in the literature by $\delta_k$, $0 \leq k \leq l$, and are called the {\em $\delta$-sequence} of $C'$.\\

For positive integers $q,p$, and a curve-germ $C$ at $O$, we say that $C$ is of $(q,p)$-type with respect to $(u,v)$-coordinates iff $C$ has a Puiseux expansion $v = \phi(u)$ such that $(q,p)$ is the only Puiseux pair of $\phi$. The following result is a straightforward corollary of Theorem \ref{semigroup-prop} and the fact (which is a special case of \cite[Proposition 2.1]{herzog}) that the greatest integer not belonging to $\zz_{\geq 0}\langle p, p-q \rangle$ is $ p(p-q) - p - (p-q)$. 
	
\begin{cor} \label{cor-l=1}
Let $p,q$ be positive relatively prime integers and $r$ be a non-negative integer. 
\begin{enumerate}
\item \label{contractibly-q-p} Let $C$ be a $(q,p)$-type curve germ at $O$ with respect to $(u,v)$-coordinates. Then $\tilde E_{L,C,r}$ is contractible iff $r < p(p-q)$.
\item \label{algebraicity-q-p} There is a $(q,p)$-type curve germ $C$ at $O$ with respect to $(u,v)$-coordinates such that $\tilde E_{L,C,r}$ is contractible, but {\em not} algebraically contractible, iff $2p -q < r < p(p-q)$. \qed
\end{enumerate}
\end{cor}

\begin{rem} \label{1-puiseux-remark}
In fact, if $2p -q < r < p(p-q)$, Proposition \ref{effective-answer-Puiseux-1} gives an easy recipe to construct a curve $C$ such that $E_{L,C,r}$ is contractible, but {\em not} algebraically contractible; e.g.\ the curve given by $(v - f(u))^p = u^q$ would suffice for any polynomial $f(u) \in \cc[u]$ such that the coefficient of $u^2$ in $f(u)$ is non-zero. In Examples \ref{non-example} and \ref{non-example-again} we considered the case $(q,p) = (3,5)$ and took $f(u) := u^2$. 
\end{rem}

\begin{example}[Dual graphs arising from {\em only} non-algebraic contractions]\label{non-2-remexample}
If $(\tilde q_1,p_1), (\tilde q_2, p_2)$ are pairs of relatively prime positive integers such that $p_1, p_2 \geq 2$, $\tilde q_1 < p_1$ and 
\begin{align} \label{non-2-eqn}
\tilde q_2 = (p_1-\tilde q_1)(p_2-1)(p_1-1) + p_1(p_2+1),
\end{align}
then the `fact' stated preceding Corollary \ref{cor-l=1} implies that the condition \eqref{semigroup-criterion-1} fails for $k = 2$ and therefore Theorem \ref{semigroup-prop} implies that if $C$ is any curve germ at $O$ such that its Puiseux expansion in $u$ has Puiseux pairs $(\tilde q_1, p_1),(\tilde q_2,p_2)$ then $\tilde E_{L,C,r}$ for $r = 1$ corresponds only to non-algebraic analytic contractions. Setting $(\tilde q_1,p_1) = (3,5)$ and $p_2 = 2$ in equation \eqref{non-2-eqn} gives $\tilde q_2 = 23$. Figure \ref{fig:non-2-example} depicts the dual graph of $\tilde E_{L,C,1}$ for a curve $C$ with Puiseux pairs $\{(3,5),(23,2)\}$ (for its Puiseux expansion in $u$).

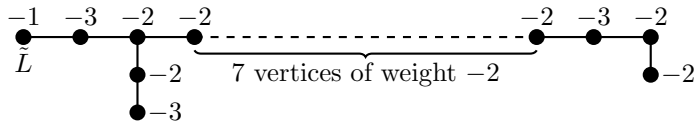
\begin{figure}[htp]
\begin{center}
\begin{tikzpicture}[scale=1, font = \small] 	
 	\pgfmathsetmacro\dashedge{4.5}	
 	\pgfmathsetmacro\edge{.75}
 	\pgfmathsetmacro\vedge{.5}

 	\draw[thick] (-2*\edge,0) -- (\edge,0);
 	\draw[thick] (0,0) -- (0,-2*\vedge);
 	\draw[thick, dashed] (\edge,0) -- (\edge + \dashedge,0);
 	\draw[thick] (\edge + \dashedge,0) -- (3*\edge + \dashedge,0);
 	\draw[thick] (3*\edge + \dashedge,0) -- (3*\edge + \dashedge,-\vedge);
 	
 	\fill[black] ( - 2*\edge, 0) circle (3pt);
 	\fill[black] (-\edge, 0) circle (3pt);
 	\fill[black] (0, 0) circle (3pt);
 	\fill[black] (0, -\vedge) circle (3pt);
 	\fill[black] (0, - 2*\vedge) circle (3pt);
 	\fill[black] (\edge, 0) circle (3pt);
 	\fill[black] (\edge + \dashedge, 0) circle (3pt);
 	\fill[black] (2*\edge + \dashedge, 0) circle (3pt);
 	\fill[black] (3*\edge + \dashedge, 0) circle (3pt);
 	\fill[black] (3*\edge + \dashedge, -\vedge) circle (3pt);
 	
 	\draw (-2*\edge,0)  node (e0up) [above] {$-1$};
 	\draw (-2*\edge,0 )  node (e0down) [below] {$\tilde L$};
 	\draw (-\edge,0 )  node (e1up) [above] {$-3$};
 	\draw (0,0 )  node (middleup) [above] {$-2$};	
 	\draw (0,-\vedge )  node (down1) [right] {$-2$};
 	\draw (0, -2*\vedge)  node (down2) [right] {$-3$};
 	\draw (\edge,0)  node (e+1-up) [above] {$-2$};
 	\draw (\edge + \dashedge,0)  node (e-last-1-up) [above] {$-2$};
 	\draw (2*\edge + \dashedge,0)  node [above] {$-3$};
 	\draw (3*\edge + \dashedge,0)  node [above] {$-2$};
 	\draw (3*\edge + \dashedge,-\vedge)  node [right] {$-2$};
 	
 	\draw [thick, decoration={brace, mirror, raise=5pt},decorate] (\edge,0) -- (\edge + \dashedge,0);
 	\draw (\edge + 0.5*\dashedge,-0.5) node [text width= 5cm, align = center] (extranodes) {$7$ vertices of weight $-2$};
 			 	
\end{tikzpicture}

\caption{A dual graph of $\tilde E_{L,C,r}$ which corresponds to only non-algebraic contractions}\label{fig:non-2-example}
\end{center}
\end{figure}

\end{example}

\section{Background: valuations, key polynomials, semidegrees and key forms} \label{sec-defn}
In this section we define the terms used in Theorem \ref{key-thm-effective} and list the background material we use for its proof. In Subsection \ref{pre-key-section} we introduce `key polynomials' and `generic Puiseux series' associated to discrete valuations on the field of rational functions on $\cc^2$. In Subsection \ref{key-section} we translate the notions of Subsection \ref{pre-key-section} into the language of {\em semidegrees} (which are simply negative of valuations) which is more convenient in dealing with valuations {\em centered at infinity}. This means in particular that instead of Puiseux series and key polynomials, we work with `degree-wise Puiseux series' and `key forms'. In Subsection \ref{essential-section} we introduce `essential key forms' - these are a sub-collection of key forms which by themselves contain all the information about the semidegree (or equivalently, the valuation) and list some of their properties we use in the sequel. We also describe the algorithm to compute essential key forms of a semidegree from the `generic degree-wise Puiseux series'. Finally, in Subsection \ref{degree-like-section} we recall briefly the language of {\em degree-like functions} and corresponding compactifications - we use only a very simple result (namely Proposition \ref{basic-prop}) from this subsection.

\subsection{Puiseux series and Key Polynomials corresponding to valuations} \label{pre-key-section}

\begin{defn}[Divisorial valuations] \label{divisorial-defn}
Let $u, v$ be polynomial coordinates on $X' \cong \cc^2$. A {\em discrete valuation} on $\cc(u,v)$ is a map $\nu: \cc(u,v)\setminus\{0\} \to \zz$, such that for all $f,g \in \cc(u,v)\setminus \{0\}$,
\begin{compactenum}
\item $\nu(f+g) \geq \min\{\nu(f), \nu(g)\}$,
\item $\nu(fg) = \nu(f) + \nu(g)$.
\end{compactenum}
Let $\bar X'$ be an algebraic compactification of $X'$. A discrete valuation $\nu$ on $\cc(u,v)$ is called {\em divisorial} iff there exists a normal algebraic surface $Y$ equipped with a birational morphism $\sigma: Y \to \bar X$ and a curve $C_\nu$ on $Y$ such that for all non-zero $f \in \cc[x,y]$, $\nu(f)$ is the order of vanishing of $\sigma^*(f)$ along $C_\nu$. The {\em center} of $\nu$ on $\bar X'$ is $\sigma(C_\nu)$. 
\end{defn}

Let $u,v$ be as in Definition \ref{divisorial-defn} and $\nu$ be a divisorial valuation on $\cc(u,v)$ with $\nu(u) >0$ and $\nu(v) > 0$. We recall two of the standard ways of representing a valuation: by a Puiseux series and by {\em key polynomials} \cite{maclane-key}. 

\begin{defn}[Meromorphic Puiseux series] \label{puiseuxnition}
Recall that the ring of Puiseux series in $u$ is 
$$\bigcup_{p=1}^\infty \cc[[u^{1/p}]] = \left\{\sum_{k=0}^\infty a_k u^{k/p} : p \in \zz,\ p \geq 1\right\}.$$
The field of {\em Meromorphic Puiseux series} is the quotient field of the ring of Puiseux series, i.e.\ the following field: 
$$\bigcup_{p=1}^\infty \cc((u^{1/p})) = \left\{\sum_{k=k_0}^\infty a_k u^{k/p} : k_0, p \in \zz,\ p \geq 1\right\},$$
where we denoted by $\cc((u^{1/p}))$ the field of Laurent series in $u^{1/p}$. Let $\phi$ be a meromorphic Puiseux series in $u$. The {\em polydromy order} \cite[Chapter 1]{casas-alvero} of $\phi$ is the smallest positive integer $p$ such that $\phi \in \cc((u^{1/p}))$. For any $r \in \qq$, let us denote by $[\phi]_{<r}$ (resp.\ $[\phi]_{\leq r}$) sum of all terms of $\phi$ with order less than (resp.\ less than or equal to) $r$. Then the {\em Puiseux pairs} of $\phi$ are the unique sequence of pairs of relatively prime integers $(q_1, p_1), \ldots, (q_k,p_k)$ such that the polydromy order of $\phi$ is $p_1\cdots p_k$, and for all $j$, $1 \leq j \leq k$,
\begin{compactenum}
\item $p_j \geq 2$,
\item $[\phi]_{<\frac{q_j}{p_1\cdots p_j}} \in \cc((u^{\frac{1}{p_0\cdots p_{j-1}}}))$ (where we set $p_0 := 1$), and 
\item $[\phi]_{\leq \frac{q_j}{p_1\cdots p_j}} \not\in \cc((u^{\frac{1}{p_0\cdots p_{j-1}}}))$.
\end{compactenum}
The {\em characteristic exponents} of $\phi$ are $q_1/p_1, q_2/(p_1p_2), \ldots, q_k/(p_1 \cdots p_k)$. 
\end{defn}
\begin{prop}[Valuation via Puiseux series: reformulation of {\cite[Proposition 4.1]{favsson-tree}}] \label{valseux}
There exists a {\em Puiseux polynomial} (i.e.\ a Puiseux series with finitely many terms) $\phi_\nu$ in $u$ and a rational number $r_\nu$ such that for all $f \in \cc[u,v]$, 
\begin{align}
\nu(f) = \nu(u)\ord_u\left( f(u,v)|_{v = \phi_\nu(u) + \xi u^{r_\nu}}\right), \label{favsson-puiseux}
\end{align}
where $\xi$ is an indeterminate. 
\end{prop}

\begin{defn} \label{generic-puiseux-nition}
If $\phi_\nu$ and $r_\nu$ are as in Proposition \ref{valseux}, we say that $\tilde \phi_\nu(u,\xi):= \phi_\nu(x) + \xi u^{r_\nu}$ is the {\em generic Puiseux series} associated to $\nu$. 
\end{defn}

\begin{defn}[{Key Polynomials of \cite{maclane-key} after \cite[Chapter 2]{favsson-tree}}] \label{key-pol-defn}
Let $\nu$ be as above. A sequence of polynomials $\tilde g_0, \tilde g_1, \ldots, \tilde g_{n+1} \in \cc[u,v]$ is called the sequence of {\em key polynomials} for $\nu$ if the following properties are satisfied:
\begin{compactenum}
\addtocounter{enumi}{-1}
\item $\tilde g_0 = u$, $\tilde g_1 = v$.
\item \label{semigroup-val-property} Let $\tilde \omega_j := \nu(\tilde g_j)$, $0 \leq j \leq n+1$. Then 
\begin{align*}
\tilde \omega_{j+1} > \alpha_j \tilde \omega_j = \sum_{i = 0}^{j-1}\beta_{j,i}\tilde \omega_i\ \text{for}\ 1 \leq j \leq n,
\end{align*}
where $\alpha_j \in \zz_{> 0}$ and $\beta_{j,i} \in \zz_{\geq 0}$ satisfy
\begin{gather*}
\alpha_j = \min\{\alpha \in \zz_{> 0}:\alpha\tilde \omega_j \in \zz \tilde \omega_0 + \cdots + \zz \tilde \omega_{j-1}\}\ \text{for}\ 1 \leq j \leq n,\ \text{and}\\
0 \leq \beta_{j,i} < \alpha_i\ \text{for}\ 1 \leq i < j \leq n.
\end{gather*}

\item For $1 \leq j \leq n$, there exists $\theta_j \in \cc^*$ such that 
\begin{align*}
\tilde g_{j+1} = \tilde g_j^{\alpha_j} - \theta_j \tilde g_0^{\beta_{j,0}} \cdots \tilde g_{j-1}^{\beta_{j,j-1}}.
\end{align*}

\item \label{pre-generating-property} Let $u_0, \ldots, u_{n+1}$ be indeterminates and $\tilde \omega$ be the {\em weighted order} on $\cc[u_0, \ldots, u_{n+1}]$ corresponding to weights $\tilde \omega_j$ for $u_j$, $0 \leq j \leq n+1$ (i.e.\ the value of $\tilde \omega$ on a polynomial is the smallest `weight' of its monomials). Then for every polynomial $f \in \cc[u,v]$, 
\begin{align*}
\nu(f) = \max\{\tilde \omega(F): F \in \cc[u_0, \ldots, u_{n+1}],\ F(\tilde g_0, \ldots, \tilde g_{n+1}) = f\}.
\end{align*}
\end{compactenum} 
\end{defn}

\begin{thm}[{\cite[Theorem 2.29]{favsson-tree}}] \label{key-pol-thm}
There is a unique and finite sequence of key polynomials for $\nu$. 
\end{thm}

\begin{example}
If $\nu$ is the multiplicity valuation at the origin, then the generic Puiseux series corresponding to $\nu$ is $\tilde \phi_\nu = \xi u$ and the key polynomials are $u,v$.
\end{example}

\begin{example}
If $\nu$ is the weighted order in $(u,v)$-coordinates corresponding to weights $p$ for $u$ and $q$ for $v$ with $p,q$ positive integers, then $\tilde \phi_\nu = \xi u^{q/p}$ and the key polynomials are again $u,v$.
\end{example}

\begin{example} \label{general-example}
Let $C$ be a singular irreducible analytic curve-germ at the origin with Puiseux expansion $v = \phi(u)$. Pick any non-negative integer $r$ and construct $E^*_{L,C,r}$ (with $L := \{u=0\}$) as in Question \ref{contract-2-question}. Then the generic Puiseux series associated to the valuation $\nu$ corresponding to $E^*_{L,C,r}$ is
\begin{align*} 
\tilde \phi_\nu(x,\xi) &:= [\phi(u)]_{< (\tilde q+r)/p} + \xi u^{(\tilde q+r)/p},\quad \text{where}\\
p &:= \text{the polydromy order of $\phi$,} \\
\tilde q/p &:= \text{the {\em last} characteristic exponent of $\phi$,}\\
[\phi(u)]_{< (\tilde q+r)/p} &:= \text{sum of all terms of $\phi(u)$ with order less than $(\tilde q+r)/p$.}
\end{align*}
\end{example}

\begin{example} \label{non-example-aagain}
Let $C_1$ and $C_2$ be the curves from Example \ref{non-example-again}. We apply the construction of Example \ref{general-example} to $C_1$ and $C_2$. The Puiseux expansion for $C_1$ and $C_2$ at the origin are respectively given by: $v = u^{3/5}$ and $v= u^{3/5} + u^2$. It follows that the generic Puiseux series for the valuation of Example \ref{general-example} applied to $C_i$'s are:
\begin{align*}
\tilde \phi_{\nu_1} &= \begin{cases}
			\xi u^{3/5} & \text{if}\ r = 0,\\			
			u^{3/5} + \xi u^{(3+r)/5} & \text{if}\ r\geq 1.
			\end{cases}
		&& &
\tilde \phi_{\nu_2} &= \begin{cases}
			\xi u^{3/5} & \text{if}\ r = 0,\\			
			u^{3/5} + \xi u^{(3+r)/5} & \text{if}\ 1 \leq r\leq 7, \\
			u^{3/5} + u^2 + \xi u^{(3+r)/5} & \text{if}\ 8 \leq r.
			\end{cases}
\end{align*}		
The sequence of key polynomials for $\nu_1$ and $\nu_2$ for $0 \leq r < 10$ are as follows:
\begin{align*}
\parbox{2cm}{key polynomials for $\nu_1$} 
		 &= \begin{cases}
			u,v & \text{if}\ r = 0,\\			
			u,v, v^5 - u^3 & \text{if}\ r\geq 1.
			\end{cases}
		&& &
\parbox{2cm}{key polynomials for $\nu_2$}
		 &= \begin{cases}
			u,v & \text{if}\ r = 0,\\			
			u,v,v^5 - u^3 & \text{if}\ 1 \leq r\leq 7, \\
			u,v, v^5 - u^3, v^5 - u^3 - 5v^4u^2 & \text{if}\ 8 \leq r\leq 9.
			\end{cases}
\end{align*}			
In particular, note that for $r \geq 1$ the last key polynomials are precisely the $\tilde f_i$'s of Example \ref{non-example-again}. This is in fact the key observation in our proof of Proposition \ref{effective-answer-Puiseux-1} in Section \ref{sec-proof}.
\end{example}

\subsection{Degree-wise Puiseux series and Key forms corresponding to semidegrees} \label{key-section}
\begin{defn}[Degree-wise Puiseux series]
The field of {\em degree-wise Puiseux series} in $x$ is 
$$\dpsxc := \bigcup_{p=1}^\infty \cc((x^{-1/p})) = \left\{\sum_{j \leq k} a_j x^{j/p} : k,p \in \zz,\ p \geq 1 \right\},$$
where for each integer $p \geq 1$, $\cc((x^{-1/p}))$ denotes the field of Laurent series in $x^{-1/p}$. In particular $\phi(x)$ is a degree-wise Puiseux series in $x$ iff $\phi(1/x)$ is a meromorphic Puiseux series in $x$. The notions regarding meromorphic Puiseux series introduced in Definition \ref{puiseuxnition} extend naturally to the setting of degree-wise Puiseux series. In particular, if $\phi$ is a degree-wise Puiseux series in $x$, then the {\em Puiseux pairs} of $\phi$ are $(-q_1, p_1), \ldots, (-q_l, p_l)$, where $(q_1, p_1), \ldots, (q_l, p_l)$ are the Puiseux pairs of $\phi(1/x)$; similarly, the {\em characteristic exponents} of $\phi$ are $-q_k/(p_1 \cdots p_k)$ for $1 \leq k \leq l$ and the {\em polydromy order} of $\phi$ is the smallest positive integer $p$ such that the exponents of $x$ in all terms of $\phi$ are of the form $q/p$, $q \in \zz$. Let $\phi = \sum_{q \leq q_0} a_q x^{q/p}$, where $p$ is the polydromy order of $\phi$. Then the {\em conjugates} of $\phi$ are $\phi_j := \sum_{q \leq q_0} a_q \zeta^q x^{q/p}$, $1 \leq j \leq p$, where $\zeta$ is a primitive $p$-th root of unity. The usual factorization of polynomials in terms of Puiseux series implies the following
\end{defn}

\begin{thm} \label{dpuiseux-factorization}
Let $f \in \cc[x,y]$. Then there are unique (up to conjugacy) degree-wise Puiseux series $\phi_1, \ldots, \phi_k$, a unique non-negative integer $m$ and $c \in \cc^*$ such that
$$f = cx^m \prod_{i=1}^k \prod_{\parbox{1.75cm}{\scriptsize{$\phi_{ij}$ is a con\-ju\-ga\-te of $\phi_i$}}}\mkern-27mu \left(y - \phi_{ij}(x)\right)$$
\end{thm}

\begin{defn}[cf.\ Definition \ref{sub-defn}] \label{semi-defn}
Let $x,y$ be indeterminates. A {\em divisorial semidegree} on $\cc(x,y)$ is a map $\delta : \cc(x,y)\setminus\{0\} \to \zz$ such that $-\delta$ is a divisorial valuation.
\end{defn}
Let $\delta$ be a divisorial semidegree on $\cc(x,y)$ such that $\delta(x) > 0$. Set $u := 1/x$ and $v := y/x^k$ for some $k$ such that $\delta(y) < k\delta(x)$. Applying Proposition \ref{valseux} to $\nu := -\delta$ and $\cc(u,v)$ and then translating in terms of $(x,y)$-coordinates yields the following result. 

\begin{prop}[{\cite[Theorem 1.2]{sub2-1}}] \label{valdeg}
There exists a {\em degree-wise Puiseux polynomial} (i.e.\ a degree-wise Puiseux series with finitely many terms) $\phi_\delta \in \dpsxc$ and a rational number $r_\delta < \ord_x(\phi_\delta)$ such that for every polynomial $f \in \cc[x,y]$, 
\begin{align}
\delta(f) = \delta(x)\deg_x\left( f(x,y)|_{y = \phi_\delta(x) + \xi x^{r_\delta}}\right), \label{phi-delta-defn}
\end{align}
where $\xi$ is an indeterminate. 
\end{prop}

\begin{defn} \label{generic-deg-wise-defn}
If $\phi_\delta$ and $r_\delta$ are as in Proposition \ref{valdeg}, we say that $\tilde \phi_\delta(x,\xi):= \phi_\delta(x) + \xi x^{r_\delta}$ is the {\em generic degree-wise Puiseux series} associated to $\delta$. Let the Puiseux pairs of $\phi_\delta$ be $(q_1, p_1), \ldots, (q_l,p_l)$. Express $r_\delta$ as $q_{l+1}/(p_1 \cdots p_lp_{l+1})$ where $p_{l+1} \geq 1$ and $\gcd(q_{l+1}, p_{l+1}) = 1$. Then the {\em formal Puiseux pairs} of $\tilde \phi_\delta$ are $(q_1, p_1), \ldots, (q_{l+1},p_{l+1})$. Note that 
\begin{enumerate}
\item $\delta(x) = p_1 \cdots p_{l+1}$,
\item it is possible that $p_{l+1} = 1$ (as opposed to other $p_k$'s, which are always $\geq 2$). 
\end{enumerate}
\end{defn}

We will need the following geometric interpretation of generic degree-wise Puiseux series: let $X := \cc^2$ with coordinates $(x,y)$, $\bar X$ be a normal analytic compactification of $X$ with an irreducible curve $C_\infty$ at infinity and $\delta$ be the order of pole along $C_\infty$. Let $\bar X^0 \cong \pp^2$ be the compactification of $X$ induced by the map $(x,y) \mapsto [1:x:y]$, $\sigma: \bar X \dashrightarrow \bar X^0$ be the natural bimeromorphic map, and $S$ (resp.\ $S'$) be the finite set of points of indeterminacy of $\sigma$ (resp.\ $\sigma^{-1})$. Assume that $\sigma$ maps $C_\infty \setminus S$ to a point $O \in L_\infty := \bar X^0 \setminus X$. It then follows that $\sigma^{-1}$ maps $L_\infty \setminus S'$ to a point $P_\infty \in C_\infty$.

\begin{prop}[{\cite[Proposition 4.2]{sub2-1}}] \label{prop-param}
Let $\tilde \phi_\delta(x,\xi)$ be the generic degree-wise Puiseux series associated to $\delta$ and $\gamma$ be an (analytically) irreducible curve-germ at $O$ (on $\bar X^0$) which is distinct from the germ of $L_\infty$. Then the strict transform of $\gamma$ on $\bar X$ intersects $C_\infty \setminus \{P_\infty\}$ iff $\gamma \cap X$ (i.e.\ the {\em finite} part of $\gamma$) has a parametrization of the form 
\begin{align}
t \mapsto (t, \tilde \phi_\delta(t,\xi)|_{\xi = c} + \lot)\quad \text{for}\ |t| \gg 0 \label{dpuiseux-param} \tag{$*$}
\end{align}
for some $c \in \cc$, where $\lot$ means `lower order terms' (in $t$).
\end{prop}

Now we adapt the notion of key polynomials to the case of semidegrees. Let $(u,v)$ be as in the paragraph preceding Proposition \ref{valdeg} and $\tilde g_0 = u, \tilde g_1 = v, \tilde g_2, \ldots, \tilde g_{n+1} \in \cc[u,v]$ be the key polynomials of $-\delta$. Define $g_0 := x$ and
\begin{align}
g_j(x,y) := x^{k\deg_{v}(\tilde g_j)}\tilde g_j(1/x,y/x^k) \in \cc[x,x^{-1},y],	\quad 1 \leq j \leq n+1. \label{form-from-pol}
\end{align}
The properties of key polynomials imply that $g_0, \ldots, g_{n+1}$ have the analogous properties of {\em key forms} of Definition \ref{key-defn} below. That these are {\em unique} key forms of $\delta$ follows from Theorem \ref{key-thm} which is a straightforward corollary of Theorem \ref{key-pol-thm}. The main difference between key polynomials and the key forms is that the latter may {\em not} be polynomials (hence the word `form'\footnote{We use the word `form' in particular because of the following reason: consider the map $\pi$ and the weighted degree $\omega$ of Property \ref{generating-property} of Definition \ref{key-defn}. Then the key forms are in fact images under $\pi$ of certain {\em weighted homogeneous forms} (with respect to $\omega$) in $B$.} instead of `polynomial') - see Remark \ref{generating-remark} and Example \ref{non-example-aaagain}.

\begin{defn}[Key forms] \label{key-defn}
Let $\delta$ be a divisorial semidegree on $\cc[x,y]$ such that $\delta(x) > 0$. A sequence of elements $g_0, g_1, \ldots, g_{n+1} \in \cc[x,x^{-1},y]$ is called the sequence of {\em key forms} for $\delta$ if the following properties are satisfied:
\begin{compactenum}
\let\oldenumi\theenumi
\renewcommand{\theenumi}{P\oldenumi}
\addtocounter{enumi}{-1}
\item $g_0 = x$, $g_1 = y$.
\item \label{semigroup-property} Let $\omega_j := \delta(g_j)$, $0 \leq j \leq n+1$. Then 
\begin{align*}
\omega_{j+1} < \alpha_j \omega_j = \sum_{i = 0}^{j-1}\beta_{j,i}\omega_i\ \text{for}\ 1 \leq j \leq n,
\end{align*}
where 
\begin{compactenum}
\item $\alpha_j = \min\{\alpha \in \zz_{> 0}: \alpha\omega_j \in \zz \omega_0 + \cdots + \zz \omega_{j-1}\}$ for $1 \leq j \leq n$,
\item $\beta_{j,i}$'s are integers such that $0 \leq \beta_{j,i} < \alpha_i$ for $1 \leq i < j \leq n$ (in particular, $\beta_{j,0}$'s are allowed to be {\em negative}). 
\end{compactenum}

\item \label{next-property} For $1 \leq j \leq n$, there exists $\theta_j \in \cc^*$ such that 
\begin{align*}
g_{j+1} = g_j^{\alpha_j} - \theta_j g_0^{\beta_{j,0}} \cdots g_{j-1}^{\beta_{j,j-1}}.
\end{align*}

\item \label{generating-property} Let $y_1, \ldots, y_{n+1}$ be indeterminates and $\omega$ be the {\em weighted degree} on $B := \cc[x,x^{-1},y_1, \ldots, y_{n+1}]$ corresponding to weights  $\omega_0$ for $x$ and $\omega_j$ for $y_j$, $0 \leq j \leq n+1$ (i.e.\ the value of $\omega$ on a polynomial is the maximum `weight' of its monomials). Then for every polynomial $g \in \cc[x,x^{-1},y]$, 
\begin{align}
\delta(g) = \min\{\omega(G): G(x,y_1, \ldots, y_{n+1}) \in B,\ G(x,g_1, \ldots, g_{n+1}) = g\}. \label{generating-eqn}
\end{align}
\end{compactenum} 
\end{defn}

\begin{thm} \label{key-thm}
There is a unique and finite sequence of key forms for $\delta$. 
\end{thm}

\begin{rem} \label{generating-remark}
Compare Property \ref{pre-generating-property} of key polynomials and Property \ref{generating-property} of key forms: the difference is due to the fact that the change of coordinates $(x,y) \mapsto (u,v)$ introduces negative powers of $x$. One of our technical results states that if the $g_j$'s are polynomial in $(x,y)$, then in fact the minimum of the right hand side of \eqref{generating-eqn} is achieved with some $G$ which is also a polynomial in $(x,y_1, \ldots, y_{n+1})$ (Corollary \ref{minimum-cor}).  
\end{rem}

\begin{example} \label{deg-dpuiseux-key}
If $\delta$ is a weighted degree in $(x,y)$-coordinates corresponding to weights $p$ for $x$ and $q$ for $y$ with $p,q$ positive integers, then the generic degree-wise Puiseux series corresponding to $\delta$ is $\tilde \phi_\delta = \xi x^{q/p}$ and the key polynomials are $g_0 = x$ and $g_1 = y$. Note that $\cc^2$ is embedded into the weighted projective space $\pp^2(1,p,q)$ via the embedding $(x,y) \mapsto [1:x:y]$, then $\delta$ is precisely the order of the pole along the curve at infinity.
\end{example}

\begin{example}[cf.\ Example \ref{general-example}] \label{general-example-again}
Consider the set up of the paragraph preceding Proposition \ref{prop-param}. Let $O := [0:1:0] \in L_\infty$, so that $(u,v) := (1/x,y/x)$ is a system of coordinates at $O$ with $L_\infty = \{u=0\}$. Let $C$ be a curve-germ at $O$ with Puiseux expansion $v = \phi(u)$. Then the {\em degree-wise Puiseux expansion} of $C$ at $O$ is 
$$y := x\phi(1/x).$$
Set $\tilde \phi(x) := x\phi(1/x)$. Let $r$ be a non-negative integer, $E^*_{L_\infty,C,r}$ be as in Question \ref{contract-2-question}, and $\delta$ be the order of pole along $E^*_{L_\infty,C,r}$. Then the generic degree-wise Puiseux series associated to $\delta$ is
\begin{align*} 
\tilde \phi_\delta(x,\xi) &:= [\tilde \phi(x)]_{> (q-r)/p} + \xi x^{(q-r)/p},\quad \text{where}\\
p &:= \text{the polydromy order of $\tilde \phi$,} \\
q/p &:= \text{the {\em last} characteristic exponent of $\tilde \phi$,}\\
[\tilde \phi(x)]_{> (q-r)/p} &:= \text{sum of all terms of $\tilde \phi(x)$ with order greater than $(q-r)/p$.}
\end{align*}
\end{example}

\begin{example} \label{non-example-aaagain}
Set $u := 1/x$ and $v := y/x$. Let $\nu_1$ and $\nu_2$ be valuations from Example \ref{non-example-aagain} and set $\delta_i := -\nu_i$, $1 \leq i \leq 2$. It follows from Examples \ref{non-example-aagain} and \ref{general-example-again} that the generic degree-wise Puiseux series for $\delta_i$'s are:
\begin{align*}
\tilde \phi_{\delta_1} &= \begin{cases}
			\xi x^{2/5} & \text{if}\ r = 0,\\			
			x^{2/5} + \xi x^{(2-r)/5} & \text{if}\ r\geq 1.
			\end{cases}
		&& &
\tilde \phi_{\delta_2} &= \begin{cases}
			\xi x^{2/5} & \text{if}\ r = 0,\\			
			x^{2/5} + \xi x^{(2-r)/5} & \text{if}\ 1 \leq r\leq 7, \\
			x^{2/5} + x^{-1} + \xi x^{(2-r)/5} & \text{if}\ 8 \leq r.
			\end{cases}
\end{align*}		
The sequence of key polynomials for $\delta_1$ and $\delta_2$ for $0 \leq r < 10$ are as follows:
\begin{align*}
\parbox{2cm}{key polynomials for $\delta_1$} 
		 &= \begin{cases}
			x,y & \text{if}\ r = 0,\\			
			x,y, y^5 - x^2 & \text{if}\ r\geq 1.
			\end{cases}
		&& &
\parbox{2cm}{key polynomials for $\delta_2$}
		 &= \begin{cases}
			x,y & \text{if}\ r = 0,\\			
			x,y, y^5 - x^2 & \text{if}\ 1 \leq r\leq 7, \\
			x,y, y^5 - x^2, y^5 - x^2 - 5y^4x^{-1} & \text{if}\ 8 \leq r\leq 9.
			\end{cases}
\end{align*}			
In particular, for $8 \leq r \leq 9$, the last key polynomial for $\delta_2$ is {\em not} a polynomial. It is instructive to contrast this with the fact that for these values of $r$, both $\delta_i$'s are in fact {\em positive} on $\cc[x,y]\setminus\{0\}$ (since Example \ref{non-example-again} shows that $\tilde E_{L,C_2,r}$ is contractible).
\end{example}

\subsection{An algorithm to compute key forms from degree-wise Puiseux series} \label{essential-section}
Let $\delta$ be a divisorial semidegree on $\cc[x,y]$ such that $\delta(x) > 0$ and $g_0, \ldots, g_{n+1}$ be the key forms of $\delta$. Pick the subsequence $g_{j_1}, g_{j_2}, \ldots, g_{j_m}$ of $g_j$'s consisting of all $g_{j_k}$ such that $\alpha_{j_k} > 1$ (where $\alpha_{j_k}$ is as in Property \ref{semigroup-property} of Definition \ref{key-defn}). Set 
\begin{align*}
f_k &:= \begin{cases}
			g_0 = x & \text{if}\ k = 0,\\
			g_{j_k} & \text{if}\ 1 \leq k \leq m,\\
			g_{n+1}	& \text{if}\ k = m+1.
		\end{cases}
\end{align*}
We say that $f_0, \ldots, f_{m+1}$ are the {\em essential key forms} of $\delta$. An application of Theorem \ref{key-thm} immediately yields the following properties of essential key forms

\begin{cor} \label{essential-corollary}
$f_0, \ldots, f_{m+1}$ are unique elements in $\cc[x,x^{-1},y]$ such that
\begin{compactenum}
\let\oldenumi\theenumi
\renewcommand{\theenumi}{P$\oldenumi'$}
\addtocounter{enumi}{-1}
\item $f_0 = x$.
\item $f_1 = y - \sum_{i=1}^{k_0} c_i x^{\beta_0-i}$ for non-negative integers $k_0$ and $\beta_0$ such that $\delta(y) = \beta_0\delta(x)$ and $\delta(f_1) < (\beta_0-k_0)\delta(x)$.
\item \label{ess-next-property} Let $\omega_k := \delta(f_k)$, $0 \leq k \leq m+1$ and $\alpha_k := \min\{\alpha \in \zz_{> 0}; \alpha\omega_k \in \zz \omega_0 + \cdots + \zz \omega_{k-1}\}$ for $1 \leq k \leq m+1$. Then 
\begin{compactenum}
\item $\alpha_k \geq 2$ for each $k$, $1 \leq k \leq m$.
\end{compactenum}
Moreover, for each $k$, $1 \leq k \leq m$, $f_{k+1} = f_k^{\alpha_k} - \sum_{i=0}^{k_n}c_{k,i} f_0^{\beta^i_{k,0}} \cdots f_{k}^{\beta^i_{k,k}}$, where 
\begin{compactenum}
\addtocounter{enumii}{1}
\item $k_n \geq 0$, 
\item $c_{k,i} \in \cc^*$ for all $i$, $0 \leq i \leq k_n$,
\item \label{exponent-bound} $\beta^i_{k,j}$'s are integers such that $0 \leq \beta^i_{k,j} < \alpha_j$ for $1 \leq j \leq k$ and $0 \leq i \leq k_n$,
\item $\beta^0_{k,k} = 0$,
\item \label{decreasing-omega} $\alpha_k \omega_k = \sum_{j = 0}^{k-1}\beta^0_{k,j}\omega_j > \sum_{j = 0}^{k}\beta^1_{k,j}\omega_j > \cdots > \sum_{j = 0}^{k}\beta^{k_n}_{k,j}\omega_j > \omega_{k+1}$.
\end{compactenum} 
\item \label{ess-generating-property} Let $y_1, \ldots, y_{m+1}$ be indeterminates and $\omega$ be the {\em weighted degree} on $B:= \cc[x,x^{-1},y_1, \ldots, y_{m+1}]$ corresponding to weights $\omega_0$ for $x$ and $\omega_k$ for $y_k$, $1 \leq k \leq m+1$ (i.e.\ the value of $\omega$ on a polynomial is the maximum `weight' of its monomials). Let $\pi:B \to \cc[x,x^{-1},y]$ be the ring homomorphism that sends $x \mapsto x$ and $y_k \to f_k$, $1 \leq k \leq m+1$. Then for every polynomial $f \in \cc[x,y]$, 
\begin{align*}
\delta(f) = \min\{\omega(F): F \in B,\ \pi(F) = f\}.
\end{align*}
\end{compactenum} 
\end{cor}
Let $y_1, \ldots, y_{m+1}$, $\omega$ and $\pi:B \to \cc[x,x^{-1},y]$ be as in Property \ref{ess-generating-property}. For each $k$, $0 \leq k \leq m+1$, let $B_k := \cc[x,x^{-1},y_1, \ldots, y_{k}] \subseteq B$. Define 
\begin{align} \label{F-k}
F_{k+1} &:= \begin{cases}
			y_1 & \text{if}\ k=0,\\
			y_{k}^{\alpha_{k}} - \sum_{i=0}^{k_n}c_{k,i} x^{\beta^i_{k,0}} y_1^{\beta^i_{k,1}} \cdots y_{k}^{\beta^i_{k,k}}, 
				& \text{for}\ 1 \leq k \leq m,\ \text{where $\alpha_k$, $c_{k,i}$ and} \\		
				& \text{$\beta^i_{k,j}$'s are as in Property \ref{ess-next-property}.}
			\end{cases}
\end{align} 
In particular, for each $k$, $1 \leq k \leq m$, $F_{k+1}$ is an element of $B_k$ such that $\pi(F_{k+1}) = f_{k+1}$. We now show how these $F_{k}$'s can be computed from the degree-wise Puiseux series.\\

Assume the generic degree-wise Puisuex series for $\delta$ is 
\begin{align}
\tilde \phi_\delta(x,\xi) 
	&:= \sum_{j=0}^{k'_0} a_{0j}x^{q_0 - j} + \left(a_1 x^{\frac{q_1}{p_1}} + \sum_{j=1}^{k'_1} a_{1j}x^{\frac{q_{1j}}{p_1}} \right)  +  \left(a_2 x^{\frac{q_2}{p_1p_2}} + \sum_{j=1}^{k'_2} a_{2j}x^{\frac{q_{2j}}{p_1p_2}} \right)  \notag\\
	& \mbox{\phantom{$:=\qquad$}} + \cdots + \left(a_{l} x^{\frac{q_{l}}{p_1p_2 \cdots p_{l}}} + \sum_{j=1}^{k'_{l}} a_{l,j}x^{\frac{q_{l,j}}{p_1p_2 \cdots p_{l}}} \right)  +  \xi x^{\frac{q_{l + 1}}{p_1p_2 \cdots p_{l+1}}} \label{tilde-phi-delta}
\end{align}
where $(q_1, p_1), \ldots, (q_{l + 1}, p_{l+1})$ are the {\em formal} Puiseux pairs of $\tilde \phi_\delta$ (Definition \ref{generic-deg-wise-defn}).

\begin{algorithm}[Construction of essential key forms from degree-wise Puiseux series] \label{essential-algorithm}
\mbox{}
\vspace{-5mm}
\paragraph{0. Number of essential key forms:} $l+2$, i.e.\ the {\em last} essential key form is $f_{l+1}$. The essential key forms can be calculated as follows:
\vspace{-5mm}
\paragraph{1. Base step:} $f_0 = x$ and $f_1 = y - \phi_1(x) = y - \sum_{j=0}^{k'_0} a_{0j}x^{q_0 - j}$.
\vspace{-5mm}
\paragraph{2. Inductive step (construction of $f_{k+1}$ assuming that $f_0, \ldots, f_k$ has been calculated, $1 \leq k \leq l$):} Set $F_{k+1,0} := y_k^{p_k}$ and $\tilde f_{k+1,0}(x,\xi) := f_k^{p_k}(x,\tilde \phi_{\delta}(x, \xi))$. Then 
\begin{align*}
\tilde f_{k+1,0}(x,\xi) 
	&=  b'_kx^{\frac{\tilde q_k}{p_1\cdots p_{k-1}}} + \lot
\end{align*}
for some $b'_k \in \cc^*$, where $\lot$ means lower order terms in $x$. Let
\begin{align*}
\tilde \omega_{k+1} 
	&:= \begin{cases}
			\max\left\{\tilde \omega \in \qq :\ \text{coefficient of $x^{\tilde \omega}$ in $\tilde f_{k+1,0}(x,\xi)$ is non-zero and $\tilde \omega \not\in \frac{1}{p_1\cdots p_{k}}\zz$}\right\} 
				& \text{if}\ k < l, \\
			\max\left\{\tilde \omega \in \qq :\ \text{coefficient of $x^{\tilde \omega}$ in $\tilde f_{k+1,0}(x,\xi)$ is in}\ \cc[\xi]\setminus\cc \right\}
				& \text{if}\ k = l.
		\end{cases} 
\end{align*}   
\begin{enumerate}
\renewcommand{\theenumi}{\bf Substep 2.1}
\item \label{substep} Assume $F_{k+1,i} \in B_k$ has been constructed for some $i \geq 0$ with $\tilde f_{k+1, i} := \pi(F_{k+1,i})|_{y = \tilde \phi_{\delta}(x, \xi)}$. If $\tilde \omega_{k+1,i} := \deg_x(\tilde f_{k+1, i}) = \tilde \omega_{k+1}$, then set $F_{k+1} := F_{k+1,i}$, $f_{k+1} := \pi(F_{k+1})$ and  {\bf stop.} Otherwise $\tilde \omega_{k+1,i} > \tilde \omega_{k+1}$ and there are unique integers $\beta^i_{k,0}, \cdots, \beta^i_{k,k}$ such that 
\begin{compactenum}
\item $0 \leq \beta^i_{k,j} < p_j$ for $1 \leq j \leq k$, and
\item $\sum_{j=0}^k \beta^i_{k,j} \omega_j = p_1 \cdots p_{l+1}\tilde \omega_{k+1,i}$.
\end{compactenum}
Let $c_{k,i} \in \cc^*$ be the coefficient of $x^{\tilde \omega_{k+1,i}}$ in $\tilde f_{k+1, i}$. Then set 
\begin{align}
F_{k+1,i+1} := F_{k+1,i} - c_{k,i}x^{\beta^i_{k,0}} y_1^{\beta^i_{k,1}} \cdots y_k^{\beta^i_{k,k}},
\end{align}
and {\bf repeat} \ref{substep}.
\end{enumerate}
\end{algorithm}

\begin{rem} \label{essential-example-remark}
Combining Example \ref{general-example-again} and Algorithm \ref{essential-algorithm} together gives the following algorithm to compute the essential key forms of the semidegree $\delta$ corresponding to $E^*_{L,C,r}$ of Question \ref{contract-2-question} from the Puiseux expansion $v = \phi(u)$ of $C$:
\begin{enumerate}
\let\oldenumi\theenumi
\renewcommand{\theenumi}{Step \oldenumi}
\addtocounter{enumi}{-1}
\item Set $(x,y) := (1/u, v/u)$ so that $(x,y)$ defines a system of coordinates on $\pp^2 \setminus L$. 
\item Apply Example \ref{general-example-again} to compute the generic degree-wise Puiseux series $\tilde \phi_\delta(x,\xi)$ for the semidegree $\delta$ (on $\cc(x,y)$) corresponding to $E^*_{L,C,r}$.
\item Apply Algorithm \ref{essential-algorithm} to compute the essential key forms of $\delta$.
\end{enumerate}
\end{rem}

\begin{rem}\label{essentially-all}
Algorithm \ref{essential-algorithm} in fact produces {\em all} the key forms of $\delta$. More precisely, w.l.o.g.\ we may assume that the key forms of $\delta$ are $f_0, f_{0,1}, \ldots, f_{0,i_0}, \ldots, f_l, f_{l,1}, \ldots, f_{l,i_l}, f_{l+1}$. Then the uniqueness of key polynomials implies that 
\begin{enumerate}
\item $f_{0,i}$'s are of the form $y - \sum_{j=0}^{i'} a_{0j}x^{q_0 - j}$ for some $i' \leq k'_0$ (where $a_{0j}x^{q_0-j}$'s and $k'_0$ are from \eqref{tilde-phi-delta}).
\item For all $k,j$, $1 \leq k \leq l$, $1 \leq j \leq i_k$, $f_{k,j} = \pi(F_{k,j})$  where $F_{k,j}$'s are from Algorithm \ref{essential-algorithm}.
\end{enumerate}
\end{rem}

Algorithm \ref{essential-algorithm} terminates (for each $k$) since $\deg_x(\tilde f_{k+1,i+1}) < \deg_x(\tilde f_{k+1,i})$ for each $i$. We do {\em not} prove the correctness of Algorithm \ref{essential-algorithm} in this article, since it can be proved by a (more or less straightforward, but long) induction on $k$ using simply the defining properties of essential key forms. Here we content ourselves with an example of how it works.

\begin{example} \label{essential-example}
Let $\tilde \phi_\delta(x,\xi) := x^3 + x^2+ x^{5/3} + x + x^{-13/6} + x^{-7/3} + \xi x^{-8/3}$. The formal Puiseux pairs of $\tilde \phi_\delta$ are $(5,3), (-13,2), (-16,1)$. Algorithm \ref{essential-algorithm} implies that $\delta$ has $4$ essential key forms $f_0, f_1, f_2, f_3$. We calculate the $f_j$'s following Algorithm \ref{essential-algorithm}. At first note that $f_0 = x$ and $f_1 = y - x^3 - x^2$. Then
\begin{align}
	y_1 &~\quotequal~ x^{5/3} + x + x^{-13/6} + x^{-7/3} + \xi x^{-8/3},\label{y_1-substitution}
\end{align}
where $\quotequal$ in the above equation denotes the composition of the map $\pi$ of Property \ref{ess-generating-property} (of essential key forms) followed by the substitution $y = \tilde \phi_\delta(x,\xi)$. It follows that
\begin{align*}
y_1^3 &~\quotequal~ x^5 + 3x^{13/3} + 3x^{11/3} + x^3 + 3x^{7/6} + 3x + 3\xi x^{2/3} + \lot,
\end{align*}
where $\lot$ denotes all the terms with degree in $x$ less than $2/3$. Then $\tilde \omega_2 = 7/6$ and we follow \ref{substep} to `absorb' all the terms with degree in $x$ greater than $7/6$. Since $13/3 = 1 + 2\cdot(5/3)$ and $11/3 = 2 + 5/3$, identity \eqref{y_1-substitution} implies that
\begin{align*}
y_1^3	& ~\quotequal~ x^5 + 3x\left(y_1 - x - x^{-13/6} - x^{-7/3} - \xi x^{-8/3}\right)^2  + 3x^2\left(y_1 - x - x^{-13/6} - x^{-7/3} - \xi x^{-8/3}\right)  \\
		& \mbox{\phantom{~\quotequal~}}\qquad + x^3 + 3x^{7/6} + 3x + 3\xi x^{2/3} + \lot\\
		&~\quotequal~ x^5 + 3xy_1^2 - 3x^2y_1 + x^3 + 3x^{7/6} + 3x + 3\xi x^{2/3} + \lot,
\end{align*}
where $\lot$ denotes all the monomials in $x,y_1$ with $\omega$-value less than $2/3$. It follows that 
\begin{subequations} 
\begin{align}
F_2 &= y_1^3 - x^5 - 3xy_1^2 + 3x^2y_1 - x^3, \\
f_2 &= \pi(F_2) = (y - x^3 - x^2)^3 - x^5 - 3x(y - x^3 - x^2)^2 + 3x^2(y - x^3 - x^2) - x^3,  \\
y_2 &~\quotequal~ 3x^{7/6} + 3x + 3\xi x^{2/3} + \lot \label{y_2-substitution}
\end{align}
\end{subequations} 
The calculation for $f_3$ follows in the similar fashion: $p_2 = 2$ and
\begin{align*}
y_2^2	& ~\quotequal~ 9x^{7/3} + 18x^{13/6} + 18\xi x^{11/6} + \lot
\end{align*}
It follows that $\tilde \omega_3 = 11/6$, and we proceed to absorb all the terms with $\omega$-value $> 11/6$ via repeated substitutions. Since $7/3 = -1 + 2\cdot(5/3) + 0\cdot(7/6)$ and $13/6 = 1 + 0\cdot(5/3) + 7/6$, identities \eqref{y_1-substitution} and \eqref{y_2-substitution} imply that
\begin{align*}{3}
y_2^2	& ~\quotequal~ 9x^{-1}\left(y_1 - x - x^{-13/6} - x^{-5/2} - \xi x^{-8/3}\right)^2  + 18x\frac{1}{3}\left(y_2 - 3x - 3\xi x^{2/3} - \lot\right) + 18\xi x^{11/6} + \lot \\
		& ~\quotequal~ 9x^{-1} y_1^2 + 6xy_2 -18x^2+ 18\xi x^{11/6} + \lot,
\end{align*}
so that $F_3 = y_2^2 - 9x^{-1} y_1^2 -  6xy_2 + 18x^2$ and $f_3 = \pi(F_3)$. This completes the computation of the essential key forms of $\delta$.
\end{example}

In Proposition \ref{essential-prop} below we list some properties of essential key forms $f_k$, $0 \leq k \leq l+1$, which we use in the sequel. We omit the proof of the proposition, since all of its assertions are straightforward implications of Algorithm \ref{essential-algorithm} and Remark \ref{essentially-all}. \\

Let $\xi_1, \ldots, \xi_{l+1}$ be new indeterminates, and for each $k$, $1 \leq k \leq l+1$, let $\delta_k$ be the semidegree on $\cc[x,y]$ corresponding to the generic degree-wise Puisuex series
\begin{align} 
\tilde \phi_k(x,\xi_k) 
			&:= \phi_k(x)  +  \xi_k x^{\frac{q_{k}}{p_1p_2 \cdots p_{k}}},\quad \text{where}\\
\phi_k(x)	&:= \sum_{j=0}^{k'_0} a_{0j}x^{q_0 - j} + \left(a_1 x^{\frac{q_1}{p_1}} + \sum_{j=1}^{k'_1} a_{1j}x^{\frac{q_{1j}}{p_1}} \right)  +  \cdots + \left(a_{k-1} x^{\frac{q_{k-1}}{p_1p_2 \cdots p_{k-1}}} + \sum_{j=1}^{k'_{k-1}} a_{k-1,j}x^{\frac{q_{k-1,j}}{p_1p_2 \cdots p_{k-1}}} \right), \label{phi-k}
\end{align} 
i.e.\ $\delta_k(x) = p_1 \cdots p_k$ and for each $f \in \cc[x,y]\setminus \{0\}$, 
\begin{align}
\delta_k(f(x,y)) = \delta_k(x) \deg_x\left(f(x,\tilde \phi_k(x,\xi_k)\right). \label{delta-k}
\end{align} 

\begin{prop} \label{essential-prop}
$m = l$, i.e.\ the last essential key form is $f_{l+1}$, where $l+1$ is the number of formal Puiseux pairs of $\tilde \phi_\delta$.
\begin{enumerate}
\item For each $k$, $1 \leq k \leq l+1$, 
\begin{compactenum}
\item \label{p-key} $\alpha_k = p_k$, where $\alpha_k$ is from Property \ref{ess-next-property} of key forms and $p_k$ is from \eqref{tilde-phi-delta}.
\item \label{essential-k} the essential key forms of $\delta_k$ are precisely $f_0, \ldots, f_k$. 
\end{compactenum}
\item \label{omega-k-properties} For all $k$, $1 \leq k \leq l$,
\begin{compactenum}
\item \label{gcd-omega} $\gcd(\omega_0, \ldots, \omega_k) = p_{k+1} \cdots p_{l+1}$.
\item $\delta_k(f_j) = \omega_j/(p_{k+1} \cdots p_{l+1})$ for $0 \leq j \leq k$.
\item \label{unique-omega-combination} For all $n \in \zz$, there are unique integers $\alpha,\beta_1, \ldots, \beta_{k}$ such that 
\begin{compactenum}
\item $0 \leq \beta_j < p_j$ for all $j$, $1 \leq j \leq k$, and
\item $\alpha \omega_0 + \sum_{j=1}^{k} \beta_j \omega_j = np_{k+1} \cdots p_{l+1}$.
\end{compactenum}
\item \label{omega-k-k+1} Let $\omega_{k,j} := \omega_k + q_{j}p_{j+1} \cdots p_{l+1} - q_kp_{k+1} \cdots p_{l+1}$ for $1 \leq k\leq j \leq l$. Then $\omega_{k+1,j} = (p_k-1)\omega_k + \omega_{k,j}$ for all $j$, $k+1 \leq j \leq l+1$. In particular, $\omega_{k+1} = \omega_{k+1,k+1} = p_k\omega_k + (q_{k+1} - q_kp_{k+1})p_{k+2} \cdots p_{l+1}$. 
\end{compactenum}
\item Set $p_0 :=1$. Then for each $k$, $1 \leq k \leq l+1$, 
\begin{compactenum}
\item \label{f-k-y-deg} $f_{k}$ is monic in $y$ with $\deg_y(f_{k}) = p_0p_2 \cdots p_{k-1}$.
\item \label{f-k-substitution} $f_k(x,\tilde \phi_k(x, \xi_k)) = b_k\xi_k x^{\tilde q_k/(p_1\cdots p_k)} + \lot$ for some $b_k \in \cc^*$ and $\tilde q_k \in \zz$ with $\gcd(\tilde q_k,p_k) = 1$.
\item \label{f-phi-k} Define
$$f_{\phi_k} := \prod_{\parbox{1.75cm}{\scriptsize{$\phi_{kj}$ is a con\-ju\-ga\-te of $\phi_k$}}}\mkern-27mu \left(y - \phi_{kj}(x)\right) \in \cc[x,x^{-1},y]$$
Then $f_{\phi_k} = f_k + \pi(\tilde F_k)$ for some $\tilde F_k \in B_{k-1}$ such that $\omega(\tilde F_k) \leq \omega_k$.
\end{compactenum}
\item \label{truncassertion} Let $g_0, \ldots, g_{n+1}$ be the key forms of $\delta$. Fix $n_* \leq n$ and let $\delta_*$ be the semidegree corresponding to the sequence of key forms $g_0, \ldots, g_{n_*+1}$. Then $\delta_*$ has a generic degree-wise Puiseux series of the form 
$$\tilde \phi_{\delta_*}(x,\xi) = \phi_*(x) + \xi x^{r_*},$$
where 
\begin{compactenum}
\item $r_* \leq r_\delta = q_{l+1}/(p_1 \cdots p_{l+1})$, and
\item $\phi_*(x) = [\phi_{l+1}(x)]_{< r_*}$.
\end{compactenum}
That is, $\tilde \phi_{\delta_*}$ is a {\em truncation} of $\tilde \phi_\delta$.
\end{enumerate} 
\end{prop}

\subsection{Degree-like functions and compactifications} \label{degree-like-section}
\begin{defn} \label{degree-like-defn}
Let $X$ be an irreducible affine variety over an algebraically closed field $\kk$. A map $\delta: \kk[X] \setminus \{0\} \to \zz$ is called a {\em degree-like function} if 
\begin{compactenum}
\item \label{deg1} $\delta(f+g) \leq \max\{\delta(f), \delta(g)\}$ for all $f, g \in \kk[X]$, with $<$ in the preceding inequality implying $\delta(f) = \delta(g)$.
\item \label{deg2} $\delta(fg) \leq \delta(f) + \delta(g)$ for all $f, g \in \kk[X]$.
\end{compactenum}
\end{defn}

Every degree-like function $\delta$ on $\kk[X]$ defines an {\em ascending filtration} $\scrF^\delta := \{F^\delta_d\}_{d \geq 0}$ on $\kk[X]$, where $F^\delta_d := \{f \in \kk[X]: \delta(f) \leq d\}$. Define
$$ \kk[X]^\delta := \dsum_{d \geq 0} F^\delta_d, \quad \gr  \kk[X]^\delta := \dsum_{d \geq 0} F^\delta_d/F^\delta_{d-1}.$$

\begin{rem} \label{(f)_d-remark}
For every $f \in \kk[X]$, there are infinitely many `copies' of $f$ in $\kk[X]^\delta$, namely the copy of $f$ in $F^\delta_d$ for {\em each $d \geq \delta(f)$}; we denote the copy of $f$ in $F^\delta_d$ by $(f)_d$. If $t$ is a new indeterminate, then 
$$\kk[X]^\delta \cong \sum_{d \geq 0} F^\delta_d t^d,$$
via the isomorphism $(f)_d \mapsto ft^d$. Note that $t$ corresponds to $(1)_1$ under this isomorphism.
\end{rem}

We say that $\delta$ is {\em finitely-generated} if $\kk[X]^\delta$ is a finitely generated algebra over $\kk$ and that $\delta$ is {\em projective} if in addition $F^\delta_0 = \kk$. The motivation for the terminology comes from the following straightforward

\begin{prop}[{\cite[Proposition 2.5]{sub1}}] \label{basic-prop}
If $\delta$ is a projective degree-like function, then $\xdelta := \proj \kk[X]^\delta$ is a projective compactification of $X$. The {\em hypersurface at infinity} $\xdelta_\infty := \xdelta \setminus X$ is the zero set of the $\qq$-Cartier divisor defined by $(1)_1$ and is isomorphic to $\proj \gr \kk[X]^\delta$. Conversely, if $\bar X$ is any projective compactification of $X$ such that $\bar X \setminus X$ is the support of an effective ample divisor, then there is a projective degree-like function $\delta$ on $\kk[X]$ such that $\xdelta \cong \bar X$.
\end{prop}

\begin{defn} \label{sub-defn}
A degree-like function $\delta$ is called a {\em semidegree} if it always satisfies property \ref{deg2} with an equality. A semidegree is the negative of a {\em discrete valuation}. A {\em divisorial semidegree} is the negative of a divisorial valuation.
\end{defn}

\section{Proof of the main results} \label{sec-proof}
In this section we give proofs of Theorems \ref{key-thm-geometric}, \ref{key-thm-effective}, \ref{semigroup-prop} and Proposition \ref{effective-answer-Puiseux-1} assuming Proposition \ref{key-prop} below. 

\begin{defn}
Let $X := \cc^2$ with coordinates $(x,y)$. Let $\phi(x)$ be a degree-wise Puiseux series in $x$ and $C \subseteq X$ be an analytic curve. We say that $(x,\phi(x))$ is a {\em parametrization of a branch of $C$ at infinity} iff there is a branch of $C$ with a parametrization of the form $t \mapsto (t, \phi(t))$ for $|t| \gg 0$.
\end{defn}

Let $\bar X$ be a normal analytic compactification of $X$ with $C_\infty := \bar X \setminus X$ irreducible and let $\delta$ be the semidegree on $\cc(x,y)$ corresponding to $C_\infty$. Let $\tilde \phi_\delta(x,\xi)$ be the generic degree-wise Puiseux series for $\delta$. The following proposition holds the key for the proofs of our main results. 

\begin{prop} \label{key-prop}
Let $\delta$ be as in the preceding paragraph and let $g_0, \ldots, g_{n+1}$ be the key forms associated to $\delta$.
\begin{enumerate}
\item \label{easy-assertion} Let $f_0, \ldots, f_{l+1}$ be the {\em essential} key forms of $\delta$. If all the key forms are polynomials, then $\bar X$ is isomorphic to the closure of the image of $X$ in the weighted projective variety $\pp^{l+2}(1,\delta(f_0), \ldots, \delta(f_{l+1}))$ under the mapping $(x,y) \mapsto [1:f_0:\cdots:f_{l+1}]$. 
\item \label{one-place-assertion} If $g_{n+1}$ is a polynomial then $C_{n+1} := V(g_{n+1}) \subseteq X$ is a curve with one place at infinity and its unique branch at infinity has a parametrization of the form \eqref{dpuiseux-param} (from Proposition \ref{prop-param}).
\item \label{hard-assertion} If there exists $k$, $0 \leq k \leq n+1$, such that $g_k$ is {\em not} a polynomial, then there does not exist any polynomial $f \in \cc[x,y]$ such that {\em every} branch of $V(f) \subseteq X$ at infinity has a parametrization of the form \eqref{dpuiseux-param}.
\end{enumerate}
\end{prop}

\begin{proof}[Proof of Theorem \ref{key-thm-effective}]
Consider the set up of Theorem \ref{key-thm-effective}. Let $X := \pp^2 \setminus L \cong \cc^2$ and $\bar X$ be the normal analytic surface arising from the contraction of $\tilde E$. Then $\bar X$ is a normal analytic compactification of $X$ and identities \eqref{form-from-pol-0} and \eqref{form-from-pol} show that $g_0, \ldots, g_{n+1}$ of Theorem \ref{key-thm-effective} are precisely the key forms of the semidegree $\delta$ corresponding to the curve at infinity on $\bar X$. The first part of Theorem \ref{key-thm-effective} therefore translates into the following:

{
\renewcommand{\thethm}{$\ref{key-thm-effective}^*$}
\begin{thm} \label{key-thm-effective'}
$\bar X$ is algebraic iff all the key forms associated to $\delta$ are {\em polynomials} iff the {\em last} key form associated to $\delta$ is a polynomial.
\end{thm}
\addtocounter{thm}{-1}
}

\begin{proof}[Proof of Theorem \ref{key-thm-effective'}]
Note that assertions \ref{one-place-assertion} and \ref{hard-assertion} of Proposition \ref{key-prop} imply that the {\em last} key form of $\delta$ is a polynomial iff {\em all} the key forms of $\delta$ are polynomials. Moreover, assertion \ref{easy-assertion} shows that the latter (and hence both) of the equivalent properties of the preceding sentence imply that $\bar X$ is algebraic. Therefore it only remains to show that if $\bar X$ is algebraic then all the key forms of $\delta$ are polynomials. So assume that $\bar X$ is algebraic. Let $\bar X^0 \cong \pp^2$ be the compactification of $X$ induced by the map $(x,y) \mapsto [1:x:y]$, $\sigma: \bar X \dashrightarrow \bar X^0$ be the natural bimeromorphic map, and $S$ (resp.\ $S'$) be the finite set of points of indeterminacy of $\sigma$ (resp.\ $\sigma^{-1})$. We have two cases to consider:

\paragraph{Case 1: $\sigma(C_\infty \setminus S)$ is dense in $L_\infty := \bar X^0 \setminus X$.} In this case it follows from basic geometry of bimeromorphic maps that $\sigma$ must be an isomorphism. In particular, this implies that $\delta$ is precisely the usual degree in $(x,y)$-coordinates. Example \ref{deg-dpuiseux-key} then implies that the key forms of $\delta$ are polynomials.

\paragraph{Case 2: $\sigma(C_\infty \setminus S)$ is a point $O \in L_\infty$.} In this case we are in the situation of Proposition \ref{prop-param}. In particular, $\sigma^{-1}(L_\infty\setminus S')$ is a point $P_\infty \in C_\infty$. Since $\bar X$ is algebraic, it follows that there is an algebraic curve $C \subseteq X$ such that the closure of $C$ in $\bar X$ does {\em not} intersect $P_\infty$. Proposition \ref{prop-param} then implies that {\em every} branch of $C$ at infinity has a parametrization of the form \eqref{dpuiseux-param}. Then assertion \ref{hard-assertion} of Proposition \ref{key-prop} implies that all the key forms of $\delta$ are polynomials, as required to complete the proof of Theorem \ref{key-thm-effective'}
\end{proof}
It remains to prove the last assertion of Theorem \ref{key-thm-effective}. Assume $\tilde E$ is algebraically contractible. Then Theorem \ref{key-thm-effective'} then implies that $g_{n+1}$ is a polynomial. Assertion \ref{one-place-assertion} of Proposition \ref{key-prop} and Proposition \ref{prop-param} then imply that $\tilde C := V(g_{n+1}) \subseteq \tilde Y$ satisfies the requirement of Assertion \ref{existence-assertion-one-place} of Theorem \ref{key-thm-geometric}. This completes the proof of Theorem \ref{key-thm-effective}.
\end{proof}

\begin{proof}[Proof of Theorem \ref{key-thm-geometric}]
At first we show that \eqref{algebraic-assertion} $\im$ \eqref{existence-assertion-1} $\im$ \eqref{existence-assertion-2}. As in Figure \ref{fig:geom-answer}, let $\tilde \pi: \tilde Y \to Y$ be the contraction of $\tilde E$ and let $T := \tilde \pi(\tilde E)$. Since $T$ is a finite set of points, it is clear that if $Y$ is algebraic then there is a (compact) algebraic curve $\tilde C \subseteq Y \setminus T$. This shows that \eqref{algebraic-assertion} $\im$ \eqref{existence-assertion-1}. It is obvious that \eqref{existence-assertion-1} $\im$ \eqref{existence-assertion-2}.\\

We now show that \eqref{existence-assertion-2} $\im$ \eqref{algebraic-assertion}. As in the proof of Theorem \ref{key-thm-effective}, let $X := \pp^2 \setminus L \cong \cc^2$, $(x,y)$ be a system of coordinates on $X$ and $\delta$ be the semidegree on $\cc(x,y)$ corresponding to $E^*$. Recall from Remark \ref{simplest-remark} that the contraction of $\tilde E \setminus \tilde L$ produces a projective compactification $\bar X'$ of $X$ with {\em rational singularities}. Let $L'$ (resp.\ $E'^*$) be the image of $\tilde L$ (resp.\ $E^*$) on $\bar X'$. Let $\tilde C$ be as in Assertion \ref{existence-assertion-2} and $C'$ be the image of $\tilde C$ on $\bar X'$. Since $\bar X'$ has only rational singularities, it follows that $C'$, $L'$ and $E'^*$ are $\qq$-Cartier divisors. Pick $f \in \cc[x,y]$ such that $\tilde C \cap X = V(f)$. Then $C'$ is linearly equivalent (as a $\qq$-Cartier divisor) to $\deg(f)L' + \delta(f)E'^*$. Note that $\delta(f) > 0$, since $\delta$ is the order of pole of $f$ along the curve at infinity on $Y$ (where $Y$ is as in the preceding paragraph). Since $\deg(f)$ is also positive, a theorem of Zariski-Fujita \cite[Remark 2.1.32]{lazativityI} implies that the line-bundle $\sheaf_{\bar X'}(mC')$ is base-point free for some $m \geq 1$. Let $\bar X''$ be the image of the morphism defined by sections of $\sheaf_{\bar X'}(mC')$. Since $C'$ does not intersect $L'$, it follows that $L'$ maps to a point in $\bar X''$. It is then straightforward to see that $\bar X''$ is precisely the contraction of $\tilde E$ from $\tilde Y$. It follows that $\tilde E$ is algebraically contractible, as required for Assertion \ref{algebraic-assertion}.\\

Since the implication \eqref{existence-assertion-one-place} $\im$ \eqref{existence-assertion-2} is obvious and the implication \eqref{algebraic-assertion} $\im$ \eqref{existence-assertion-one-place} is a consequence of the last assertion of Theorem \ref{key-thm-effective}, the proof of Theorem \ref{key-thm-geometric} is complete.
\end{proof}

\begin{proof} [Proof of Proposition \ref{effective-answer-Puiseux-1}]
Recall that $L = \{u = 0\}$. Let the Puiseux expansion for $C$ at $O := (0,0)$ be
$$v = a_0u^{\tilde q/p} + a_1 u^{(\tilde q+1)/p} + \cdots$$
Let $\tilde f$ be as in Proposition \ref{effective-answer-Puiseux-1}. Then it is straightforward to see that 
\begin{align*} 
\tilde f = \begin{cases}
			0 & \text{if}\ r = 0, \\
			\text{a monic polynomial in $v$ of degree $p$} & \text{otherwise.}
		   \end{cases}
\end{align*}
Let $\nu$ be the divisorial valuation on $\cc(u,v)$ corresponding to $E^*_{L,C,r}$ (i.e.\ the {\em last} exceptional divisor in the set up of Question \ref{contract-2-question}). Example \ref{general-example} shows that the generic Puiseux series corresponding to $\nu$ is 
\begin{align} \label{tilde-phi-nu-alg}
\tilde \phi_\nu(u,\xi) = \begin{cases}
							\xi u^{\tilde q/p} & \text{if}\ r = 0, \\
							a_0u^{\tilde q/p} + \cdots + a_{r-1}u^{(\tilde q+r-1)/p} + \xi u^{(\tilde q+r)/p} & \text{otherwise.}
						   \end{cases}
\end{align}
If $r = 0$, then the key polynomials for $\nu$ are $\tilde g_0 = u$ and $\tilde g_1 = v$. For $r \geq 1$, the sequence continues with $\tilde g_2 =  v^p - a_0^pu^{\tilde q}$ and so on, with 
$$\tilde g_j = \tilde g_{j-1} - \text{a monomial term in $u,v$}\quad \text{for $j \geq 3$.}$$
It then follows from the construction of $\tilde f$ and the defining properties (and uniqueness) of key polynomials that $\tilde f$ is precisely the {\em last key polynomial} $\tilde g_{n+1}$ of $\nu$.\\

Now identify $X := \pp^2 \setminus L$ with $\cc^2$ with coordinates $(x,y) := (1/u,v/u)$. Then $\tilde E_{L,C,r}$ is algebraically contractible iff the compactification $\bar X$ of $X$ corresponding to the semidegree $\delta :=-\nu$ is algebraic. The proposition follows from combining Theorem \ref{key-thm-effective'} with the following observations:
\begin{compactenum}
\item If $r = 0$, then $\bar X$ is the weighted projective space $\pp^2(1,p,p-q)$ (Example \ref{deg-dpuiseux-key}).
\item the last key form $g_{n+1}$ of $\delta$ is a polynomial iff $\deg_{(u,v)}(\tilde g_{n+1}) \leq \deg_v(\tilde g_{n+1})$ (follows from \eqref{form-from-pol}). \qed
\end{compactenum}
\noqed
\end{proof}

\begin{proof}[Proof of Theorem \ref{semigroup-prop}]
We use the notations of Theorem \ref{semigroup-prop} and Question \ref{contract-2-question}. Set 
\begin{align*}
q_k := \begin{cases}
				p_1 \cdots p_k - \tilde q_k &\text{for}\ 1 \leq k \leq l,\\
				p_1 \cdots p_{\tilde l} - \tilde q_{\tilde l} - r &\text{for}\ k =  l + 1.
			  \end{cases}
\end{align*}
Consider a generic degree-wise Puisuex series of the form
\begin{align*}
\tilde \phi_{\vec a}(x,\xi) 
	&:= \left(a_1 x^{\frac{q_1}{p_1}} + \sum_{j=1}^{k'_1} a_{1j}x^{\frac{q_{1j}}{p_1}} \right)  +  \left(a_2 x^{\frac{q_2}{p_1p_2}} + \sum_{j=1}^{k'_2} a_{2j}x^{\frac{q_{2j}}{p_1p_2}} \right)  \notag\\
	& \mbox{\phantom{$:=\qquad$}} + \cdots + \left(a_{l} x^{\frac{q_{l}}{p_1p_2 \cdots p_{l}}} + \sum_{j=1}^{k'_{l}} a_{l,j}x^{\frac{q_{l,j}}{p_1p_2 \cdots p_{l}}} \right)  +  \xi x^{\frac{q_{l + 1}}{p_1p_2 \cdots p_{l+1}}} 
\end{align*}
where $a_1, \ldots, a_{l} \in \cc^*$ and $a_{ij}$'s belong to $\cc$. As in the proof of Proposition \ref{effective-answer-Puiseux-1}, identify $X := \pp^2 \setminus L$ with $\cc^2$ with coordinates $(x,y) := (1/u,v/u)$. The assumption in Theorem \ref{semigroup-prop} that $\tilde E_{L,C,r}$ is contractible for every curve $C$ with Puiseux pairs $(\tilde q_1, p_1), \ldots, (\tilde q_{\tilde l},p_{\tilde l})$ is equivalent to saying that for all choices of $a_i$'s and $a_{ij}$'s, the semidegree $\delta_{\vec{a}}$ corresponding to $\tilde \phi_{\vec{a}}$ is the pole along the curve at infinity on some normal analytic compactification $\bar X_{\vec{a}}$ of $X$ with one {\em irreducible} curve at infinity. The statements of Theorem \ref{semigroup-prop} then translate into the following statements:
\begin{enumerate}
\let\oldenumi\theenumi
\renewcommand{\theenumi}{$\oldenumi'$}
\item \label{algebraic-existence'} There exist $a_i$'s and $a_{ij}$'s such that $\bar X_{\vec{a}}$ is algebraic, iff the semigroup condition \eqref{semigroup-criterion-1} holds for all $k$, $1 \leq k \leq  l$.
\item \label{non-algebraic-existence'} There exist $a_i$'s and $a_{ij}$'s such that $\bar X_{\vec{a}}$ is {\em not} algebraic iff either \eqref{semigroup-criterion-1} or \eqref{semigroup-criterion-2} {\em fails} for some $k$, $1 \leq k \leq l$.
\end{enumerate} 

At first we prove Assertion \ref{algebraic-existence'}. It follows from a straightforward application of Proposition \ref{essential-prop} that each $\delta_{\vec a}$ has precisely $l + 2$ {\em essential key forms} $f^{\vec a}_0, \ldots, f^{\vec a}_{l+1}$ and 
\begin{align}
\omega^{\vec a}_k := \delta_{\vec a}(f^{\vec a}_k) = \omega_k,\quad 0 \leq k \leq l+1, \label{omega-a-k}
\end{align}
where $\omega_k$'s are as in conditions \eqref{semigroup-criterion-1} and \eqref{semigroup-criterion-2}. To see the $(\im)$ direction of Assertion \ref{algebraic-existence'}, pick $\vec a$ such that $\bar X_{\vec{a}}$ is algebraic. Theorem \ref{key-thm-effective'} implies that the $f_k$'s are polynomial for all $k$, $0 \leq k \leq l+1$. Algorithm \ref{essential-algorithm} and identity \eqref{omega-a-k} then imply that the semigroup condition \eqref{semigroup-criterion-1} holds for all $k$, $1 \leq k \leq l$. For the $(\Leftarrow)$ implication of Statement \ref{algebraic-existence'}, consider $\vec a_0$ corresponding to the choice $a_1 = \cdots = a_{\tilde l} = 1$ and $a_{ij} = 0$ for all $i,j$, i.e.\  
\begin{align}
\tilde \phi_{\vec a_0}(x,\xi) 
	&:= x^{\frac{q_1}{p_1}} + x^{\frac{q_2}{p_1p_2}} + \cdots + x^{\frac{q_{l}}{p_1p_2 \cdots p_{l}}} +  \xi x^{\frac{q_{l + 1}}{p_1p_2 \cdots p_{l+1}}}. \label{delta-a-0}
\end{align}
If condition \eqref{semigroup-criterion-1} holds for all $k$, $1 \leq k \leq l$, then it can be seen by a straightforward application of Algorithm \ref{essential-algorithm} that $f^{\vec a_0}_k$ is a polynomial for all $k$, $0 \leq k \leq l+1$. Since $f^{\vec a_0}_{l+1}$ is the {\em last} key polynomial of $\delta_{\vec a_0}$, Theorem \ref{key-thm-effective'} then implies that $\bar X_{\vec a_0}$ is algebraic, as required to prove the $(\Leftarrow)$ implication of Statement \ref{algebraic-existence'}.\\

Now we prove the $(\Leftarrow)$ implication of Statement \ref{non-algebraic-existence'}. If \eqref{semigroup-criterion-1} fails for some $k$, $1 \leq k \leq l$, take smallest such $k$. Then applying Algorithm \ref{essential-algorithm} for $\delta_{\vec a_0}$ with $\vec a_0$ as in \eqref{delta-a-0} shows that $f^{\vec a_0}_{k+1}$ is {\em not} a polynomial, so that $\bar X_{\vec a_0}$ is not algebraic (Theorem \ref{key-thm-effective'}). Now assume that \eqref{semigroup-criterion-1} holds for all $k$, $1 \leq k \leq l$, but there exists $k$, $1 \leq k \leq l$ such that \eqref{semigroup-criterion-2} fails; w.l.o.g.\ assume that $k$ is the smallest integer such that \eqref{semigroup-criterion-2} fails. Let $\tilde \omega$ be the largest element in $(\omega_{k+1}, p_k \omega_k) \cap \zz\langle \omega_0, \ldots, \omega_k \rangle \setminus \zz_{\geq 0}\langle \omega_0, \ldots, \omega_k \rangle$. Let $r := p_k\omega_k  - \tilde \omega$ and consider 
\begin{align*}
\tilde \phi_{\vec a_1}(x,\xi) 
	&:= \left( x^{\frac{q_1}{p_1}} + \cdots + x^{\frac{q_k}{p_1\cdots p_k}} \right) + x^{\frac{q_k-r}{p_1\cdots p_k}} + \left( x^{\frac{q_{k + 1}}{p_1p_2 \cdots p_{k+1}}} + \cdots + x^{\frac{q_{l}}{p_1p_2 \cdots p_{l}}} \right) +  \xi x^{\frac{q_{l + 1}}{p_1p_2 \cdots p_{l+1}}}.
\end{align*}
Then Algorithm \ref{essential-algorithm} shows that $f^{\vec a_1}_{k+1}$ is {\em not} a polynomial. It follows that $\bar X_{\vec a_1}$ is not algebraic (Theorem \ref{key-thm-effective'}), which completes the proof of $(\Leftarrow)$ implication of Statement \ref{non-algebraic-existence'}. \\

Finally, a straightforward examination of Algorithm \ref{essential-algorithm} shows that if both \eqref{semigroup-criterion-1} and \eqref{semigroup-criterion-2} holds for all $k$, $1 \leq k \leq l$, then $f_k$ is polynomial for each $k$, $0 \leq k \leq l+1$. This proves the $(\im)$ implication of Statement \ref{non-algebraic-existence'} and completes the proof of Theorem \ref{semigroup-prop}.
\end{proof}

\section{Proof of Proposition \ref{key-prop}} \label{sec-key-proof}
In Section \ref{sec-star} we define some operations on degree-wise Puiseux series that we use in the proof of Proposition \ref{key-prop}. In Section \ref{lifting-section} we prepare the tools to compare `lifts' in the rings $B_k$ of Section \ref{essential-section} of polynomials in $\cc[x,y]$. More precisely, let $\dpsxc$ be the field of degree-wise Puiseux series in $x$, $B_k := \cc[x,x^{-1}, y_1, \ldots, y_k]$ and $\pi: B_k \to \cc[x,x^{-1},y]$ be as in Section \ref{essential-section}. Given $f \in \dpsxc[y]$, Lemma \ref{F-Phi} (which follows from a result of \cite{abhyankar-expansion}) gives a `canonical' element $F^\pi_f \in \tilde B_k := \dpsxc[y_1, \ldots, y_k] \supseteq B_k$ such that $\pi(F^\pi_f) = f$. Given certain polynomials $f, g \in \cc[x,y]$ which are {\em close} (in the sense that their degree-wise Puiseux expansions agree up to certain exponent), in the proof of Proposition \ref{key-prop} we need to compare how `close' $F^\pi_f$ and $F^\pi_g$ are; this estimate is provided by Lemma \ref{limit-of-agreement}. Lemma \ref{poly-positive-prop} of Section \ref{ordered-section} determines the generators of the graded ring $\cc[x,y]^\delta$ (from Section \ref{degree-like-section}) associated to the compactification of $\cc^2$ corresponding to a semidegree $\delta$ and a Gr\"obner basis of the ideal of the hypersurface at infinity under the assumption that all key forms of $\delta$ are polynomials and their $\delta$-values are positive. Finally in Section \ref{key-proof-subsection} we prove Proposition \ref{key-prop} based on Lemmas \ref{limit-of-agreement} and \ref{poly-positive-prop}. We prove the latter lemmas in Section \ref{sec-technical-proofs}.

\subsection{Some operations on degree-wise Puiseux series} \label{sec-star}
\begin{defn} \label{star-defn}
Let $\phi = \sum_j a_j x^{q_j/p} \in \dpsxc$ be a degree-wise Puiseux series with polydromy order $p$ and $r$ be a multiple of $p$. Then for all $c \in \cc$ we define 
$$c \star_r \phi := \sum_j a_j c^{q_jr/p}x^{q_j/p}.$$  
Let $y_1, \ldots, y_k$ be indeterminates and $\Phi = \sum_{\alpha \in \zz_{\geq 0}^k} \phi_\alpha(x)y_1^{\alpha_1} \cdots y_k^{\alpha_k} \in \tilde B_k := \dpsxc[y_1, \ldots, y_k]$. The {\em polydromy order} of $\Phi$ is the lowest common multiple of the polydromy orders of all the non-zero $\phi_\alpha$'s. Let $r$ be a multiple of the polydromy order $\Phi$. Then we define
$$c \star_r \Phi := \sum_\alpha \left(c \star_r \phi_\alpha\right) y_1^{\alpha_1} \cdots y_k^{\alpha_k}.$$   
\end{defn}

\begin{rem}
It is straightforward to see that in the case that $c$ is an $r$-th root of unity (and $r$ is a multiple of the polydromy order of $\phi$), $c \star_r \phi$ is a {\em conjugate} of $\phi$ (cf.\ Remark-Notation \ref{f-phi}).
\end{rem}

\begin{lemma} \label{star-lemma}
\mbox{}
\begin{enumerate}
\item \label{different-stars} Let $p$ be the polydromy order of $\Phi \in \tilde B_k$, $d$ and $e$ be positive integers, and $c \in \cc$. Then $c \star_{pde} \Phi = c^e \star_{pd} \Phi = c^{de} \star_p \Phi$.
\item \label{star-sum-product} Let $\Phi_j = \sum_{j} \phi_{j,\alpha}(x)y_1^{\alpha_1} \cdots y_k^{\alpha_k} \in \tilde B_k$ for $j = 1,2$, and $r$ be a multiple of the polydromy order of each non-zero $\phi_{j,\alpha}$. Then $c \star_r \left(\Phi_1 + \Phi_2\right) = \left(c \star_r \Phi_1\right) + \left(c \star_r \Phi_2\right)$ and $c \star_r \left(\Phi_1\Phi_2\right) = \left(c \star_r \Phi_1\right)\left(c \star_r \Phi_2\right)$. 
\item \label{star-projection} Let $\pi : \tilde B_k \to \dpsxc[y]$ be a $\cc$-algebra homomorphism that sends $x \mapsto x$ and $y_j \mapsto f_j \in \cc[x, x^{-1}, y]$ for $1 \leq j \leq k$. Let $\Phi = \sum_{\alpha} \phi_\alpha(x)y_1^{\alpha_1} \cdots y_k^{\alpha_k} \in \tilde B_k$, $r$ be a multiple of the polydromy order of each non-zero $\phi_\alpha$, and $\mu$ be a (not necessarily primitive) $r$-th root of unity. Then $\pi (\mu \star_r \Phi) = \mu \star_r \pi(\Phi)$.
\end{enumerate}
\end{lemma}  

\begin{proof}
Assertions \ref{different-stars} and \ref{star-sum-product} are immediate from the definitions, here we prove Assertion \ref{star-projection}.  Let $\Phi$, $r$ and $\mu$ be as in Assertion \ref{star-projection}. Then 
\begin{align*}
\mu \star_r \pi(\Phi) 
	&=  \mu \star_r \left( \sum_{\alpha} \phi_\alpha(x)f_1^{\alpha_1} \cdots f_k^{\alpha_k} \right) \\
	&= \sum_\alpha \left(\mu \star_r \phi_\alpha\right) (\mu \star_r f_1)^{\alpha_1} \cdots (\mu \star_r f_k)^{\alpha_k}\quad \text{(due to Assertion \ref{star-sum-product})}\\
	&= \sum_\alpha \left(\mu \star_r \phi_\alpha\right) f_1^{\alpha_1} \cdots f_k^{\alpha_k} \quad \text{(since $\mu^{rn} = 1$ for all $n \in \zz$)}\\
	&= \pi(\mu \star_r \Phi) \qed
\end{align*}
\noqed
\end{proof}

\begin{remtation} \label{f-phi}
If $\phi$ is a degree-wise Puiseux series in $x$ with polydromy order $p$, then we write
$$f_\phi := \prod_{\parbox{1.5cm}{\scriptsize{$\phi_{j}$ is a con\-ju\-ga\-te of $\phi$}}}\mkern-18mu \left(y - \phi_{j}(x)\right) = \prod_{j=0}^{p-1} \left(y - \zeta^j \star_p \phi(x)\right),$$
where $\zeta$ is a primitive $p$-th root of unity. If $f \in \cc[x,y]$, then its degree-wise Puiseux factorization (Theorem \ref{dpuiseux-factorization}) can be described as follows: there are unique (up to conjugacy) degree-wise Puiseux series $\phi_1, \ldots, \phi_k$, a unique non-negative integer $m$, and $c \in \cc^*$ such that
$$f = cx^m \prod_{i=1}^k f_{\phi_i}$$
Let $(q_1, p_1), \ldots, (q_l, p_l)$ be Puiseux pairs of $\phi$. Set $p_0 :=1$. For each $k$, $0 \leq k \leq l$, we write 
$$f_\phi^{(k)} := \prod_{j=0}^{p_0p_1 \cdots p_k-1} \left(y - \zeta^j \star_{p} \phi(x)\right),$$
where $\zeta$ is a primitive $(p_1 \cdots p_l)$-th root of unity. Note that $f_\phi^{(l)}  = f_\phi$, and for each $m,n$, $0 \leq m < n \leq l$, 
\begin{align}
f_\phi^{(n)} 
			&= \prod_{j=0}^{p_0p_1 \cdots p_{n}-1} \left(y - \zeta^j \star_{p} \phi(x)\right) 
			= \prod_{i=0}^{p_{m+1}\cdots p_n-1} \prod_{j=0}^{p_0p_1 \cdots p_m-1} \left(y - \zeta^{ip_0p_1 \cdots p_m +j} \star_{p} \phi(x)\right) \notag \\
			&= \prod_{i=0}^{p_{m+1}\cdots p_n-1} \zeta^{ip_0p_1 \cdots p_m} \star_{p} \left(\prod_{j=0}^{p_0p_1 \cdots p_m-1} \left(y - \zeta^{j} \star_{p} \phi(x)\right)\right) \notag\\
			&= \prod_{i=0}^{p_{m+1}\cdots p_n-1} \zeta^{ip_0p_1 \cdots p_m} \star_{p} \left(f_\phi^{(m)}\right). \label{f-phi-n-m}
\end{align}
\end{remtation}

\begin{convention}
In the following sections we frequently deal with maps $\cc[y_1, \ldots, y_k] \to \cc[x,x^{-1},y]$, where $y_1, \ldots, y_k$ are indeterminates. We always (unless there is a misprint!) use upper-case letters $F,G, \ldots$ for elements in $\cc[y_0, \ldots, y_k]$ and corresponding lower-case letters $f,g, \ldots$ for their images in $\cc[x,x^{-1},y]$. 
\end{convention}

\subsection{Comparing `canonical' pre-images of polynomials} \label{lifting-section}

\begin{lemma} \label{F-Phi}
Let $y_1, \ldots, y_k$ be indeterminates, $p_0 := 1, p_1, \ldots, p_{k-1}$ be positive integers, and $\pi:B_k := \cc[x,x^{-1},y_1, \ldots, y_k] \to \cc[x,x^{-1},y]$ be a ring homomorphism which sends $x \mapsto x$ and $y_j \mapsto f_j$, where $f_j$ is monic in $y$ of degree $p_0 \cdots p_{j-1}$, $1 \leq j \leq k$. Let $\tilde B_k := \dpsxc[y_1, \ldots, y_k] \supseteq B_k$. Then $\pi$ induces a homomorphism $\tilde B_k \to \dpsxc[y]$ which we also denote by $\pi$. If $f$ is a non-zero element in $\dpsxc[y]$, then there is a unique $F^\pi_f \in \tilde B_k$ such that 
\begin{enumerate}
\item $\pi\left(F^\pi_f \right) = f$ and
\item $\deg_{y_j}(F^\pi_f) < p_{j}$ for all $j$, $1 \leq j \leq k-1$.
\end{enumerate} 
Moreover, if $f$ is monic in $y$ of degree $p_1 \cdots p_{k-1}p_k$ for some integer $p_k$, then 
\begin{enumerate}
\addtocounter{enumi}{2}
\item \label{F-Phi-monic} $F^{\pi}_f$ is monic in $y_k$ of degree $p_k$. 
\item \label{F-Phi-exponents} If the coefficient of $x^\alpha y_1^{\beta_1}\cdots y_k^{\beta_k}$ in $F^\pi_f - y_k^{p_k}$ is non-zero, then $\sum_{i=1}^j p_0 \cdots p_{i-1}\beta_i < p_1 \cdots p_j$ for all $j$, $1 \leq j \leq k$.
\end{enumerate}
\end{lemma}

\begin{proof}
This follows from an immediate application of \cite[Theorem 2.13]{abhyankar-expansion}.
\end{proof}

Now assume $\delta$ is the semidegree on $\cc[x,y]$ with the degree-wise Puiseux series $\tilde \phi_\delta$ as in \eqref{tilde-phi-delta}. Fix $k$, $0 \leq k \leq l$. Let $\pi:B \to \cc[x,x^{-1},y]$ and $B_k \subseteq B$ be as in Section \ref{essential-section} (defined immediately following Corollary \ref{essential-corollary}). Let $\pi_k := \pi|_{B_k}$ and $\tilde B_k := \dpsxc[y_1, \ldots, y_k] \supseteq B_k$. Let $\omega$ be the weighted degree on $B$ from Section \ref{essential-section}. Note that $\omega$ induces a weighted degree on $\tilde B_k$ which we also denote by $\omega$.

\begin{remtation} \label{equivalent-remark}
If $\phi$ and $\psi$ are two degree-wise Puiseux series in $x$ and $r \in \qq$, we write $\phi \equiv_r \psi$ iff $\deg(\phi - \psi) \leq r$. Let $\phi_{k+1}$ be as in \eqref{phi-k} and set $r_{k+1} := q_{k+1}/(p_1 \cdots p_{k+1})$. We write $F_{\phi_{k+1}}$ for $F_{f_{\phi_{k+1}}}$, where $f_{\phi_{k+1}}$ is as in Remark-Notation \ref{f-phi}. Let $\psi \in \dpsxc$ be such that $\psi \equiv_{r_{k+1}} \phi_{k+1}$. Then it follows in particular that the first $k$ Puiseux pairs of $\psi$ are $(q_1, p_1), \ldots, (q_k, p_k)$ (of course $\psi$ may have more Puiseux pairs), so that $f_{\psi}^{(k)}$ is well defined (see Remark-Notation \ref{f-phi}). In this case we define 
\begin{align*}
F_{\psi}^{(k)} := \begin{cases}
					F^{\pi_1}_{f_{\psi}^{(0)}} \in \tilde B_1 & \text{for}\ k = 0,\\
					F^{\pi_k}_{f_{\psi}^{(k)}} \in \tilde B_k & \text{for}\ 1 \leq k \leq l.
				  \end{cases} 
\end{align*}
(We needed to treat the case $k=0$ separately in the definition of $F_{\psi}^{(k)}$, since $f_{\psi}^{(0)} = y - \psi(x)$ is an element of $\tilde B_1$, whereas for $k \geq 1$, $f_{\psi}^{(k)}$ is an element of $\tilde B_k$.)
\end{remtation}

\begin{lemma}  \label{limit-of-agreement}
Recall that $\omega_k := \omega(y_k)$, $1 \leq k \leq l+1$. Fix $k$, $0 \leq k \leq l$. Let $F_{k+1} \in B_k$ be as in \eqref{F-k} and $r_{k+1}$ be as in Remark-Notation \ref{equivalent-remark}.
\begin{enumerate} 
\item \label{F-k+1-agreement} $F_1 = F_{\phi_1} = y_1$. For $k \geq 1$, $F_{k+1}$ is precisely the sum of all monomial terms $H$ (in $x,y_1, \ldots, y_k$) of $F_{\phi_{k+1}}$ such that $\omega(H) > \omega_{k+1}$.
\item \label{F-phi-psi-agreement} Let $\psi \in \dpsxc$ satisfy $\psi \equiv_{r_{k+1}} \phi_{k+1}$. Then 
$$\omega\left(F_{\psi}^{(k)} - F_{\phi_{k+1}}\right) \leq \omega_{k+1}.$$
\end{enumerate}
\end{lemma}

\subsection{Determining $\cc[x,y]^{\delta}$ when key forms of $\delta$ are polynomial} \label{ordered-section}
We continue with the notations of Section \ref{lifting-section}. Fix $k$, $1 \leq k \leq l+1$. In this section we assume the following conditions are satisfied:
\begin{enumerate}
\renewcommand{\theenumi}{$\text{Positivity}_k$}
\item $\delta_k(f_j) > 0$ for all $j$, $0 \leq j \leq k$. \label{positivity-cond}
\renewcommand{\theenumi}{$\text{Polynomial}_k$}
\item $f_j \in \cc[x,y]$ for all $j$, $0 \leq j \leq k$. \label{polynomial-cond}
\end{enumerate} 

\begin{rem} \label{polytivity-remark}
A straightforward examination of Algorithm \ref{essential-algorithm} shows that \eqref{polynomial-cond} implies ($\text{Positivity}_{k-1}$).
\end{rem}

Let $A_k := \cc[x,y_1, \ldots, y_k] \subseteq B_k$. We now define a monomial ordering on $A_k$. Let $v_k := (\omega_0, \ldots, \omega_k)$ and $\prec_k$ be the ordering on $\zz_{\geq 0}^{k+1}$ by setting $(\beta_0, \ldots, \beta_k) \prec_k (\beta'_0, \ldots, \beta'_k)$ iff 
\begin{enumerate}
\item $(\beta_0 - \beta'_0,\ldots, \beta_k - \beta'_k) \cdot v_k < 0$, or  
\item $(\beta_0 - \beta'_0, \ldots, \beta_k - \beta'_k) \cdot v_k = 0$ and the right-most non-zero entry of $(\beta_0 - \beta'_0, \ldots, \beta_k - \beta'_k)$ is negative.
\end{enumerate}
It follows from the definition that $\prec_k$ is a total ordering on $\zz_{\geq 0}^{k+1}$. \eqref{positivity-cond} implies that $\prec_k$ is in fact a well order which is compatible with addition on $\zz_{\geq 0}^{k+1}$, and therefore it induces a monomial ordering on $A_k$, which we also denote by $\prec_k$. Let $z$ be a new indeterminate and $\tilde A_k := A_k[z]$. Extend $\omega$ to a weighted degree on $\tilde A_k$ by defining $\omega(z) := 1$. Recall that the essential key forms of $\delta_k$ are $f_0, \ldots, f_k$. Let $F_1, \ldots, F_k$ be as in Property \ref{F-k} and for each $j$, $1 \leq j \leq k$, let $\tilde F_j$ be the {\em homogenization} of $F_j$ with respect to $z$, i.e.\
\begin{align} \label{tilde-F-j}
\tilde F_{j+1} &:= \begin{cases}
			y_1 & \text{if}\ j=0,\\
			y_{j}^{p_{j}} - \sum_{i=0}^{j_n}c_{j,i} x^{\beta^i_{j,0}} y_1^{\beta^i_{j,1}} \cdots y_{j}^{\beta^i_{j,j}}z^{p_j\omega_j - \sum_{i'} \beta^i_{j,i'}\omega_{i'}}, 
				& \text{for}\ 1 \leq j \leq k-1,
			\end{cases} 
\end{align}
where $c_{j,i}$ and $\beta^i_{j,i'}$'s are as in Property \ref{F-k}. Finally, for $2 \leq j \leq k$, let $H_j$ be the {\em leading form} of $\tilde F_j$ with respect to $\omega$, i.e.
\begin{align}
H_{j+1} := y_{j}^{p_j} - c_{j,0} x^{\beta^0_{j,0}} y_1^{\beta^0_{j,1}} \cdots y_{j-1}^{\beta^0_{j,j-1}},\quad 1 \leq j \leq k-1.
\end{align} 

\begin{lemma} \label{poly-positive-prop}
Assume \eqref{positivity-cond} and \eqref{polynomial-cond} hold. Then 
\begin{enumerate}
\item \label{wp-assertion} Let $\tilde J_k$ be the ideal in $\tilde A_k$ generated by $y_{j+1}z^{p_j\omega_j-\omega_{j+1}} - \tilde F_{j+1}$, $1 \leq j \leq k-1$. Then $\cc[x,y]^{\delta_k} \cong \tilde A_k/\tilde J_k$. 
\item \label{grobner-assertion} Let $J_k$ be the ideal in $A_k$ generated by the {\em leading weighted homogeneous forms} (with respect to $\omega$) of polynomials $F \in A_k$ such that $\delta_k(\pi_k(F)) < \omega(F)$. Then $\scrB_k :=(H_{k}, \ldots, H_2)$ is a Gr\"obner basis of $J_k$ with respect to $\prec_k$.
\end{enumerate}
\end{lemma}

\subsection{Proof of Proposition \ref{key-prop}} \label{key-proof-subsection}
We continue to use the notation of the preceding sections. 
\begin{proof}[Proof of Assertion \ref{easy-assertion} of Proposition \ref{key-prop}]
Let $\WP$ be the weighted projective space $\pp^{l+2}(1,\omega_0$, $\ldots$, $\omega_{l+1})$ with weighted homogeneous coordinates $[z:x:y_1: \cdots: y_{l+1}]$. Then Assertion \ref{easy-assertion} of Proposition \ref{key-prop} is equivalent to showing that $\proj \left(\cc[x,y]^\delta \right) \cong \proj \left(\tilde A_{l+1}/\tilde J_{l+1}\right)$, where $\tilde A_{l+1}$ and $\tilde J_{l+1}$ are as in Lemma \ref{poly-positive-prop}. Assertion \ref{wp-assertion} of Lemma \ref{poly-positive-prop} (applied with $k := l+1$) implies the latter statement and completes the proof.
\end{proof}

\begin{proof}[Proof of Assertion \ref{one-place-assertion} of Proposition \ref{key-prop}]
Note that the last key polynomial of $\delta$ is also the last {\em essential} key polynomial $f_{l+1}$. Applying Assertion \ref{f-k-substitution} of Proposition \ref{essential-prop} with $k = l+1$ implies that there exists at least one degree-wise Puiseux root $\psi(x)$ of $f_{l+1}$ such that \begin{align}
\psi(x) \equiv_{q_{l+1}/(p_1 \cdots p_{l+1})} \phi_{l+1}(x), \label{last-equivalence}
\end{align} 
which in turn implies that the first $l$ Puiseux pairs of $\psi$ are $(q_1, p_1), \ldots, (q_l,p_l)$ (cf.\ Remark-Notation \ref{equivalent-remark}). Since $f_{l+1}$ is monic in $y$ of degree $p_1\cdots p_l$ (Assertion \ref{f-k-y-deg} of Proposition \ref{essential-prop}), it follows that up to conjugacy $\psi(x)$ is the only degree-wise Puiseux root of $f_{l+1}$. The claim now follows from identity \eqref{last-equivalence}.
\end{proof}

\begin{proof}[Proof of Assertion \ref{hard-assertion} of Proposition \ref{key-prop}]
Assertion \ref{truncassertion} of Proposition \ref{essential-prop} implies that it suffices to prove the following special case:
\begin{align} \label{special3}
\parbox{.82\textwidth}{Assume $g_k$ is a polynomial for each $k$, $0 \leq k \leq n$, but $g_{n+1}$ is {\em not} a polynomial. Then there does not exist any polynomial $f \in \cc[x,y]$ such that {\em every} branch of $V(f) \subseteq X$ at infinity has a parametrization of the form \eqref{dpuiseux-param}.}
\end{align}
We prove \eqref{special3} by contradiction. So assume the assumptions of \eqref{special3} hold, but there exists a polynomial $f \in \cc[x,y]$ such that each of its branches at infinity has a parametrization of the form 
\begin{align*}
t \mapsto (t, \tilde \phi_\delta(t,\xi)|_{\xi = c} + \lot)\quad \text{for}\ |t| \gg 0 
\end{align*}
for some $c \in \cc$ depending on the branch. Then it follows that $f$ has a factorization of the form
\begin{align}
f &= a\prod_{k=1}^m f_{\psi_k},\quad \text{where $a \in \cc^*$ and} \label{f-decomposition} \\
\psi_k &\equiv_{r_{l+1}} \phi_{l+1},\quad \text{for each}\ k,\ 1 \leq k \leq m, \label{psi-n-equivalence}
\end{align}
see Remark-Notations \ref{f-phi} and \ref{equivalent-remark} to recall the notations. W.l.o.g.\ we may (and will) assume that $a = 1$. \\

Recall that $f_{l+1} = g_{n+1}$, where $f_{l+1}$ is the last essential key form. At first we claim that $l \geq 1$. Indeed, otherwise $f_1 = g_{n+1}$ is not a polynomial, and therefore the construction of $f_1$ from Algorithm \ref{essential-algorithm} shows that $\tilde \phi_\delta$ has the following form: 
\begin{align*}
\tilde \phi_\delta(x,\xi) = h(x) + bx^{-s} + \lot
\end{align*}
where $h(x) \in \cc[x]$, $b \in \cc^*$, $s$ is a positive integer, and $\lot$ denotes terms in which the exponents of $x$ are smaller than $-s$. But then $\delta(y - h(x)) < 0$, which is impossible (since every polynomial has a pole at infinity on $\bar X$). Therefore $l \geq 1$, as required.\\

Since $f_j = \pi(F_j)$, it follows by our assumptions that $F_j$ is a polynomial for all $j$, $1 \leq j \leq l$, but $F_{l+1}$ is {\em not} a polynomial. Fix $k$, $1 \leq k \leq m$. Then identity \eqref{psi-n-equivalence} and Proposition \ref{limit-of-agreement} imply that 
\begin{align}
F_{\psi_k}^{(l)} = F_{l+1} + \tilde F_k
\end{align}
where $\tilde F_k \in \tilde B_l := \dpsxc[y_0, \ldots, y_l]$ and $\omega(\tilde F_k) \leq \omega_{l+1}$. Let $s_k$ denote the polydromy order of $\psi_k$ and $\mu_k$ be a primitive $s_k$-th root of unity. Identity \eqref{psi-n-equivalence} implies that $t_k := s_k/(p_1p_2 \cdots p_{l})$ is a positive integer. Therefore identity \eqref{f-phi-n-m} and Assertion \ref{star-projection} of Lemma \ref{star-lemma} imply that 
\begin{align}
f_{\psi_k}  
	&= \prod_{j=0}^{t_k-1} \mu_k^{jp_1 \cdots p_l} \star_{s_k} \left(f_{\psi_k}^{(l)}\right)
	= \prod_{j=0}^{t_k-1} \mu_k^{jp_1 \cdots p_l} \star_{s_k} \left(\pi_l \left(F_{l+1} + \tilde F_k\right) \right) 
	= \pi_l(G_k)\quad \text{where} \label{f-psi-n} \\
G_k &:= \prod_{j=0}^{t_k-1} \left(F_{l+1} + \mu_k^{jp_1 \cdots p_l} \star_{s_k}\left(\tilde F_k \right) \right) \in \tilde B_l. \label{G-n-defn}
\end{align}
Recall that $F_{l+1} = y_{l}^{p_{l}} - c_{l,0}x^{\beta^0_{l,0}} y_1^{\beta^0_{l,1}} \cdots y_{l-1}^{\beta^0_{l,l-1}} - \sum_{i=1}^{l_n}c_{l,i} x^{\beta^i_{l,0}} y_1^{\beta^i_{l,1}} \cdots y_{l}^{\beta^i_{l,l}}$. By assumption there exists $i'$ such that $\beta^{i'}_{l,0} < 0$; choose smallest such $i'$ and set 
$$\omega_{l+1,i'} := \omega\left(x^{\beta^{i'}_{l,0}} y_1^{\beta^{i'}_{l,1}} \cdots y_{l}^{\beta{i'}_{l,l}}\right) = \sum_{i=0}^l \beta^{i'}_{l,i}\omega_i.$$
Then $\omega_{l+1,i'} > \omega_{l+1}$ and therefore we may express $G_k$ as
\begin{align}
G_k	&= \prod_{j=0}^{t_k - 1} \left(y_{l}^{p_{l}} - c_{l,0}x^{\beta^0_{l,0}} y_1^{\beta^0_{l,1}} \cdots y_{l-1}^{\beta^0_{l,l-1}} - \sum_{i=1}^{i'}c_{l,i} x^{\beta^i_{l,0}} y_1^{\beta^i_{l,1}} \cdots y_{l}^{\beta^i_{l,l}} - G_{n,j} \right) , \label{G-n}
\end{align} 
for some $G_{n,j} \in \tilde B_l$ with $\omega(G_{n, j}) < \omega_{l+1,i'}$. Identities \eqref{f-decomposition}, \eqref{f-psi-n} and \eqref{G-n} imply that $f = \pi_l(F)$ for some $F \in \tilde B_l$ of the form
\begin{align}
F = \prod_{N=1}^M \left(y_{l}^{p_{l}} - c_{l,0}x^{\beta^0_{l,0}} y_1^{\beta^0_{l,1}} \cdots y_{l-1}^{\beta^0_{l,l-1}} - \sum_{i=1}^{i'}c_{l,i} x^{\beta^i_{l,0}} y_1^{\beta^i_{l,1}} \cdots y_{l}^{\beta^i_{l,l}} - \tilde G_{N} \right),
\end{align}
where $\omega(\tilde G_{N}) < \omega_{l+1,i'}$ for all $N$, $1 \leq N \leq M$. Let 
\begin{align*}
G := \begin{cases}
		F - y_l^{Mp_{l}}	
			& \text{if}\ i'=0,\\
		F - \left(y_l^{p_{l}} - c_{l,0}x^{\beta^0_{l,0}} y_1^{\beta^0_{l,1}} \cdots y_{l-1}^{\beta^0_{l,l-1}} - \sum_{i=1}^{i'-1}c_{l,i} x^{\beta^i_{l,0}} y_1^{\beta^i_{l,1}} \cdots y_{l}^{\beta^i_{l,l}} \right)^M 
			& \text{otherwise.}
		\end{cases}
\end{align*}
Then the {\em leading weighted homogeneous form} of $ G$ with respect to $\omega$ is
\begin{align} \label{ld-omega}
\ld_{\omega}(G) = \begin{cases}
			- c_{l,0}x^{\beta^0_{l,0}} y_1^{\beta^0_{l,1}} \cdots y_{l-1}^{\beta^0_{l,l-1}} & \text{if}\ i' = 0\ \text{and}\ M = 1, \\
			\sum_{i=1}^M {M \choose i} (-c_{l,0})^i y_l^{(M-i)p_{l}} x^{i\beta^0_{l,0}} y_1^{i\beta^0_{l,1}} \cdots y_{l-1}^{i\beta^0_{l,l-1}} & \text{if}\ i' = 0, \text{and}\ M > 1,\\
			Mc_{l,i'}\left(y_l^{p_{l}} - c_{l,0}x^{\beta^0_{l,0}} y_1^{\beta^0_{l,1}} \cdots y_{l-1}^{\beta^0_{l,l-1}}\right)^{M-1} x^{\beta^{i'}_{l,0}} y_1^{\beta^{i'}_{l,1}} \cdots y_{l}^{\beta^{i'}_{l,l}} & \text{otherwise.}
			\end{cases}
\end{align}
Since $\pi_l(F) = f \in \cc[x,y]$, it follows that $g:= \pi_l(G)$ is also a polynomial in $x$ and $y$. Assertion \ref{wp-assertion} of Proposition \ref{poly-positive-prop} then implies that there is a polynomial $\tilde G \in A_l := \cc[x,y_1, \ldots, y_l]$ such that $\pi_l(\tilde G) = g$ and $\omega(\tilde G) = \delta_l(g)$. In particular, $\omega(\tilde G) \leq \omega(G)$.

\begin{claim} \label{omega-and-tilde}
$\omega(\tilde G) = \omega(G)$.
\end{claim}

\begin{proof}
Let $\prec_l$ be the monomial ordering on $A_l$ from Section \ref{ordered-section} and let $\alpha$ be the smallest positive integer such that $x^\alpha \ld_{\omega}(G) \in A_l$, then \eqref{ld-omega} implies that the {\em leading term} of $x^\alpha \ld_{\omega}(G)$ with respect to $\prec_l$ is

\begin{align} \label{lt-G}
\lt_{\prec_l}\left(x^\alpha\ld_{\omega}\left(G \right)\right) 
			= \begin{cases}
			- c_{l,0} y_1^{\beta^0_{l,1}} \cdots y_{l-1}^{\beta^0_{l,l-1}} 
				& \text{if}\ i' = 0\ \text{and}\ M = 1, \\
			- c_{l,0} M y_l^{(M-1)p_{l}} x^{(M-1)|\beta^0_{l,0}|} y_1^{\beta^0_{l,1}} \cdots y_{l-1}^{\beta^0_{l,l-1}} 
				& \text{if}\ i' = 0, \text{and}\ M > 1,\\
			Mc_{l,i'}y_l^{(M-1)p_{l}+ \beta^{i'}_{l,l}} y_1^{\beta^{i'}_{l,1}} \cdots y_{l-1}^{\beta^{i'}_{l,l-1}} & \text{otherwise.}
			  \end{cases}
\end{align}
Assume contrary to the claim that $\omega(G) > \omega(\tilde G) = \delta_l(g)$. Then $x^\alpha \ld_\omega(G) \in J_l$, where $J_l$ is the ideal from Assertion \ref{grobner-assertion} of Proposition \ref{poly-positive-prop}. Assertion \ref{grobner-assertion} of Proposition \ref{poly-positive-prop} then implies that there exists $j$, $1 \leq j \leq l-1$, such that $y_{j}^{p_{j}} = \lt_{\prec_l}(H_{j+1})$ divides $\lt_{\prec_l}(x^\alpha\ld_{\omega}(G))$. But this contradicts the fact that $\beta^{i'}_{l,j'} < p_{j'}$ for all $j'$, $1 \leq j' \leq l-1$ (Property \ref{exponent-bound} of essential key forms) and completes the proof of the claim.
\end{proof}

Let $J_l$ and $\alpha$ be as in the proof of Claim \ref{omega-and-tilde}. Note that $\ld_{\omega}(x^\alpha \tilde G) \not\in J_l$ by our choice of $\tilde G$. Therefore, after `dividing out' $\tilde G$ by the Gr\"obner basis $\scrB_l$ of Proposition \ref{poly-positive-prop} (which does not change $\omega(\tilde G)$) if necessary, we may (and will) assume that 
\begin{align}
\parbox{.9\textwidth}{$y_{j}^{p_{j}}$ does {\em not} divide any of the monomial terms of $\ld_\omega(x^\alpha \tilde G)$ for any $j$, $1 \leq j \leq l-1$.} \label{tilde-G-division}
\end{align}
Since $\pi_l(x^\alpha G - x^\alpha \tilde G) = 0$, it follows that $\ld_\omega(x^\alpha G - x^\alpha \tilde G) \in J_l$. Since $\omega(G) = \omega(\tilde G)$ by Claim \ref{omega-and-tilde}, it follows that $H^* := \ld_\omega(x^\alpha G) - \ld_\omega(x^\alpha \tilde G) \in J_l$. Let 
\begin{align*}
H := \lt_{\prec_l}\left(\ld_\omega(x^\alpha G)\right)\ \text{and}\ \tilde H := \lt_{\prec_l}\left(\ld_\omega(x^\alpha \tilde G)\right).
\end{align*}
Since $\tilde G \in \cc[x,y_1, \ldots, y_l]$, it follows that $\deg_x(\tilde H) \geq \alpha$. On the other hand, our construction of $G$ shows that $\deg_x(H) < \alpha$. It follows in particular that $H \neq \tilde H$ and 
\begin{align*}
\lt_{\prec_l}(H^*) = \max_{\prec_l}\{H, -\tilde H\}.
\end{align*}
Then \eqref{lt-G} and \eqref{tilde-G-division} imply that $y_{j}^{p_{j}} = \lt_{\prec_l}(H_{j})$ does {\em not} divide $\lt_{\prec_l}(H^*)$ for any $j$, $1 \leq j \leq l-1$. This contradicts Assertion \ref{grobner-assertion} of Proposition \ref{poly-positive-prop} and finishes the proof of Assertion \ref{hard-assertion} of Proposition \ref{key-prop}.
\end{proof}

\section{Proofs of the technical lemmas} \label{sec-technical-proofs}

\subsection{Proof of Lemma \ref{limit-of-agreement}}
Throughout this section we use the notations of Section \ref{essential-section}.

\begin{lemma} \label{essential-omega-lemma}
Fix $k$, $1 \leq k \leq l$ and let $H$ be a non-zero element in $\tilde B_k := \dpsxc[y_1, \ldots,y_k] \subseteq B_k$ be such that $\deg_{y_j}(H) < p_j$ for all $j$, $1 \leq j \leq k$. Then $\delta(\pi(H)) = \omega(H)$.
\end{lemma}

\begin{proof}
Assertion \ref{unique-omega-combination} of Proposition \ref{essential-prop} implies that no two distinct monomials in $H$ have the same $\omega$-value, i.e.\ $H = \sum_{j \geq 1} H_j$ with $\omega(H) = \omega(H_1) > \omega(H_2) > \cdots$. Since $\delta(\pi(y_j)) = \delta(f_j) = \omega_j = \omega(y_j)$, it follows that $\delta(\pi(H_j)) = \omega(H_j)$ for all $j \geq 1$. It follows then (from the definition of degree-like functions) that $\delta(\pi(H)) = \omega(H_1) = \omega(H)$.
\end{proof}

\begin{proof}[Proof of Assertion \ref{F-k+1-agreement} of Lemma \ref{limit-of-agreement}]
The first statement of Assertion \ref{F-k+1-agreement} of Lemma \ref{limit-of-agreement} follows immediately from the definitions. We prove the second statement here. Fix $k$, $1 \leq k \leq l$. Let $\tilde H$ be the sum of all monomial terms $H$ (in $x,y_1, \ldots, y_k$) of $F_{\phi_{k+1}}$ such that $\omega(H) > \omega_{k+1}$, i.e.\ $F_{\phi_{k+1}} = \tilde H + \tilde G$ for some $\tilde G \in \tilde B_k$ with $\omega(\tilde G) \leq \omega_{k+1}$. Recall from Assertion \ref{f-phi-k} of Proposition \ref{essential-prop} that $f_{\phi_{k+1}} = \pi(F_{k+1} + \tilde F_{k+1})$ for some $\tilde F_{k+1} \in B_{k}$ such that $\omega(\tilde F_{k+1}) \leq \omega_{k+1}$. Since $f_{\phi_{k+1}} = \pi(F_{\phi_{k+1}})$, it follows that 
\begin{align*}
\delta \left(\pi(\tilde H - F_{k+1})\right) = \delta\left(\pi(\tilde F_{k+1} - \tilde G) \right) \leq \omega \left (\tilde F_{k+1} - \tilde G\right) \leq \omega_{k+1}.
\end{align*}
Now identity \eqref{F-k} and Lemma \ref{F-Phi} imply that $H:=\tilde H - F_{k+1}$ satisfies the assumption of Lemma \ref{essential-omega-lemma}, which implies that $\omega\left(\tilde H - F_{k+1}\right) \leq \omega_{k+1}$. Since both $\tilde H$ and $F_{k+1}$ consist of monomials with $\omega$-value strictly greater than $\omega_{k+1}$, it follows that $\tilde H = F_{k+1}$, which completes the proof of Assertion \ref{F-k+1-agreement} of Lemma \ref{limit-of-agreement}.
\end{proof}

\begin{notation}
For $F \in \tilde B_k$ and $I \subseteq \rr$, we write $F|_I$ for the sum of all monomial terms $H$ of $F$ such that $\omega(H) \in I$. Moreover, for $w \in \rr$, we write $F|_{> w}$ for $F|_{(w, \infty)}$.
\end{notation}

\begin{lemma} \label{F-phi-psi-inductive-agreement}
Fix $k$, $1 \leq k \leq l$. Let $\psi_1, \psi_2 \in \dpsxc$ and $w \leq \omega_k \in \rr$ be such that
\begin{compactenum}
\item the first $k$ Puiseux pairs of each $\psi_j$ are $(q_1, p_1), \ldots, (q_k,p_k)$.
\item \label{w-assumption} $F_{\psi_1}^{(k-1)}|_{> w}  = F_{\psi_2}^{(k-1)}|_{> w}$ (where $F_{\psi_j}^{(k-1)}$ are defined as in Remark-Notation \ref{equivalent-remark}), and
\item \label{omega-k-assumption} for each $j$, $1 \leq j \leq 2$,
$$\omega\left(F_{\psi_j}^{(k-1)}|_{> \omega_k} - F_k \right) \leq \omega_k.$$  
\end{compactenum}
Then $F_{\psi_1}^{(k)}|_{>(p_{k} -1)\omega_{k} + w} = F_{\psi_2}^{(k)}|_{>(p_{k} -1)\omega_{k} + w}$.
\end{lemma}

\begin{proof}
Let 
\begin{align*}
\tilde B:= \begin{cases}
			\tilde B_1 & \text{if}\ k = 1, \\
			\tilde B_{k-1} & \text{otherwise.}
		   \end{cases}
\end{align*}
Assumptions \ref{w-assumption} and \ref{omega-k-assumption} imply that there exists $G \in \tilde B$ with $\omega(G) \leq \omega_k$ such that for both $j$, $1 \leq j \leq 2$,
$$F_{\psi_j}^{(k-1)} = F_k + G + G_j$$
for some $G_j \in \tilde B$ with $\omega(G_j) \leq w$. Fix $j$, $1 \leq j \leq 2$. Let $m_j$ be the polydromy order of $\psi_j$ and $\mu_j$ be a primitive $m_j$-th root of unity. Then identity \eqref{f-phi-n-m} and Assertion \ref{star-projection} of Lemma \ref{star-lemma} imply that 
\begin{align*}
f_{\psi_j}^{(k)} &= \prod_{i=0}^{p_k-1} \mu_j^{ip_1 \cdots p_{k-1}} \star_{m_j} \left(f_{\psi_j}^{(k-1)}\right) = \pi_{k-1}(G^*_j)\quad\ \text{where}\\
G^*_j 	&:= \prod_{i=0}^{p_k-1} \mu_j^{ip_1 \cdots p_{k-1}} \star_{m_j} \left(F_{\psi_j}^{(k-1)}\right) 
		= \prod_{i=0}^{p_k-1} \mu_j^{ip_1 \cdots p_{k-1}} \star_{m_j} \left(F_k + G + G_j \right) \\
		&= \prod_{i=0}^{p_k-1} \left(F_k + \mu_j^{ip_1 \cdots p_{k-1}} \star_{m_j}G + \mu_j^{ip_1 \cdots p_{k-1}} \star_{m_j}G_j \right)
		= \prod_{i=0}^{p_k-1} \left(F_k + \mu^{ip_1 \cdots p_{k-1}} \star_{m}G + \mu_j^{ip_1 \cdots p_{k-1}} \star_{m_j}G_j \right)
\end{align*}
where $m$ is the {\em polydromy order} of $G$ (Definition \ref{star-defn}) and $\mu$ is a primitive $m$-th root of unity (the last equality is an implication of Assertion \ref{different-stars} of Lemma \ref{star-lemma}). Let 
\begin{align} \label{zeroth-step}
G_{j,0} &= \prod_{i=0}^{p_k-1} \left(y_k + \mu^{ip_1 \cdots p_{k-1}} \star_{m}G + \mu_j^{ip_1 \cdots p_{k-1}} \star_{m_j}G_j \right) \in \tilde B_k.
\end{align}
Note that $\pi_k(G_{j,0}) = f_{\psi_j}^{(k)} = \pi_k(F_{\psi_j}^{(k)})$. 
\begin{proclaim} \label{constructing-F-k}
$F_{\psi_j}^{(k)}$ can be constructed from $G_{j,0}$ in a finite number of steps, where each step consists of replacing (from a suitable collection of monomials) $y_i^{p_i}$ with $y_{i+1} - (F_{i+1} - y_i^{p_i})$ for a suitable $i$, $1 \leq i \leq k-1$.
\end{proclaim}
At first we show how Claim \ref{constructing-F-k} implies the lemma. Indeed, every monomial term in $y_{i+1} - (F_{i+1} - y_i^{p_i})$ has $\omega$-value smaller than or equal to $\omega(y_i^{p_i})$. Therefore Claim \ref{constructing-F-k} implies that all the monomials in $F_{\psi_j}^{(k)}|_{> w^*}$, where $w^* := (p_{k} -1)\omega_{k} + w$, originates (via repeated substitutions of Claim \ref{constructing-F-k}) from the monomial terms on (the expansion of the product of) the right hand side of \eqref{zeroth-step} which have $\omega$-value {\em greater than} $w^*$. It follows that $G_j$ has no effect on $F_{\psi_j}^{(k)}|_{> w^*}$, and therefore $F_{\psi_1}^{(k)}|_{> w^*} = F_{\psi_2}^{(k)}|_{> w^*}$, as required.\\

Now we prove Claim \ref{constructing-F-k}. For $\beta := (\beta_1, \ldots, \beta_{k}) \in \zz_{\geq 0}^{k}$, let us write 
$$|\beta|_{k-1} := \sum_{j=1}^{k-1} p_0 \ldots p_{j-1} \beta_j.$$
Consider the well order $\prec^*_{k-1}$ on $\zz_{\geq 0}^{k}$ defined as follows: $ \beta \prec^*_{k-1} \beta'$ iff 
\begin{enumerate}
\item $|\beta|_{k-1} <  |\beta'|_{k-1}$, or  
\item $|\beta|_{k-1} =  |\beta'|_{k-1}$ and the right-most non-zero entry of $\beta - \beta'$ is negative.
\end{enumerate}
Now we describe the process of constructing $G_{j,n+1}$ assuming we have already constructed $G_{j,n}$ such that $\pi_k(G_{j,n}) = f_{\psi_j}^{(k)}$. Write $G_{j,n}$ in the following form
\begin{align*}
G_{j,n} &= \sum_{\beta \in \zz_{\geq 0}^{k}} g_{j,n,\beta}(x) y_1^{\beta_1}\cdots y_{k}^{\beta_{k}} = \sum_{\beta \in \scrB_{n,0}}g_{j,n,\beta}(x) y_1^{\beta_1}\cdots y_{k}^{\beta_{k}}  + \sum_{\beta \in \scrB_{n,1}}g_{j,n,\beta}(x) y_1^{\beta_1}\cdots y_{k}^{\beta_{k}} ,\ \text{where} \\
\scrB_{n,0} &:= \{\beta \in \zz_{\geq 0}^{k}: g_{j,n,\beta} \neq 0\ \text{and}\ \beta_i < p_i\ \text{for all}\ i, 1 \leq i \leq k-1 \},\\
\scrB_{n,1} &:= \{\beta \in \zz_{\geq 0}^{k}: g_{j,n,\beta} \neq 0\ \text{and}\ \beta_i \geq p_i\ \text{for some}\ i, 1 \leq i \leq k-1 \}.
\end{align*} 
If $\scrB_{n,1} = \emptyset$, then $G_{j,n} = F_{\psi_j}^{(k)}$, and we {\bf stop}. Otherwise let $s_n := \max \{|\beta|_{k-1} :\beta \in \scrB_{n_1}\}$ and construct $G_{j,n+1}$ by applying the following procedure for all $\beta \in \scrB_{n,1}$ such that $|\beta|_{k-1} = s_n$: pick an $i$, $1 \leq i \leq k-1$, such that $\beta_i \geq p_i$ and replace $y_1^{\beta_1}\cdots y_{k}^{\beta_{k}}$ by \begin{align}
H_{j,n,\beta}(x,y_1, \ldots, y_k) := \left(\prod_{1 \leq j \leq k, j \neq i}y_j^{\beta_j} \right) y_i^{\beta_i-p_i}\left(y_{i+1} - (F_{i+1} - y_i^{p_i})\right). \label{sub-j-n-beta}
\end{align}
We exhibit the finiteness of the sequence by showing that $s_n$ decreases after finitely many steps. Indeed, let $H_{j,n,\beta}$ be as in \eqref{sub-j-n-beta}, and let 
$$\beta^* :=  \beta - p_i\vec{u}_i + \vec{u}_{i+1},$$
where $\vec{u}_i$ (resp.\ $\vec{u}_{i+1}$) is the $i$-th (respectively $(i+1)$-th) unit coordinate vector in $\zz^{k}$. Pick $\beta' \in \zz^k$ is such that $y_1^{\beta_1} \cdots y_{k}^{\beta_{k}}$ appears with non-zero coefficient in $H_{j,n,\beta}$. Then 
\begin{compactenum}
\item If $\beta' \neq \beta^*$, then Assertion \ref{F-k+1-agreement} of Lemma \ref{limit-of-agreement} and Assertion \ref{F-Phi-exponents} of Lemma \ref{F-Phi} imply that $|\beta'|_{k-1} < |\beta|_{k-1} = s_n$. 
\item For $\beta' = \beta^*$, there are two possibilities: 
\begin{compactenum}
\item \label{non-dropping-case} $i < k-1$, $|\beta'|_{k-1} = |\beta|_{k-1}$ and $\beta \prec^*_{k-1} \beta'$, or
\item $i = k-1$ and $|\beta'|_{k-1} < |\beta|_{k-1} = s_n$.
\end{compactenum}
\end{compactenum}
It follows that $s_{n+1} \leq s_n$, and the only way $s_{n+1}$ can equal $s_n$ if there is some $\beta \in \scrB_{n,1}$ which satisfies Case \ref{non-dropping-case}. But the definition of $\prec^*_{k-1}$ ensures that the latter scenario can repeat only finitely many times. Consequently $s_{n+1}$ eventually decreases, as required to complete the proof of Claim \ref{constructing-F-k} and Lemma \ref{F-phi-psi-inductive-agreement}. 
\end{proof}

\begin{cor} \label{stronger-phi-psi-agreement}
Fix $k$, $0 \leq k \leq l$. Let $\omega_{j+1, k+1}$, $0 \leq j \leq k$, be as in Assertion \ref{omega-k-k+1} of Proposition \ref{key-prop} and $r_{k+1}$ be as in Remark-Notation \ref{equivalent-remark}. Let $\psi \in \dpsxc$ be such that $\psi \equiv_{r_{k+1}} \phi_{k+1}$. Then for all $j$ such that $0 \leq j \leq k$,
$$\omega\left(F_{\psi}^{(j)} - F_{\phi_{k+1}}^{(j)}\right) \leq \omega_{j+1,  k+1}.$$
\end{cor}

\begin{proof}
We prove the corollary by induction on $j$. For $j = 0$, recall (from Remark-Notation \ref{f-phi}) that $f_{\psi}^{(0)} = y - \psi(x)$ and $f_{\phi_{k+1}}^{(0)} = y - \phi_{k+1}(x)$. Let $\tilde \psi(x) := \psi(x) - \phi_{k+1}(x)$ and $\tilde \phi_{k+1}(x) := \phi_{k+1}(x) - \phi_1(x)$. Since $\pi_1(y_1) = y - \phi_1(x)$, it follows (from Remark-Notation \ref{equivalent-remark}) that 
$$F_{\phi_{k+1}}^{(0)} = y_1 - \tilde \phi_{k+1}(x),\quad \text{and}\quad F_{\psi}^{(0)} = y_1 - \tilde \phi_{k+1}(x) - \tilde \psi(x).$$
Since $\deg_x(\tilde \psi(x)) \leq r_{k+1} = \omega_{1,k+1}$, the corollary follows for $j = 0$.  \\

Now assume that the corollary is true for some $j$, $0 \leq j \leq k-1$. Then by induction hypothesis,
$$F_{\psi}^{(j)} = F_{\phi_{k+1}}^{(j)} + \tilde G_j$$
for some $\tilde G_j \in \tilde B_j$ with $\omega(\tilde G_j) \leq \omega_{j+1,  k+1}$. Assertion \ref{F-k+1-agreement} of Lemma \ref{limit-of-agreement} then implies that $\psi$ and $\phi_{k+1}$ satisfy the hypothesis of Lemma \ref{F-phi-psi-inductive-agreement} with $w := \omega_{j+1, k+1}$ and the $k$ of Lemma \ref{F-phi-psi-inductive-agreement} being $j+1$. Therefore Lemma \ref{F-phi-psi-inductive-agreement} implies that
$$\omega\left(F_{\psi}^{(j+1)} - F_{\phi_{k+1}}^{(j+1)}\right) \leq (p_{j+1}-1)\omega_{j+1} + \omega_{j+1,  k+1} = \omega_{j+2,k+1},$$
as required to complete the induction, and therefore the proof of the corollary.
\end{proof}

\begin{proof}[Proof of Assertion \ref{F-phi-psi-agreement} of Lemma \ref{limit-of-agreement}]
Since $\omega_{k+1} = \omega_{k+1,k+1}$ and $F_{\phi_{k+1}}^{(k)} = F_{\phi_{k+1}}$, Assertion \ref{F-phi-psi-agreement} of Lemma \ref{limit-of-agreement} is simply a special case of Corollary \ref{stronger-phi-psi-agreement}.
\end{proof}

\subsection{Proof of Lemma \ref{poly-positive-prop}} \label{positive-section}

In this section we freely use the notations of Section \ref{ordered-section}. For each $k$, $1 \leq k \leq l+1$, let $A_k := \cc[x,y_1, \ldots, y_k]$. \eqref{polynomial-cond} implies that $H_j \in A_{k-1}$ for $2 \leq j \leq k \leq l+1$. For each $k$, $1 \leq k \leq l+1$, let $S_k \subseteq \zz$ be the semigroup generated by $\omega_0, \ldots, \omega_k$; recall that \eqref{polynomial-cond} implies that $S_{k-1} \subseteq \zz_{\geq 0}$ (Remark \ref{polytivity-remark}).

\begin{lemma} \label{semi-isomorphism}
Fix $k$, $1 \leq k \leq l+1$. Assume \eqref{polynomial-cond} holds. For each $j$, $2 \leq j \leq k$, let $\bar J_{j}$ be the ideal in $A_{j-1}$ generated by $H_2, \ldots, H_{j}$. Then for each $j$, $1 \leq j \leq k-1$, there is an isomorphism 
$$A_j/\bar J_{j+1} \cong \cc[S_j] \cong \cc[t^{\omega_0}, \ldots, t^{\omega_j}],$$
where $t$ is an indeterminate; the isomorphism maps $x \mapsto t^{\omega_0}$ and $y_i \mapsto b_{j,i} t^{\omega_i}$, $1 \leq i \leq j$, for some $b_{j,1}, \ldots, b_{j,j} \in \cc^*$. 
\end{lemma}

\begin{proof}
We prove Lemma \ref{semi-isomorphism} by induction on $j$. For $j=1$ note that
\begin{align*}
A_1/\bar J_2 = \cc[x,y_1]/\langle y_{1}^{p_1} - c_{1,0} x^{q_1} \rangle \cong \cc[t^{p_1}, t^{q_1}],
\end{align*}
where $t$ is an indeterminate and the isomorphism maps $x \mapsto t^{p_1}$ and $y_1 \mapsto c_{1,0}^{1/p_1}t^{q_1}$, where $c_{1,0}^{1/p_1}$ is a $p_1$-th root of $c_{1,0} \in \cc^*$. Since $\omega_0 = p_1p_2 \cdots p_l$ and $\omega_1 = q_1p_2 \cdots p_l$, this proves the lemma for $j = 1$. \\

So assume that the lemma is true for $j-1$, $2 \leq j \leq k$, i.e.\ there exists an isomorphism 
\begin{align*}
A_{j-1}/\bar J_{j} \cong \cc[t^{\omega_0},\ldots, t^{\omega_{j-1}}],
\end{align*}
which maps $x \mapsto t^{\omega_0}$ and $y_i \mapsto b_{j-1,i}t^{\omega_i}$, $1 \leq i \leq j-1$ for some $b_{j-1,1}, \cdots, b_{j-1,j-1} \in \cc^*$. It follows that 
\begin{align*}
A_{j}/\bar J_{j+1} 
	&= A_{j-1}[y_{j}]/\langle \bar J_{j}, y_{j}^{p_j} - c_{j,0} x^{\beta^0_{j,0}} y_1^{\beta^0_{j,1}} \cdots y_{j-1}^{\beta^0_{j,j-1}} \rangle 
	\cong \cc[t^{\omega_0},\ldots, t^{\omega_{j-1}},y_j]/ \langle y_{j}^{p_j} - \tilde c  t^{p_j\omega_j} \rangle
\end{align*}
for some $\tilde c \in \cc^*$ (the last isomorphism uses properties \ref{decreasing-omega} of key forms and Assertion \ref{p-key} of Proposition \ref{essential-prop}). Since $p_j = \min\{\alpha \in \zz_{> 0}; \alpha\omega_j \in \zz \omega_0 + \cdots + \zz \omega_{j-1}\}$ (due to Assertion \ref{p-key} of Proposition \ref{essential-prop}), it follows that 
\begin{align*}
\cc[t^{\omega_0},\ldots, t^{\omega_{j-1}},y_j]/ \langle y_{j}^{p_j} - \tilde c  t^{p_j\omega_j} \rangle \cong \cc[t^{\omega_0},\ldots, t^{\omega_{j}}]
\end{align*}
via a map which sends $y_j \to (\tilde c)^{1/p_j}t^{\omega_j}$ ($(\tilde c)^{1/p_j}$ being a $p_j$-th root of $\tilde c$), as required to complete the induction and prove the lemma.
\end{proof}

\begin{cor} \label{J-korollary}
Fix $k$, $1 \leq k \leq l+1$. Let $J_k$ be the ideal in $A_k$ from Assertion \ref{grobner-assertion} of Lemma \ref{poly-positive-prop}. If \eqref{polynomial-cond} holds, then $J_k = \bar J_kA_k$, where $\bar J_k$ is as in Lemma \ref{semi-isomorphism}.
\end{cor}

\begin{proof}
Since $f_0, \ldots, f_k$ are essential key forms of $\delta_k$ (Assertion \ref{essential-k} of Proposition \ref{essential-prop}), it follows from the definitions of $J_k$ and $\bar J_k$ and Property \ref{decreasing-omega} of essential key forms that $\bar J_k \subseteq J_k$. We prove $\bar J_k = J_k$ by contradiction, so assume that $\bar J_k \subsetneq J_k$. Lemma \ref{semi-isomorphism} implies that there is an isomorphism
\begin{align*}
\bar \chi_k: A_k/\bar J_k \cong R_k := \cc[t^{\omega_0}, \ldots, t^{\omega_{k-1}}, y_k],
\end{align*}
where $\bar \chi_k(x) = t^{\omega_0}$ and $\bar \chi_k(y_j) = c_jt^{\omega_j}$ for some $c_j \in \cc^*$ for all $j$, $1 \leq j \leq k-1$. Observe that 
\begin{enumerate}[(i)]
\item \label{graded-observation} $\bar \chi_k$ is an isomorphism of graded $\cc$-algebras, where $A_k$ is graded by $\omega$ and $R_k \subseteq \cc[t,y_k]$ is graded by the weighted degree corresponding to weights $1$ for $t$ and $\omega_k$ for $y_k$. 
\item $J_k$ is (by definition) a homogeneous ideal of $A_k$ with respect to the grading by $\omega$.
\item $J_k$ is a prime ideal of $A_k$ (since $\delta_k$ is a {\em semidegree}).
\item $\cc[t,y_k]$ is integral over $R_k$.
\end{enumerate}
Since $\bar J_k \subsetneq J_k$, the above observations imply that $\bar \chi_k(J_k) = \ppp \cap R_k$ for a non-zero prime ideal $\ppp$ of $\cc[t,y_k]$. Moreover, $\ppp$ is homogeneous with respect to the grading on $\cc[t,y_k]$ from observation \eqref{graded-observation}, which implies that one of the following scenarios must occur:
\begin{enumerate}
\item $\ppp = \langle t, y_k \rangle$ or $\ppp = \langle t \rangle$, so that $t^{\omega_0} \in \bar \chi_k(J_k)$, and consequently $x \in J_k$.
\item $\ppp = y_k$, so that $y_k \in \bar \chi_k(J_k)$, and consequently $y_k \in J_k$.
\item $\ppp = \langle y_k - ct^{\omega_k} \rangle$ for some $c \in \cc^*$, so that $y_k^{m} - \tilde c t^{n} \in \bar \chi_k(J_k)$ for some $\tilde c \in \cc^*$ and positive integers $m, n$, and consequently $y_k^{m} - \tilde c x^{\beta_0}y_1^{\beta_1} \cdots y_{k-1}^{\beta_{k-1}} \in J_k$ for some $\beta_0, \ldots, \beta_{k-1} \in \zz_{\geq 0}$.
\item \label{last-scenario} $\ppp = \langle t^{-\omega_k} y_k - c \rangle$ for some $c \in \cc^*$, so that $t^n y_k^{m} - \tilde c \in \bar \chi_k(J_k)$ for some $\tilde c \in \cc^*$ and positive integers $m, n$, and consequently $x^{\beta_0}y_k^{m} - \tilde c y_1^{\beta_1} \cdots y_{k-1}^{\beta_{k-1}} \in J_k$ for some $\beta_0, \ldots, \beta_{k-1} \in \zz_{\geq 0}$.
\end{enumerate}
Note that the last scenario can occur only if $\omega_k \leq 0$. Now, $x \not\in J_k$ and $y_k \not\in J_k$ by definition of $J_k$, and the last two scenarios cannot occur, since $f_k$ is the {\em last} key form of $\delta_k$. In particular, each of the scenarios lead to a contradiction. This completes the proof of the corollary. 
\end{proof}

\begin{cor} \label{minimum-cor}
Fix $k$, $1 \leq k \leq l+1$. If \eqref{polynomial-cond} holds, then 
\begin{align}\label{delta-k-omega}
\delta_k(f) = \min\{\omega(F): F \in A_k,\ \pi_k(F) = f\}\ \text{for all}\ f \in \cc[x,y]\setminus\{0\}.
\end{align}
\end{cor}

\begin{rem} \label{pijective-remark}
\eqref{polynomial-cond} in particular implies that $f_1 = y - h(x)$ for some polynomial $h \in \cc[x,y]$; it follows in particular that $\pi_k:A_k \to \cc[x,y]$ is surjective, and therefore the number on the right hand side of \eqref{delta-k-omega} is well defined for all $f \in \cc[x,y]\setminus\{0\}$.
\end{rem}

\begin{proof}

Let $\delta'_k: \cc[x,y]\setminus\{0\} \to \zz$ be defined by the formula on the right hand side of \eqref{delta-k-omega}. Since $\delta_k(x) = \omega(x)$ and $\delta_k(f_j) = \omega(y_j)$ for all $j$, $1 \leq j \leq k$, it immediately follows (from the definition of degree-like functions) that $\delta_k \leq \delta'_k$. We prove $\delta_k = \delta'_k$ by contradiction. So assume there exists $f \in \cc[x,y]\setminus\{0\}$ such that $\delta_k(f) < \delta'_k(f)$. By definition of $\delta'_k$, there exists $F \in A_k$ such that $\pi(F) = f$ and $\omega(F) = \delta'_k(f)$. Then the leading weighted homogeneous form (with respect to $\omega$) $\ld_\omega(F)$ of $F$ belongs to $J_k$, and therefore $\ld_\omega(F) = \sum_{j=2}^k G_jH_j$ for some polynomials $G_2, \ldots, G_k \in A_k$ which are weighted homogeneous with respect to $\omega$. Let $F' := F - \sum_{j=2}^k G_j(F_j - y_j)$. Since $\omega(F_j) > \omega_j = \omega(y_j)$ for $2 \leq j \leq k$, it follows that $\omega(F') < \omega(F)$. But then $\pi_k(F') = f$ and $\omega(F') < \delta'_k(f)$, which contradicts the definition of $\delta'_k$. This completes the proof of the corollary.
 \end{proof}

\begin{proof}[Proof of Assertion \ref{wp-assertion} of Lemma \ref{poly-positive-prop}]
Consider the map $\tilde \pi_k: \tilde A_k \to \cc[x,y]^{\delta_k}$ which is defined as follows: for $\tilde F \in \tilde A_k$, express $\tilde F$ as $\tilde F = \tilde F_1 + \cdots + \tilde F_m$, where $\tilde F_j$'s are weighted homogeneous with respect to $\omega$, and set
\begin{align}\label{tilde-pi-k}
\tilde \pi_k(\tilde F) := \sum_{j=1}^m \left(\pi_k\left(\left.\tilde F_j\right|_{z=1}\right)\right)_{\omega(\tilde F_j)},
\end{align}
where on the right hand side we used the notation from Remark \ref{(f)_d-remark}; note that the right hand side of is well defined since for all $j$, $1 \leq j \leq m$,
$$\omega(\tilde F_j) 
		\geq \omega\left(\left.\tilde F_j\right|_{z=1}\right) 
		\geq \delta_k\left(\pi_k\left(\left. \tilde F_j\right|_{z=1} \right)\right).$$
It is straightforward to see that $\tilde \pi_k$ is a graded $\cc$-algebra homomorphism, where $\tilde A_k$ is graded by $\omega$, and the grading on $\cc[x,y]^{\delta_k}$ is the natural one induced by $\delta_k$. Corollary \ref{minimum-cor} implies that $\tilde \pi_k$ is surjective. We now show that $\ker \tilde \pi_k = \tilde J_k$. Indeed, it follows (from the definition of $\tilde J_k$) that $\ker \tilde \pi_k \supseteq \tilde J_k $. The inclusion in the opposite direction we prove by contradiction. So assume that $\ker \tilde \pi_k \supsetneq \tilde J_k$. Let $\tilde \omega := \min\{\omega(\tilde G): \tilde G \in \ker \tilde \pi_k\setminus \tilde J_k\}$ (\eqref{positivity-cond} ensures that $\tilde \omega$ is a positive integer), and $\tilde F \in \ker \tilde \pi_k\setminus \tilde J_k$ such that $\omega(\tilde F) = \tilde \omega$. Then $\ld_\omega(\tilde F) \in J_k$, so that Corollary \ref{J-korollary} implies that $\ld_\omega(\tilde F) = \sum_{j=2}^k G_jH_j$ for some polynomials $G_2, \ldots, G_k \in A_k$ which are weighted homogeneous with respect to $\omega$. Let 
$$\tilde F' := \tilde F - \sum_{j=2}^k G_j\left(y_{j}z^{p_{j-1}\omega_{j-1}-\omega_{j}} - \tilde F_{j}\right).$$
It follows that $\tilde F' \in \ker \tilde \pi_k\setminus \tilde J_k$ and $\omega(\tilde F') < \omega(\tilde F) = \tilde \omega$, which contradicts the minimality of $\tilde \omega$. It follows that $\ker \tilde \pi_k = \tilde J_k$, as required to complete the proof of Assertion \ref{wp-assertion} of Lemma \ref{poly-positive-prop}.
\end{proof}

\begin{proof}[Proof of Assertion \ref{grobner-assertion} of Lemma \ref{poly-positive-prop}]
Corollary \ref{J-korollary} shows that $\scrB_k :=(H_{k}, \ldots, H_2)$ generates $J_k$. Therefore, to show that $\scrB_k$ is a Gr\"obner basis of $J_k$ with respect to $\prec_k$, it suffices to show that running one step of {\em Buchberger's algorithm} with $\scrB_k$ as input leaves $\scrB_k$ unchanged. We follow Buchberger's algorithm as described in \cite[Section 2.7]{littlesheacox}, which consists of performing the following steps for each pair of $H_i, H_j \in \scrB_k$, $2 \leq i < j \leq k$:

\paragraph{Step 1: Compute the {\em S-polynomial} $S(H_i, H_j)$ of $H_i$ and $H_j$.} The leading terms of $H_i$ and $H_j$ with respect to $\prec_k$ are
$$\lt_{\prec_k}(H_i) = y_{i-1}^{p_{i-1}}\ \text{and}\ \lt_{\prec_k}(H_j) = y_{j-1}^{p_{j-1}},$$
and therefore, the {\em S-polynomial} of $H_i$ and $H_j$ is
\begin{align*}
S(H_i, H_j) &:= y_{j-1}^{p_{j-1}}H_i - y_{i-1}^{p_{i-1}}H_j  \\
			&= -\left(c_{i-1,0}x^{\beta^0_{i-1,0}} y_1^{\beta^0_{i-1,1}} \cdots y_{i-2}^{\beta^0_{i-1,i-2}}\right)y_{j-1}^{p_{j-1}} 
				+ \left(c_{j-1,0}x^{\beta^0_{j-1,0}} y_1^{\beta^0_{j-1,1}} \cdots y_{j-2}^{\beta^0_{j-1,j-2}}\right)y_{i-1}^{p_{i-1}}.
\end{align*}

\paragraph{Step 2: Divide $S(H_i, H_j)$ by $\scrB_k$ and if the remainder is non-zero, then adjoin it to $\scrB_k$.} Since $i < j$, the leading term of $S(H_i,H_j)$ is 
\begin{align*}
\lt_{\prec_k}\left(S(H_i, H_j)\right) &= -\left(c_{i-1,0}x^{\beta^0_{i-1,0}} y_1^{\beta^0_{i-1,1}} \cdots y_{i-2}^{\beta^0_{i-1,i-2}}\right)y_{j-1}^{p_{j-1}}.
\end{align*}
Since $\beta^0_{i-1,j'} < p_{j'}$ for all $j'$, $1 \leq j' \leq i-1$ (Property \ref{exponent-bound} of key forms), it follows that $H_j$ is the only element of $\scrB_k$ such that $\lt_{\prec_k}(H_j)$ divides $\lt_{\prec_k}\left(S(H_i, H_j)\right)$. The remainder of the division of $S(H_i, H_j)$ by $H_j$ is
\begin{align*}
S_1 &:= S(H_i, H_j) + \left(c_{i-1,0}x^{\beta^0_{i-1,0}} y_1^{\beta^0_{i-1,1}} \cdots y_{i-2}^{\beta^0_{i-1,i-2}}\right)H_j
	= \left(c_{j-1,0}x^{\beta^0_{j-1,0}} y_1^{\beta^0_{j-1,1}} \cdots y_{j-2}^{\beta^0_{j-1,j-2}}\right)H_i,
\end{align*}
so that the leading term of $S_1$ is 
\begin{align*}
\lt_{\prec_k}(S_1) &= \left(c_{j-1,0}x^{\beta^0_{j-1,0}} y_1^{\beta^0_{j-1,1}} \cdots y_{j-2}^{\beta^0_{j-1,j-2}}\right)y_{i-1}^{p_{i-1}}.
\end{align*}
It follows as in the case of $S(H_i,H_j)$ that $H_i$ is the only element of $\scrB_k$ whose leading term divides $\lt_{\prec_k}(S_1)$. Since $H_i$ divides $S_1$, the remainder of the division of $S_1$ by $H_i$ is zero, and it follows that the remainder of the division of $S(H_i, H_k)$ by $\scrB_k$ is zero. Consequently {\bf Step 2 concludes without changing $\scrB_k$}.\\

It follows from the preceding paragraphs that running one step of Buchberger's algorithm keeps $\scrB_{k}$ unchanged, and consequently $\scrB_{k}$ is a Gr\"obner basis of $J_{k}$ with respect to $\prec_{k}$ \cite[Theorem 2.7.2]{littlesheacox}. This completes the proof of Assertion \ref{grobner-assertion} of Lemma \ref{poly-positive-prop}.
\end{proof}

\bibliographystyle{alpha}
\bibliography{bibi}

\end{document}

%% file: commands-2012.tex
\makeatletter

\newcommand{\Rmnum}[1]{\expandafter\@slowromancap\romannumeral #1@}
\makeatother

\DeclareMathOperator\gr{gr}
 
\DeclareMathOperator\ord{ord}

\DeclareMathOperator\proj{Proj}

\newcommand{\scrB}{\ensuremath{\mathcal{B}}}

\newcommand{\scrF}{\ensuremath{\mathcal{F}}}

\newcommand{\scrV}{\ensuremath{\mathcal{V}}}

\newcommand{\cc}{\ensuremath{\mathbb{C}}}
\newcommand{\kk}{\ensuremath{\mathbb{K}}}

\newcommand{\pp}{\ensuremath{\mathbb{P}}}
\newcommand{\qq}{\ensuremath{\mathbb{Q}}}
\newcommand{\rr}{\ensuremath{\mathbb{R}}}
\newcommand{\zz}{\ensuremath{\mathbb{Z}}}

\newcommand{\ppp}{\ensuremath{\mathfrak{p}}}

\newcommand{\sheaf}{\ensuremath{\mathcal{O}}}

\newcommand{\im}{\ensuremath{\Rightarrow}}

\newcommand{\dsum}{\ensuremath{\bigoplus}}

\newtheorem{thm}{Theorem}[section]
\newtheorem*{thm*}{Theorem}
\newtheorem{lemma}[thm]{Lemma}
\newtheorem*{lemma*}{Lemma}
\newtheorem{lemma-in-thm}{Lemma}[thm]

\newtheorem{prop}[thm]{Proposition}
\newtheorem*{prop*}{Proposition}
\newtheorem{cor}[thm]{Corollary}
\newtheorem{claim}[thm]{Claim}
\newtheorem*{claim*}{Claim}
\newtheorem{proclaim}{Claim}[thm]
\newtheorem*{conjecture*}{Conjecture}

\theoremstyle{definition} 
\newtheorem{algorithm}[thm]{Algorithm}

\newtheorem*{constrinition*}{Construction-Definition}
\newtheorem{convention}[thm]{Convention}
\newtheorem*{convention*}{Convention}
\newtheorem{defn}[thm]{Definition}
\newtheorem*{defn*}{Definition}

\newtheorem*{definotation*}{Definition-Notation}
\newtheorem{example}[thm]{Example}
\newtheorem*{example*}{Example}

\newtheorem*{fact*}{Fact}

\newtheorem*{facts*}{Facts}

\newtheorem{notation}[thm]{Notation}

\newtheorem*{bold-note*}{Note}
\newtheorem{bold-question}[thm]{Question}
\newtheorem{rem}[thm]{Remark}

\newtheorem*{reminition*}{Remark-Definition}

\newtheorem{remexample*}{Remark-Example}
\newtheorem{remtation}[thm]{Remark-Notation}
\newtheorem*{remtation*}{Remark-Notation}

\newtheorem*{remuestion*}{Remark-Question}

\theoremstyle{remark}
\newtheorem*{rem*}{Remark}
\newtheorem*{note*}{Note}
\newtheorem*{notation*}{Notation}
\newtheorem*{question*}{Question}
\newtheorem{question}[thm]{Question}
\newtheorem*{questions*}{Questions}

\newcounter{UnorderedProofTempCtr}
\newcommand{\tempcommand}{}

\newcommand{\WP}{\ensuremath{{\bf W}\pp}}